\documentclass[final,onefignum,onetabnum]{siamart250211}
\usepackage{graphicx}
\usepackage{amsmath,multirow}
\usepackage{amssymb}
\usepackage{amscd}
\usepackage{amsfonts}
\usepackage{enumerate}
\usepackage{mathrsfs}
\usepackage{xypic}
\usepackage{stmaryrd}
\usepackage{graphicx}
\usepackage{fancybox}
\usepackage{color}
\usepackage{amsmath}
\usepackage{exscale}
\usepackage{enumitem,array}%
\usepackage{amssymb}
\usepackage{amsmath}
\usepackage{amssymb}
\usepackage{amscd}
\usepackage{amsfonts}
\usepackage{enumerate}
\usepackage{mathrsfs}
\usepackage{xy}
\usepackage{stmaryrd}
\usepackage{graphicx}
\usepackage{fancybox}
\usepackage{color}
\usepackage{multirow}
\usepackage{exscale}
\usepackage{enumitem,array}%
\usepackage{subfigure}
\usepackage{extarrows}
\usepackage{amsfonts}
\usepackage{xypic}
\usepackage{tikz}
\usepackage{algorithm}
\usepackage{algorithmicx}
\usepackage{algpseudocode}
\usepackage{manyfoot}
\usepackage{booktabs}
\usepackage{verbatim}
\usepackage{graphicx}
\usepackage{amsopn}

\usetikzlibrary{arrows,chains,matrix,positioning,scopes}
\makeatletter
\tikzset{join/.code=\tikzset{after node path={%
\ifx\tikzchainprevious\pgfutil@empty\else(\tikzchainprevious)%
edge[every join]#1(\tikzchaincurrent)\fi}}}
\makeatother
\tikzset{>=stealth',every on chain/.append style={join},
         every join/.style={->}}
\tikzstyle{labeled}=[execute at begin node=$\scriptstyle,
   execute at end node=$]

\newsiamremark{remark}{Remark}
\newsiamremark{hypothesis}{Hypothesis}
\crefname{hypothesis}{Hypothesis}{Hypotheses}
\newsiamthm{claim}{Claim}

\usepackage{amssymb,amsmath,geometry,color}
\usepackage{epstopdf}
\usepackage{epsfig,multirow}
\usepackage{color,cases}
\usepackage{mathtools}

\newtheorem{myalgorithm}{Algorithm}[section]
\DeclarePairedDelimiter{\norm}{\lVert}{\rVert}

\newcommand{\fe}{\mathrm{e}}
\renewcommand{\(}{\left(}
\renewcommand{\)}{\right)}
\newcommand{\eps}{\varepsilon}

\newcommand{\abs}[1]{\left\vert#1\right\vert}

\newcommand{\bw}{{\bf w}}

\newcommand{\bx}{\mathbf{x}}
\newcommand{\bv}{\mathbf{v}}
\newcommand{\bz}{\mathbf{z}}

\newcommand{\mh}{\mathfrak{h}}
\newcommand{\tbz}{\tilde{\mathbf{z}}}
\newcommand{\dtbz}{\dot{\tilde{\mathbf{z}}}}

\newcommand{\bE}{\mathbf{E}}
\newcommand{\bB}{\mathbf{B}}

\newcommand{\hbB}{\widehat{\mathbf{B}}}
\newcommand{\sgn}{\operatorname{sgn}}

\headers{Explicit relaxation PIC for Vlasov-Poisson equations}{L. Wang, B. Wang}

\title{Explicit relaxation Particle-in-Cell methods for Vlasov-Poisson equations with a strong magnetic field
\thanks{Submitted to the editors DATE.
\funding{B. Wang is supported by NSFC 12371403. }}}

\author{Lina Wang\thanks{School of Mathematics and Statistics, Xi'an Jiaotong University, 710049 Xi'an, China
  (wanglina@stu.xjtu.edu.cn).}
  \and Bin Wang\thanks{Corresponding author.
  School of Mathematics and Statistics, Xi'an Jiaotong University, 710049 Xi'an, China
  (wangbinmaths@xjtu.edu.cn).}
 }

\begin{document}

\maketitle

\begin{abstract}
In this work, we present a novel family of explicit relaxation Particle-in-Cell (ER-PIC) methods for the Vlasov-Poisson equation with a strong magnetic field. These schemes achieve exact energy conservation by combining a splitting framework with the dynamic updating of a relaxation parameter at each time step. Using an averaging technique, we rigorously establish second-order error bounds for the Strang-type ER-PIC method and uniform first-order accuracy in position for the Lie-Trotter  ER-PIC scheme.  Numerical experiments across the fluid, finite Larmor radius, and diffusion regimes confirm the accuracy and  energy conservation of our methods.
\end{abstract}



\section{Introduction}

The Vlasov equation describes the evolutions of the probability distribution function in high-dimensional phase space under a self-consistent or external electromagnetic field. It is widely used to simulating plasma dynamics in magnetic confinement devices such as tokamaks \cite{Bel,Mi}, where a strong external magnetic field needs to be applied in order to keep the particles on the desired tracks.
In this paper, we consider the Vlasov-Poisson equation with a strong non-homogeneous magnetic field (\cite{CCLMZ2,CCLMZ3}), which is also called fluid scaling:
\begin{subequations}\label{vp}
\begin{align}
&\partial_{t}f(t,\bx,\bv)+\bv\cdot\nabla_{\bx}f(t,\bx,\bv)+\left(\bE(t,\bx)+\frac{1}{\eps}\bv\times\bB(\bx)\right)\cdot\nabla_{\bv}f(t,\bx,\bv)=0, \label{vp-a} \\
&\nabla_{\bx}\cdot\bE(t,\bx)=\int_{\mathbb{R}^d}f(t,\bx,\bv)d\bv-n_{i}, \label{vp-b} \\
&f(0,\bx,\bv)=f_{0}(\bx,\bv), \label{vp-c}
\end{align}
\end{subequations}
where 
 $\bE(t,\bx)\in\mathbb{R}^{d}$ denotes the self-consistent electric-field function and the external magnetic field is denoted by (\cite{CCLMZ3,HLS})  
\begin{equation}\label{equ-9-28-2}
\bB:\bx\in \Omega\subset\mathbb{R}^{d}\mapsto  \bB(\bx)=\bB_{0}+\eps\bB_{1}(\bx) \in\mathbb{R}^{d}, \ d=2,3,
\end{equation}
with a constant vector $\bB_{0}$ and a uniformly bounded vector field $\bB_{1}(\bx)$. For a given $T>0$, 
$f:(t,\bx,\bv)\in [0,T]\times\mathbb{R}^{d}\times\mathbb{R}^{d}\mapsto f(t,\bx,\bv)\in \mathbb{R}$
is the unknown distribution function, $f_{0}(\bx,\bv)$ is a given initial distribution, $0<\eps\leq 1$ a dimensionless parameter inversely proportional to the strength of the magnetic field, and $n_{i}\geq 0$ the ion density of the background. It is noted that the so-called maximal ordering scaling case (\cite{Brizard,HLW}), where  
$
\bB/\eps =\bB(\eps \bx)/\eps,
$
can be decomposed as  
$
\bB =\bB(\eps \bx_0)+ (\bB(\eps \bx) - \bB(\eps \bx_0)),
$
in which the second term is bounded as $\mathcal{O}(\eps)$. 
Consequently, the framework developed in this paper naturally extends to this maximal ordering scaling.
As is well known, the Vlasov-Poisson system \eqref{vp} satisfies the following energy conservation law
\begin{equation}\label{vp-ener}
H(t)=\frac{1}{2}\int_{\mathbb{R}^d}\int_{\Omega}\abs{\bv}^{2}f(t,\bx,\bv)d\bx d \bv+\frac{1}{2}\int_{\Omega}\abs{\bE(t,\bx)}^{2}d\bx=H(0).
\end{equation}

For simulating plasma dynamics, the Particle-in-Cell (PIC) method is a cornerstone  \cite{CLS,LHSNQL}, which discretizes the distribution function into weighted macroparticles via the Klimontovich representation and tracks their trajectories \cite{L}. Building upon this framework, numerous classical methods have been developed for the Vlasov-Poisson system under a non-strong magnetic field, such as Hamiltonian splitting \cite{CCFM,GHS}, semi-Lagrangian schemes \cite{BS,CDM,Lu1}, and dynamical low-rank algorithms \cite{Lu2,Lu3}. However, for the strong magnetic field case, a significant challenge for the Vlasov-Poisson equation arises, where the small parameter $0<\eps\ll 1$ induces rapid gyration of charged particles (with a cyclotron period of order $\eps$) and confinement along field lines. This causes the solution of \eqref{vp} to become highly oscillatory in time, posing substantial difficulties for numerical simulation. To address this, various methods have been proposed, including time splitting \cite{CCZ}, Crank-Nicolson \cite{RC}, and asymptotic-preserving (AP) schemes \cite{CLM,FR1,FR2,FRZ,FHLS}. However, these conventional approaches generally fail to provide uniform accuracy with respect to $\eps$.
In response, a new class of uniformly accurate (UA) methods has emerged, whose error bounds and permissible time-steps are both independent of $\eps$. Prominent examples include forward semi-Lagrangian methods \cite{CLMZ1}, multi-revolution composition methods \cite{CCZ,CCLMZ3}, and two-scale methods \cite{CCLMZ2,CCLMZ3,CLMZ2}. Despite their success in achieving uniform accuracy, these UA methods often suffer from a critical limitation: the lack of  long-term near/exact energy conservation. This deficiency is frequently linked to numerical instabilities arising from the spatial discretization on finite grids, which can lead to a spurious, monotonic increase in the total energy \cite{BL}.
 
To mitigate the long-term numerical drift of energy, geometric Particle-in-Cell (PIC) methods have been developed \cite{CC2,CCB,CC1,G,KKMS,L2,GeoHybridPIC,ML,RH}. A prominent subclass of these methods, implicit energy-preserving (EP) PIC schemes \cite{CC2,CCB,CC1,L2,ML}, achieves exact energy conservation at the cost of solving a large nonlinear system. This computational burden has spurred interest in explicit EP schemes, which have been successfully developed for scenarios without strong magnetic fields \cite{G,RH}.
A powerful strategy for constructing such EP schemes is the relaxation approach, which modifies a base method to enforce energy conservation without sacrificing accuracy \cite{K,Ld1,Ld2,RSDPK}. This idea has been applied to the Vlasov-Maxwell equations in non-strong magnetic field regimes using a Hermite-discontinuous Galerkin framework \cite{PMKDR}. The principle also extends to the core of the PIC method: the discretization of the distribution function into a system of charged-particle dynamics (CPD). For CPD in strong magnetic fields, implicit EP  schemes are established in \cite{WZ,Yin}, and explicit EP methods  are developed in \cite{LW} by using the relaxation technique, where a uniform error bound is rigorously derived for the first-order scheme under the maximal ordering scaling.  
Although these EP methods are state-of-the-art in the energy conservation, there are some critical gaps. For the Vlasov-Poisson system under a strong magnetic field, EP schemes, especially explicit ones, have yet to be developed. Furthermore, while error analysis exists for first-order  schemes \cite{LW}, a rigorous analysis for second-order explicit EP methods in the highly oscillatory regime is notably absent. This work is dedicated to closing these gaps.

 In this paper, we introduce a new class of explicit relaxation Particle-in-Cell (ER-PIC) methods for the Vlasov-Poisson system \eqref{vp} in the strong magnetic field. 
The approach first applies the PIC method to the Vlasov-Poisson equation \eqref{vp} and then incorporates the relaxation technique into the time discretization of the CPD system, yielding two novel relaxation PIC schemes. The new ER-PIC schemes offer several significant advantages:
\begin{itemize}
  \item 
  The ER-PIC schemes are completely explicit, thus circumventing the high computational cost associated with solving nonlinear systems in traditional implicit or semi-implicit methods.

  \item They preserve the discrete total energy exactly, which guarantees excellent long-term numerical behavior and stability.
  \item The rigorous convergence analysis, based on an averaging technique, establishes that in approximating the particle position, the first ER-PIC scheme has a uniform error bound  and the second one achieves  second-order accuracy.
      
  \item  
The proposed methods demonstrate robust performance for the Vlasov-Poisson system under various  scaling regimes, such as the fluid, Larmor and diffusion regimes. Numerical experiments confirm the efficacy of these methods in all considered cases.
\end{itemize}

The remainder of this paper is organized as follows. Section \ref{sec-2} formulates the explicit relaxation PIC (ER-PIC) schemes, which combine PIC, splitting, and relaxation techniques to achieve exact energy conservation. In Section \ref{sec-3}, we provide a rigorous error analysis for the first- and second-order ER-PIC methods using an averaging technique. Section \ref{sec-4} presents numerical experiments that validate the scheme’s accuracy and energy conservation for the Vlasov-Poisson system under various scalings.  Some conclusions are included in the last section.

\section{Explicit relaxation Particle-in-Cell (ER-PIC) method}\label{sec-2}

In this section, we develop a new class of explicit relaxation Particle-in-Cell methods which begin with a brief review of the PIC method (\cite{Sonn}).

\subsection{PIC framework}

In the PIC method, the unknown distribution $f(t,\bx,\bv)$ of \eqref{vp} is approximated by a sum of Dirac masses centered at $(\bx_{k}(t),\bv_{k}(t))$ with weights $\omega_{k}>0$ for $k=1,\ldots, N_{p}\ (N_{p}\in\mathbb{N}^{+})$ as
\begin{equation}\label{equ-9-28-4}
f_{p}(t,\bx,\bv)=\sum\limits_{k=1}^{N_{p}}\omega_{k}\delta(\bx-\bx_{k}(t))\delta(\bv-\bv_{k}(t)), \quad t\geq 0, \ (\bx,\bv)\in\Omega\times\mathbb{R}^{d}.
\end{equation}
Inserting \eqref{equ-9-28-4} into \eqref{vp}, in the sense of distribution, each particle obeys the characteristic equation
\begin{equation}\label{equ-9-28-5}
\begin{aligned}
\dot{\bx}_{k}(t)=\bv_{k}(t), \ \dot{\bv}_{k}(t)=\bE(t,\bx_{k}(t))+\frac{1}{\eps}\bv_{k}(t)\times\bB(\bx_{k}(t)), \quad \bx_{k}(0)=\bx_{k,0}, \ \bv_{k}(0)=\bv_{k,0}, \ t>0. 
\end{aligned}
\end{equation}
For $k=1,\ldots,N_{p}$, the weight $\omega_{k}$ and initial particle states $\bx_{k,0}, \bv_{k,0}$ are assigned according to the initial distribution $f_{0}(\bx,\bv)$ given by \eqref{vp-c}.
  A general approach employs Monte Carlo rejection sampling \cite{Sonn}. The weights $\omega_k$ must satisfy the normalization condition obtained by integrating \eqref{equ-9-28-4} over the entire domain at $t=0$:
  $\sum_{k=1}^{N_{p}}\omega_{k}=\int_{\Omega\times\mathbb{R}^{d}}f_{0}(\bx,\bv)d\bx d\bv.$
  Consequently, a simple choice is to assign uniform weights to all particles:
  $$\omega_{k}=\frac{1}{N_{p}}\int_{\Omega\times\mathbb{R}^{d}}f_{0}(\bx,\bv)d\bx d\bv, \quad k=1,\ldots,N_{p}.$$
  The characteristic equations \eqref{equ-9-28-5} are coupled with the Poisson equation \eqref{vp-b} through the electric field $\bE(t,\bx)$. At any time $t>0$, given the particle positions $\{\bx_{k}(t)\}_{k=1}^{N_{p}}$, the Poisson equation
  $$\nabla_{\bx}\cdot \bE_{p}(t,\bx)=\sum_{k=1}^{N_{p}}\omega_{k}\delta(\bx-\bx_{k}(t))-n_{i}, \quad \bx\in\Omega,$$
  is solved numerically on a spatial mesh grid with a step size $\triangle x$. Notably, the time dependence of the electric field arises solely from the particle positions, which can be expressed as $\bE(t,\bx)=\bE_{[\bx^{p}(t)]}(\bx)$, where $\bx^{p}(t)=(\bx_{1}(t),\ldots,\bx_{N_{p}}(t))$.
 Utilizing the PIC framework, the continuous total energy, derived from the distribution function \eqref{equ-9-28-4}, is discretized as:
  \begin{equation}\label{equ-9-30-1}
  H(t)=\frac{1}{2}\sum_{k=1}^{N_{p}}\omega_{k}\abs{\bv_{k}}^{2}+\frac{1}{2}\int_{\Omega}\abs{\bE_{[\bx^{p}(t)]}(\bx)}^{2}d\bx.
  \end{equation}

\subsection{Construction of ER-PIC method}
In this subsection, we present explicit relaxation schemes, constructed by introducing a relaxation parameter into the symmetric St\"{o}rmer-Verlet method. Let $h=\triangle t>0$ denote the time step and $t_{n}=nh$ for $n\in\mathbb{N}$. The system \eqref{equ-9-28-5} is first decomposed into two subflows (the particle index $k$ is omitted for brevity):
\begin{equation}\label{equ-9-30-3}
\dot{\bx}(t)=0, \dot{\bv}(t)=\bv(t)\times \dfrac{\bB(\bx(t))}{\eps}, \ \bx(0)=\bx_{0}, \ \bv(0)=\bv_{0},
\end{equation}
and
\begin{equation}\label{equ-9-30-4}
\dot{\bx}(t)=\bv(t), \dot{\bv}(t)=\bE_{[\bx^{p}(t)]}(\bx(t)), \ \bx(0)=\bx_{0}, \ \bv(0)=\bv_{0}.
\end{equation}
For the first flow \eqref{equ-9-30-3}, we can get the exact $h$-flow 
$
\psi_{h}^{[1]}:\begin{pmatrix} \bx(t)\\ \dot{\bx}(t) \end{pmatrix}=\begin{pmatrix} \bx_{0} \\ \fe^{\frac{h}{\eps}\widehat{\bB}(\bx_{0})}\dot{\bx}_{0} \end{pmatrix},
$
where the skew symmetric matrix $\widehat{\bB}$ is given as
$
\widehat{\bB}(\bx)=\begin{pmatrix}
                0 & b_{3}(\bx) & -b_{2}(\bx) \\
                -b_{3}(\bx) & 0 & b_{1}(\bx) \\
                b_{2}(\bx) & -b_{1}(\bx) & 0 
              \end{pmatrix}
$
with the magnetic field $\bB=(b_{1},b_{2},b_{3})^{\intercal}\in\mathbb{R}^{3}$. By applying the Rodriguez formula \cite{HLW,KKMS}, the matrix exponential $\fe^{h\widehat{\bB}}$ can be computed efficiently. For the second subflow \eqref{equ-9-30-4}, which is equivalent to a nonlinear second-order Hamiltonian system $\ddot{\bx}(t)=\bE_{[\bx^{p}(t)]}(\bx(t))$ with $\bx(0)=\bx_{0}, \dot{\bx}(0)=\dot{\bx}_{0}=\bv_{0}$, an approximate $h$-flow $\psi_{h}^{[2]}$ is developed. To achieve exact energy conservation, we choose a classical explicit  symmetric  second-order St\"{o}rmer-Verlet (SV) method as the underlying approach and incorporate a relaxation parameter in its update stage, which by denoting $\bx^{n}\approx \bx(t_{n})$ and $\dot{\bx}^{n}\approx \dot{\bx}(t_{n})$ is given as
\begin{equation}\label{equ-10-1-1}
\begin{aligned}
\psi_{h}^{[2]}: \ & X^{1}=\bx^{n}+\dfrac{h}{2}\dot{\bx}^{n}, \ \bx^{n+1}=\bx^{n}+h\dot{\bx}^{n}+\dfrac{h^{2}}{2}\bE_{[X^{1}]}(X^{1}), \ \dot{\bx}^{n+1}=\dot{\bx}^{n}+h(1+\gamma_{n})\bE_{[X^{1}]}(X^{1}),
\end{aligned}
\end{equation}
where the relaxation parameter $\gamma_{n}$ is determined by the condition of energy conservation
\begin{equation}\label{equ-10-1-2}
H(\bx^{n+1},\dot{\bx}^{n+1})-H(\bx^{n},\dot{\bx}^{n})=0.
\end{equation}
The one-step numerical approximation  yielded by the underlying SV method without relaxation is denoted by $\tilde{\bx}^{n+1}$ and $\dot{\tilde{\bx}}^{n+1}$.
Then the system \eqref{equ-10-1-1} can be expressed equivalently as
\begin{equation}\label{equ-10-1-3}
\begin{aligned}
& X^{1}=\bx^{n}+\dfrac{h}{2}\dot{\bx}^{n}, \ \tilde{\bx}^{n+1}=\bx^{n}+h\dot{\bx}^{n}+\dfrac{h^{2}}{2}\bE_{[X^{1}]}(X^{1}), \ \dot{\tilde{\bx}}^{n+1}=\dot{\bx}^{n}+h\bE_{[X^{1}]}(X^{1}), \\
& \bx^{n+1}=\tilde{\bx}^{n+1}, \ \dot{\bx}^{n+1}=\dot{\tilde{\bx}}^{n+1}+h\gamma_{n}\bE_{[X^{1}]}(X^{1}).
\end{aligned}
\end{equation}
Inserting the equations \eqref{equ-9-30-1} and \eqref{equ-10-1-3} into \eqref{equ-10-1-2} yields
\begin{equation*}
\begin{aligned}
&H(\bx^{n+1},\dot{\bx}^{n+1})-H(\bx^{n},\dot{\bx}^{n})
=\frac{1}{2}\sum\limits_{k=1}^{N_{p}}\omega_{k}\Big(\abs{\dot{\bx}_{k}^{n+1}}^{2}-\abs{\dot{\bx}_{k}^{n}}^{2}\Big)+\frac{1}{2}\int_{\Omega}\abs{\bE_{[\bx^{n+1}]}(\bx^{n+1})}^{2}d\bx
-\frac{1}{2}\times \\
&\int_{\Omega}\abs{\bE_{[\bx^{n}]}(\bx^{n})}^{2}d\bx =H(\tilde{\bx}^{n+1},\dot{\tilde{\bx}}^{n+1})-H(\tilde{\bx}^{n},\dot{\tilde{\bx}}^{n})+h\gamma_{n}\sum\limits_{k=1}^{N_{p}}\omega_{k}(\bE_{[X^{1}]}(X^{1}_{k}))^{\intercal}
\Big(\dot{\tilde{\bx}}_{k}^{n+1}+\frac{1}{2}h\gamma_{n}\bE_{[X^{1}]}(X^{1}_{k})\Big).
\end{aligned}
\end{equation*}
Denoting $\widetilde{H}:=H(\tilde{\bx}^{n+1},\dot{\tilde{\bx}}^{n+1})-H(\tilde{\bx}^{n},\dot{\tilde{\bx}}^{n})$ and 
\begin{equation*}
\mathcal{A}=\sum\limits_{k=1}^{N_{p}}\omega_{k}(\bE_{[X^{1}]}(X^{1}_{k}))^{\intercal}\bE_{[X^{1}]}(X^{1}_{k}), \quad \mathcal{C}=\sum\limits_{k=1}^{N_{p}}\omega_{k}(\bE_{[X^{1}]}(X^{1}_{k}))^{\intercal}\dot{\tilde{\bx}}_{k}^{n+1},
\end{equation*}
then the relaxation parameter $\gamma_{n}$ can be determined explicitly as $\gamma_{n}=(-\mathcal{C}\pm\sqrt{\mathcal{C}^{2}-2\mathcal{A}\widetilde{H}})/h\mathcal{A}$.
Provided that the relaxation parameter $\gamma_{n}$ is sufficiently small, the SV method with relaxation can maintain the algebraic order of the underlying SV method. Since different roots may arise in $\gamma_{n}$, we choose the relaxation parameter $\gamma_{n}$ as
\begin{equation}\label{equ-10-1-6}
\gamma_{n}=\begin{cases}
  0, & \mbox{if } \mathcal{C}^{2}-2\mathcal{A}\widetilde{H}<0, \\
  (-\mathcal{C}+\sgn(\mathcal{C})\sqrt{\mathcal{C}^{2}-2\mathcal{A}\widetilde{H}})/h\mathcal{A}, & \mbox{if } \mathcal{C}^{2}-2\mathcal{A}\widetilde{H}\geq 0,
\end{cases}
\end{equation}
where $\sgn$ is a sign function.
The SV method \eqref{equ-10-1-1} with a relaxation parameter $\gamma_{n}$ is referred to as relaxation St\"{o}rmer-Verlet (RSV) method. The following theorem establishes the algebraic order of the relaxation parameter $\gamma_{n}$ and the resulting RSV method. The proof is analogous to that in \cite{LW} and is therefore omitted.

\begin{theorem}\label{the-1}
For a sufficiently small time step size $h>0$, the relaxation parameter $\gamma_{n}$ in \eqref{equ-10-1-6} is second-order accurate. Consequently, the RSV method is also second-order.
\end{theorem}


Denote the numerical solution $\bx^{n}\approx\bx(t_{n})$, $\dot{\bx}^{n}\approx\dot{\bx}(t_{n})$ and choose $\bx^{0}=\bx_{0}$, $\dot{\bx}^{0}=\dot{\bx}_{0}$. We end up with the following full scheme which conserves energy through the combination of $\psi_{h}^{[1]}$ and $\psi_{h}^{[2]}$. 
\begin{myalgorithm}\label{algo-1} (\textbf{ER-PIC method})
Given an initial distribution function $f_{0}(\bx,\bv)$, the explicit relaxation energy-conserving Particle-in-Cell Method reads as follows:
\begin{enumerate}[label=(\roman*)]
\item Compute the density $\rho(t,\bx)=\sum\limits_{k=1}^{N_{p}}\omega_{k}\delta(\bx-\bx_{k}(t))$ on the grid points.
\item Update the electric field $E(t,\mathbf{x})$ on the grid by solving the Poisson equation spectrally and compute magnetic field $\bB(\bx)$.
\item Interpolate the electric field at the particles position.
\item Update the position and velocity of particles using the Lie-Trotter splitting scheme $\psi^{[2]}_{h}\circ \psi^{[1]}_{h}$ (referred as RS1-PIC) or Strang splitting scheme $\psi^{[1]}_{h/2}\circ \psi^{[2]}_{h} \circ \psi^{[1]}_{h/2}$ (referred as RS2-PIC).
\item Repeat steps (i)-(iv) until the final time $T$.
\end{enumerate}
\end{myalgorithm}

The energy conservation of Algorithm \ref{algo-1} is stated as follows.
\begin{proposition}
Algorithm \ref{algo-1} exactly preserves the energy \eqref{equ-9-30-1} at the discrete level, i.e., for $n\in\mathbb{N}$, $H(\bx^{n},\dot{\bx}^{n})=H(\bx^{0},\dot{\bx}^{0})$.
\end{proposition}
\begin{proof}
Denote the one-step numerical solutions in RS2-PIC
\begin{equation}\label{equ-10-3-6}
(\bx^{0},\dot{\bx}^{0})\xmapsto{\psi_{h/2}^{[1]}}(\bx^{1/2}_{1},\dot{\bx}^{1/2}_{1})\xmapsto{\psi_{h}^{[2]}}(\bx^{1/2}_{2},\dot{\bx}^{1/2}_{2})\xmapsto{\psi_{h/2}^{[1]}}
(\bx^{1},\dot{\bx}^{1}).
\end{equation}
Then we obtain $
               \begin{pmatrix} \bx^{1/2}_{1} \\ \dot{\bx}^{1/2}_{1} \end{pmatrix}=
               \begin{pmatrix}
               \bx^{0} \\
               \fe^{\frac{h}{2\eps}\widehat{\bB}(\bx^{0})}\dot{\bx}^{0}
               \end{pmatrix}
              $
and           $
               \begin{pmatrix} \bx^{1} \\ \dot{\bx}^{1} \end{pmatrix}=
               \begin{pmatrix}
               \bx^{1/2}_{2} \\
               \fe^{\frac{h}{2\eps}\widehat{\bB}(\bx^{1/2}_{2})}\dot{\bx}^{1/2}_{2}
               \end{pmatrix}
              $.
Since the matric $\widehat{\bB}(\bx)$ is skew symmetry, it is derived that
\begin{equation*}
\begin{aligned}
H(\bx^{1/2}_{1},\dot{\bx}^{1/2}_{1})=&\frac{1}{2}\sum\limits_{k=1}^{N_{p}}\omega_{k}\abs{\dot{\bx}^{1/2}_{k,1}}^{2} +\frac{1}{2}\int_{\Omega_{x}}\abs{\bE_{[\bx^{1/2}_{1}]}(\bx^{1/2}_{1})}^{2}d\bx =\frac{1}{2}\sum\limits_{k=1}^{N_{p}}\omega_{k}\left(\fe^{\frac{h}{2\eps}\widehat{\bB}(\bx^{0}_{k})}\dot{\bx}^{0}_{k}\right)^{\intercal}
\left(\fe^{\frac{h}{2\eps}\widehat{\bB}(\bx^{0}_{k})}\dot{\bx}^{0}_{k}\right) \\
+&\frac{1}{2}\int_{\Omega}\abs{\bE_{[\bx^{0}]}(\bx^{0})}^{2}d\bx =\frac{1}{2}\sum\limits_{k=1}^{N_{p}}\omega_{k}\abs{\dot{\bx}^{0}_{k}}^{2}+\frac{1}{2}\int_{\Omega}\abs{\bE_{[\bx^{0}]}(\bx^{0})}^{2}d\bx
=H(\bx^{0},\dot{\bx}^{0}).
\end{aligned}
\end{equation*}
Analogously, we have $H(\bx^{1/2}_{2},\dot{\bx}^{1/2}_{2})=H(\bx^{1},\dot{\bx}^{1})$. Together with the result $$H(\bx^{1/2}_{1},\dot{\bx}^{1/2}_{1})=H(\bx^{1/2}_{2},\dot{\bx}^{1/2}_{2})$$ from \eqref{equ-10-1-2},   this directly implies the energy conservation property of Algorithm \ref{algo-1}.
\end{proof}

\section{Global convergence}\label{sec-3}

This section presents the convergence results for the proposed ER-PIC schemes. For notational simplicity, the symbol $A\lesssim B$ is used throughout to denote $A\leq CB$, where $C>0$ is a generic constant independent of $\eps$ or the time step $h$ or $n$. Additionally, $\tau_{s}^{n}$ denotes intermediate time values that may vary from line to line in the proofs. The notation $\abs{\cdot}$ denotes the Euclidean norm for a vector or matrix.

\subsection{Main result}
The main convergence results for the two ER-PIC methods are summarized in the following theorem.

\begin{theorem}\label{theo-2}
Assume that $\bB(\cdot), \bE(\cdot)\in C^{1}(\mathbb{R}^d)$. For a fixed time $T>0$, let $\bx^{n}, \dot{\bx}^{n}$ be the numerical solution of RS1-PIC or RS2-PIC for solving \eqref{equ-9-28-5} up to $T$. Then there exists a constant $N_0>0$ independent of $\eps$, such that for any integer $N\geq N_0$ and the time step $h=\frac{T_{0}}{N}\eps$, we have for some $m_0>0$ arbitrarily large
\begin{equation}\label{opt-err1}
\begin{aligned}
\textmd{RS1-PIC}:\ &\abs{\bx^{n}-\bx(t_{n})}+\eps\abs{\dot{\bx}^{n}-\dot{\bx}(t_{n})} \lesssim h+N^{-m_0},  \\
\textmd{RS2-PIC}: \ &\abs{\bx^{n}-\bx(t_{n})}+\eps\abs{\dot{\bx}^{n}-\dot{\bx}(t_{n})} \lesssim  h^{2}/\eps+N^{-m_0}, \quad 
0\leq n \leq T/h.
\end{aligned}
\end{equation}
\end{theorem}

In the remainder of this section, we prove Theorem \ref{theo-2}. We focus on the more complex RS2-PIC method, presenting its proof in full detail. The proof for  RS1-PIC is then a direct consequence.

\subsection{Transformation of RS2-PIC scheme}

To obtain optimal error bounds (with explicit $\eps$-dependence) over an $\eps$-independent time interval $[0,T]$, we employ a time rescaling of \eqref{equ-9-28-5}:
\begin{equation}\label{equ-10-4-1}
\tau:=t/\eps, \quad \bz(\tau):=\bx(t), \quad \bw(\tau):=\bv(t), \quad 0\leq \tau\leq T/\eps, 
\end{equation}
such that the equation \eqref{equ-9-28-5} is reformulated as a second-order long-time problem
\begin{equation}\label{long-sys}
\ddot{\bz}(\tau)=\dot{\bz}(\tau)\times \bB(\bz(\tau))+\eps^{2}\bE(\bz(\tau)), \ \bz(0)=\bx_{0}, \ \dot{\bz}(0)=\eps\dot{\bx}_{0}, \ 0<\tau\leq T/\eps .
\end{equation}
The long-time formulation elucidates the scale separation between $h$ and $\eps$ and the resulting averaging effect, which will be leveraged in the subsequent analysis. Given that $\bB(\cdot), \bE(\cdot)\in C^{1}(\mathbb{R}^{d})$, the system \eqref{long-sys} clearly satisfies
\begin{equation}\label{equ-10-4-2}
\norm{\bz}_{L^{\infty}(0,T/\eps)}\lesssim 1, \quad \norm{\dot{\bz}}_{L^{\infty}(0,T/\eps)}\lesssim \eps.
\end{equation}
Moreover, since $\widehat{\bB}$ is skew-symmetry, the propagator $\fe^{\tau \widehat{\bB}_{0}}$ generates a periodic flow with a single period denoted by $T_{0}>0$. Introducing the step size $\mathfrak{h}=\triangle\tau >0$ and the temporal grids $\tau_{n}=n\mathfrak{h}$ for the scaled time $\tau$, and letting $\bz^{n}\approx \bz(\tau_{n})$, $\dot{\bz}^{n}\approx \dot{\bz}(\tau_{n})$ denote the numerical solutions, the RS2-PIC method applied to \eqref{long-sys} under the long-time scaling reads: $\bz^{0}=\bx^{0}, \dot{\bz}^{0}=\eps \dot{\bx}^{0}$,
\begin{equation}\label{st-rsv-resc}
\begin{aligned}
&Z^{1}=\bz^{n}+\dfrac{\mh}{2}\fe^{\frac{\mh}{2}\widehat{\bB}(\bz^{n})}\dot{\bz}^{n}, \
\bz^{n+1}=\bz^{n}+\mh \fe^{\frac{\mh}{2}\widehat{\bB}(\bz^{n})}\dot{\bz}^{n}+\dfrac{\mh^{2}\eps^{2}}{2}\bE_{[Z^{1}]}(Z^{1}), \\
&\dot{\bz}^{n+1}=\fe^{\frac{\mh}{2}\widehat{\bB}(\bz^{n+1})}\fe^{\frac{\mh}{2}\widehat{\bB}(\bz^{n})}\dot{\bz}^{n}
+\fe^{\frac{\mh}{2}\widehat{\bB}(\bz^{n+1})}\mh\eps^{2}(1+\tilde{\gamma}^{n})\bE_{[Z^{1}]}(Z^{1}),
\end{aligned}
\end{equation}
where
$$
\tilde{\gamma}_{n}=\(-\tilde{\mathcal{C}}+\sgn(\tilde{\mathcal{C}})\sqrt{\tilde{\mathcal{C}}^{2}-2\tilde{\mathcal{A}}(\bar{H}(\bz^{n+1},\dot{\bz}^{n+1})-\bar{H}(\bz^{n},\dot{\bz}^{n}))}\)
\Big/\mh\eps^{2}\tilde{\mathcal{A}}
$$
with $\tilde{\mathcal{A}}=\sum\limits_{k=1}^{N_{p}}\omega_{k}(\bE_{[Z^1]}(Z^{1}_{k}))^{\intercal}\bE_{[Z^1]}(Z^{1}_{k})$ and $\tilde{\mathcal{C}}=\sum\limits_{k=1}^{N_{p}}\omega_{k}(\bE_{[Z^1]}(Z^{1}_{k}))^{\intercal}\dot{\tilde{\bz}}^{n+1}_{k}.$
Here $$\bar{H}(\bz,\dot{\bz})=\frac{1}{2}\sum\limits_{k=1}^{N_{p}}\omega_{k}\abs{\dot{\bz}_{k}}^{2}+\frac{\eps^{2}}{2}
\int_{\Omega_{z}}\abs{\bE_{[\bz^{p}(t)]}(\bz)}^{2}d\bz$$
and $$\dot{\tilde{\bz}}^{n+1}=\fe^{\frac{\mh}{2}\widehat{\bB}(\bz^{n})}\dot{\bz}^{n}+\mh\eps^{2}\bE_{[Z^{1}]}(Z^{1}).$$

When $h=\eps\mh$ is chosen, it follows directly from \eqref{equ-10-4-1} that for all $n\geq 0$,
\begin{equation}\label{equ-10-4-3}
\bx(t_{n})=\bz(\tau_{n}), \ \eps\dot{\bx}(t_{n})=\dot{\bz}(\tau_{n}), \ \bx^{n}=\bz^{n}, \ \eps\dot{\bx}^{n}=\dot{\bz}^{n}, \ \gamma_{n}=\tilde{\gamma}_{n}.
\end{equation}
Recalling that the relaxation parameter $\gamma_{n}$ satisfies $\gamma_{n}=\mathcal{O}(h^{2})$, it can be deduced that $\tilde{\gamma}_{n}=\mathcal{O}(\mh^{2}\eps^{2})$. The main convergence result of the RS2-PIC scheme is established as follows.

\subsection{A rough error estimate}

A preliminary coarse bound for the numerical solution is first established. Specifically, for a given time $\tau=\tau_{n}+s$ with $n\geq 0$, we introduce a truncated system of \eqref{long-sys} as
\begin{equation}\label{trun-sys}
\ddot{\tilde{\bz}}^{n}(s)=\dot{\tilde{\bz}}^{n}(s)\times\bB(\bz(\tau_{n}+ \mh/2))+\eps^{2}\bE_{[\tilde{\bz}^{n}(s)]}(\tilde{\bz}^{n}(s)), \ 0<s\leq \mh, \ \
\tilde{\bz}^{n}(0)=\bz(\tau_{n}), \ \dot{\tilde{\bz}}^{n}(0)=\dot{\bz}(\tau_{n}).
\end{equation}
It follows directly that for all $0\leq n<T/(\eps\mh)$, there exists a uniform constant $C>0$ depending on $\norm{\bz}_{L^{\infty}(0,T/\eps)}$, $\norm{\dot{\bz}}_{L^{\infty}(0,T/\eps)}$ and norms of $\bB$ and $\bE$ such that
$
\norm{\tilde{\bz}^{n}}_{L^{\infty}(0,\mh)}\leq C, \ \norm{\dot{\tilde{\bz}}^{n}}_{L^{\infty}(0,\mh)}\leq C\eps.
$

For the numerical scheme \eqref{st-rsv-resc}, the local truncation errors $\xi_{\mathbf{z}}^{n}$ and $\xi_{\dot{\mathbf{z}}}^{n}$ are defined for $0 \leq n<T/(\eps\mh)$ as:
\begin{equation}\label{equ-10-5-3}
\begin{aligned}
\tilde{\bz}^{n}(\mh)=&\bz(\tau_n)+\mh\fe^{\frac{\mh}{2}\hbB(\bz(\tau_n))}\dot{\bz}(\tau_n)+\frac{\eps^{2}\mh^2}{2}
\bE 
(\bz(\tau_n)+\frac{\mh}{2}\fe^{\frac{\mh}{2}\hbB(\bz(\tau_n))}\dot{\bz}(\tau_n))+\xi_{\bz}^{n}, \\
\dot{\tilde{\bz}}^{n}(\mh)=&\fe^{\frac{\mh}{2}\hbB(\tilde{\bz}(\mh))}\fe^{\frac{\mh}{2}\hbB(\bz(\tau_n))}\dot{\bz}(\tau_n) +\fe^{\frac{\mh}{2}\hbB(\tilde{\bz}(\mh))}\mh\eps^{2}(1+\tilde{\gamma}^n)
\bE 
(\bz(\tau_n)+\frac{\mh}{2}\fe^{\frac{\mh}{2}\hbB(\bz(\tau_n))}\dot{\bz}(\tau_n))+\xi_{\dot{\bz}}^{n}. 
\end{aligned}
\end{equation}
Here we introduce the compact notation $\bE(\bx)$ to denote the function of $\bE_{[\bx]}(\bx)$. The local error of the truncated system is characterized by the following Lemma.

\begin{lemma}(Local error)\label{loca-err}
Under the condition that $\bB(\cdot), \bE(\cdot)\in C^{1}(\mathbb{R}^d)$, let $\bz^{n}, \dot{\bz}^{n}$ denote the numerical solution obtained from the RS2-PIC \eqref{st-rsv-resc} for solving \eqref{long-sys} up to $T/\eps$ for a fixed $T>0$. Then there exists a constant $\mh_0>0$ such that $0<\mh\leq \mh_{0}$, the local error $\xi_{\bz}^{n}$ and $\xi_{\dot{\bz}}^{n}$ of the scheme \eqref{st-rsv-resc} for truncated system \eqref{trun-sys} satisfies
\begin{equation}\label{equ-10-9-1}
\abs{\xi_{\bz}^{n}}\lesssim \eps\mh^{3},\quad \abs{\xi_{\dot{\bz}}^{n}}\lesssim \eps^{2}\mh^{3}, \quad 0\leq n<T/(\eps\mh).
\end{equation}
\end{lemma}
\begin{proof}We only prove the  second statement and skip the proof of the first one  for brevity.

 Applying the variation-of-constants formula to the truncated system \eqref{trun-sys} gives
\begin{equation}\label{equ-10-7-1}
\begin{aligned}
\tilde{\bz}^{n}(\mh)=&\bz(\tau_n)+\int_{0}^{\mh}\dot{\tilde{\bz}}^{n}(s)ds,   \ \
\dtbz^{n}(\mh)=\fe^{\mh\hbB(\bz(\tau_n+\frac{\mh}{2}))}\dtbz(\tau_n)+\eps^{2}\int_{0}^{\mh}\fe^{(\mh-s)\hbB(\bz(\tau_n+\frac{\mh}{2}))}\bE(\tbz^{n}(s))ds,
\end{aligned}
\end{equation}
which further implies
\begin{equation}\label{equ-10-7-2}
\tbz^{n}(\mh)=\bz(\tau_n)+\int_{0}^{\mh}\fe^{s\hbB(\bz(\tau_n+\frac{\mh}{2}))}ds\dot{\bz}(\tau_n)+\eps^{2}\int_{0}^{\mh}\int_{0}^{s}\fe^{(s-\sigma)\hbB(\bz(\tau_n+\frac{\mh}{2}))}
\bE(\tbz^{n}(\sigma))d\sigma ds.
\end{equation} Subtracting \eqref{equ-10-5-3} from \eqref{equ-10-7-1} yields
$$
\xi_{\dot{\bz}}^{n}=\xi_{\dot{\bz},1}^{n}+\xi_{\dot{\bz},2}^{n}, \ 0\leq n<T/(\eps\mh),
$$
where
\begin{equation}\label{equ-10-7-3}
\begin{aligned}
\xi_{\dot{\bz},1}^{n}=&\fe^{\mh\hbB(\bz(\tau_n+\frac{\mh}{2}))}\dot{\bz}(\tau_n)-\fe^{\frac{\mh}{2}\hbB(\tbz^{n}(\mh))}\fe^{\frac{\mh}{2}\hbB(\bz(\tau_n))}\dot{\bz}(\tau_n),   \\
\xi_{\dot{\bz},2}^{n}=&\eps^{2}\int_{0}^{\mh}\fe^{(\mh-s)\hbB(\bz(\tau_n+\frac{\mh}{2}))}\bE(\tbz^{n}(s))ds  -\fe^{\frac{\mh}{2}\hbB(\tbz^{n}(\mh))}h\eps^{2}
(1+\tilde{\gamma}_{n})\bE\left(\bz(\tau_n)+\frac{\mh}{2}\fe^{\frac{\mh}{2}\hbB(\bz(\tau_n))}\dot{\bz}(\tau_n)\right).
\end{aligned}
\end{equation}

\textbf{(I): Estimate of $\xi_{\dot{\bz},1}^{n}$.}
The term $\xi_{\dot{\bz},1}^{n}$ is derived as
$
\xi_{\dot{\bz},1}^{n}=\xi_{\dot{\bz},1,1}^{n}+\xi_{\dot{\bz},1,2}^{n},
$
where 
\begin{equation*}
\begin{aligned}
\xi_{\dot{\bz},1,1}^{n}=&\fe^{\mh\hbB(\bz(\tau_{n}+\frac{\mh}{2}))}\dot{\bz}(\tau_{n})-\fe^{\frac{\mh}{2}\hbB(\bz(\tau_{n+1}))}\fe^{\frac{\mh}{2}\hbB(\bz(\tau_{n}))}\dot{\bz}(\tau_{n}), \\
\xi_{\dot{\bz},1,2}^{n}=&\fe^{\frac{\mh}{2}\hbB(\bz(\tau_{n+1}))}\fe^{\frac{\mh}{2}\hbB(\bz(\tau_{n}))}\dot{\bz}(\tau_{n})
-\fe^{\frac{\mh}{2}\hbB(\tbz^{n}(\mh))}\fe^{\frac{\mh}{2}\hbB(\bz(\tau_{n}))}\dot{\bz}(\tau_{n}).
\end{aligned}
\end{equation*}
Using the Taylor expansion with integral remainder, i.e., $
\fe^{\mh\hbB}=I+\mh\hbB+\frac{\mh^2}{2}\hbB^{2}+\frac{1}{2}\int_{0}^{\mh}\fe^{t\hbB}\hbB^{3}(\mh-t)^{2}dt,$
one obtains $
\abs{\xi_{\dot{\bz},1,1}^{n}}\lesssim \abs{\mh\left[\hbB(\bz(\tau_{n}+\frac{\mh}{2}))-\frac{1}{2}\left(\hbB(\bz(\tau_{n+1}))-\hbB(\bz(\tau_{n}))\right) \right]\dot{\bz}(\tau_{n})+R_{2}^{n}\dot{\bz}(\tau_{n})+R_{3}^{n}\dot{\bz}(\tau_{n})},
$
where $R_{2}^{n}$ is given by $
R_{2}^{n}=\frac{\mh^{2}}{2}\left[\hbB^{2}(\bz(\tau_{n}+\frac{\mh}{2}))-\frac{1}{4}\(\hbB^{2}(\bz(\tau_{n+1}))+\hbB^{2}(\bz(\tau_{n}))\)-\frac{1}{2}\hbB(\bz(\tau_{n+1}))\hbB(\bz(\tau_{n}))   \right], $
and $R_{3}^{n}$ denotes the corresponding higher-order remainder term. Given that $\bB(\bx)=\bB_{0}+\eps\bB_{1}(\bx)$, it follows that
\begin{equation}\label{equ-10-7-8}
\nabla\bB(\bx)=\eps\nabla \bB_{1}(\bx).
\end{equation}
Performing Taylor expansions of $\hbB(\bz(\tau_{n+1}))$ and $\hbB(\bz(\tau_{n}))$ about $\hbB(\bz(\tau_{n}+\frac{\mh}{2}))$ and using \eqref{equ-10-7-8}, we have
$\abs{\hbB(\bz(\tau_{n}+\frac{\mh}{2}))-\frac{1}{2}\(\hbB(\bz(\tau_{n+1}))+\hbB(\bz(\tau_{n}))\)}\lesssim \eps^{2}\mh^{2}$. Introducing the notation $\bB^{n}=\bB(\bz(\tau_{n}+\frac{\mh}{2}))-\frac{1}{2}\left(\bB(\bz(\tau_{n+1}))+\bB(\bz(\tau_{n}))\right)$, the term $R_{2}^{n}\dot{\bz}(\tau_{n})$ can be expressed as $$\abs{R_{2}^{n}\dot{\bz}(\tau_{n})}=\frac{\mh^2}{2}\abs{\dot{\bz}(\tau_{n})\times \bB^{n} \times \bB(\bz(\tau_{n}+\frac{\mh}{2}))+R_{2,1}^{n}+R_{2,2}^{n}},$$where
\begin{equation*}
\begin{aligned}
R_{2,1}^{n}=&\frac{1}{2}\dot{\bz}(\tau_{n})\times \bB(\bz(\tau_{n+1}))\times \left(\bB(\bz(\tau_{n}+\frac{\mh}{2}))-\bB(\bz(\tau_{n+1}))\right) \\
&+\frac{1}{2}\dot{\bz}(\tau_{n})\times \bB(\bz(\tau_{n}))\times \left(\bB(\bz(\tau_{n}+\frac{\mh}{2}))-\bB(\bz(\tau_{n}))\right),\\
R_{2,2}^{n}=&\frac{1}{4}\dot{\bz}(\tau_{n})\times \left(\bB(\bz(\tau_{n+1}))-\bB(\bz(\tau_{n}))\right)+\frac{1}{4}\dot{\bz}(\tau_{n})\times \bB(\bz(\tau_{n}))\times \left(\bB(\bz(\tau_{n}))-\bB(\bz(\tau_{n+1}))\right).
\end{aligned}
\end{equation*}
It is obvious that $\abs{R_{2,1}^{n}}\lesssim \eps^{3}\mh$, $\abs{R_{2,2}^{n}}\lesssim \eps^{3}\mh$, which yields
$\abs{R_{2}^{n}\dot{\bz}(\tau_{n})}\lesssim \eps^{3}\mh^{3}$. Combining these estimates, we obtain
$
\abs{\xi_{\dot{\bz},1,1}^{n}}\lesssim \eps^{3}\mh^{3}.
$
Using the inequality $
\abs{\fe^{\frac{\mh}{2}\hbB(\bz(\tau_{n+1}))}-\fe^{\frac{\mh}{2}\hbB(\tbz^{n}(\mh))}}\lesssim \mh\eps\abs{\bz(\tau_{n+1})-\tbz^{n}(\mh)}\lesssim \mh\eps \abs{\zeta_{\bz}^{n}(\mh)}\lesssim \eps^{4}\mh^{5},$
it follows that $\abs{\xi_{\dot{\bz},1,2}^{n}}\lesssim \eps^{5}\mh^{5}$. Overall, this leads to
\begin{equation}\label{equ-10-8-1}
\abs{\xi_{\dot{\bz},1}^{n}}\lesssim \abs{\xi_{\dot{\bz},1,1}^{n}}+\abs{\xi_{\dot{\bz},1,2}^{n}} \lesssim \eps^{3}\mh^{3}.
\end{equation}

\textbf{(II): Estimate of $\xi_{\dot{\bz},2}^{n}$.} Applying a Taylor expansion to \eqref{equ-10-7-3} leads to
\begin{equation*}
\begin{aligned}
&\eps^{2}\int_{0}^{\mh}\fe^{(\mh-s)\hbB(\bz(\tau_n+\frac{\mh}{2}))}\bE(\tbz^{n}(s))ds
=\eps^{2}\fe^{\frac{\mh}{2}\hbB(\bz(\tau_n+\frac{\mh}{2}))}\int_{0}^{\mh}\fe^{(\frac{\mh}{2}-s)\hbB(\bz(\tau_n+\frac{\mh}{2}))}\bE(\tbz^{n}(s))ds \\
=&\eps^{2}\fe^{\frac{\mh}{2}\hbB(\bz(\tau_n+\frac{\mh}{2}))}\int_{0}^{\mh}\left[I+(\frac{\mh}{2}-s)\hbB(\bz(\tau_n+\frac{\mh}{2}))+\frac{1}{2}(s-\frac{\mh}{2})^{2} 
\fe^{(\frac{\mh}{2}-\tau_{s}^{n})\hbB(\bz(\tau_n+\frac{\mh}{2}))}\hbB^{2}(\bz(\tau_n+\frac{\mh}{2})) \right] \bE(\tbz^{n}(s))ds \\
=&\eps^{2}\fe^{\frac{\mh}{2}\hbB(\bz(\tau_n+\frac{\mh}{2}))}\int_{0}^{\mh}\bE(\tbz^{n}(s))ds+\delta_{1}+\delta_{2},
\end{aligned}
\end{equation*}
where the $ \abs{\delta_{1}}=\abs{\eps^{2}\fe^{\frac{\mh}{2}\hbB(\bz(\tau_n+\frac{\mh}{2}))}\int_{0}^{\mh}(\frac{\mh}{2}-s)\hbB(\bz(\tau_n+\frac{\mh}{2}))\bE(\tbz^{n}(s))ds}
\lesssim \eps^{2}\mh^{3}$ 
is estimated utilizing the midpoint integral formula and
\begin{align*}
\abs{\delta_{2}}=\abs{\eps^{2}\fe^{\frac{\mh}{2}\hbB(\bz(\tau_n+\frac{\mh}{2}))}\int_{0}^{\mh}\frac{1}{2}(s-\frac{\mh}{2})^{2} 
\fe^{(\frac{\mh}{2}-\tau_{s}^{n})\hbB(\bz(\tau_n+\frac{\mh}{2}))}\hbB^{2}(\bz(\tau_n+\frac{\mh}{2}))\bE(\tbz^{n}(s))ds}\lesssim \eps^{2}\mh^{3}.
\end{align*} 
From the midpoint integral formula again, one obtains
\begin{align*}
&\int_{0}^{\mh}\bE\left(\bz(\tau_n)+s\fe^{s\hbB(\bz(\tau_n))}\dot{\bz}(\tau_n))\right)ds 
= \mh\bE\left(\bz(\tau_n)+\frac{\mh}{2}\fe^{\frac{\mh}{2}\hbB(\bz(\tau_n))}\dot{\bz}(\tau_{n}))\right)
+\mathcal{O}(\eps \mh^{3}).
\end{align*} 
It follows that
\begin{equation}\label{equ-10-7-4}
\begin{aligned}
\xi_{\dot{\bz},2}^{n}
=&\eps^{2}\left(\fe^{\frac{\mh}{2}\hbB(\bz(\tau_n+\frac{\mh}{2}))}-\fe^{\frac{\mh}{2}\hbB(\tbz^{n}(\mh))}\right)\int_{0}^{\mh}\bE(\tbz^{n}(s))ds+\mathcal{O}(\eps^{2}\mh^{3}) \\
&+\eps^{2}\fe^{\frac{\mh}{2}\hbB(\tbz^{n}(\mh))}\int_{0}^{\mh}\left[\bE(\tbz^{n}(s)) -\bE\left(\bz(\tau_n)+s\fe^{s\hbB(\bz(\tau_n))}\dot{\bz}(\tau_{n}))\right)\right]ds  \\
&+\mathcal{O}(\eps^{3}\mh^{3})-\mh \eps^{2}\tilde{\gamma}_{n}\fe^{\frac{\mh}{2}\hbB(\tbz^{n}(\mh))}\bE
\left(\bz(\tau_n)+\frac{\mh}{2}\fe^{\frac{\mh}{2}\hbB(\bz(\tau_n))}\dot{\bz}(\tau_n))\right).
\end{aligned}
\end{equation}
Building upon \eqref{equ-10-7-8}, the estimate
\begin{equation}\label{equ-10-7-5}
\begin{aligned}
\abs{\fe^{\frac{\mh}{2}\hbB(\bz(\tau_n+\frac{\mh}{2}))}-\fe^{\frac{\mh}{2}\hbB(\tbz^{n}(\mh))}}
\lesssim \eps^{2}\mh^{2}+\eps\mh \abs{\zeta_{\bz}^{n}(\mh)}\lesssim \eps^{2}\mh^{2}
\end{aligned}
\end{equation}
is obtained. Expanding with Taylor's formula gives
\begin{equation*}
\bz(\tau_{n})+s\fe^{s\hbB(\bz(\tau_{n}))}\dot{\bz}(\tau_{n})=\bz(\tau_{n})+s\left[I+\int_{0}^{s}\hbB\fe^{\alpha \hbB}d\alpha\right]\dot{\bz}(\tau_{n})=\bz(\tau_{n})+s\dot{\bz}(\tau_{n})+s\int_{0}^{s}\hbB\fe^{\alpha \hbB}d\alpha \dot{\bz}(\tau_{n}).
\end{equation*}
This leads to
\begin{equation*}
\begin{aligned}
\tbz^{n}(s)=&\bz(\tau_{n}+s)-\zeta_{\bz}^{n}(s)=\bz(\tau_{n})+s\dot{\bz}(\tau_{n})+\int_{0}^{s}(s-\sigma)\ddot{\bz}(\tau_{n}+\sigma)d\sigma-\zeta_{\bz}^{n}(s) \\
=&\bz_(\tau_{n})+s\fe^{s\hbB(\bz(\tau_{n}))}\dot{\bz}(\tau_{n})-s\int_{0}^{s}\hbB\fe^{\alpha \hbB}d\alpha \dot{\bz}(\tau_{n})+\int_{0}^{s}(s-\sigma)\ddot{\bz}(\tau_{n}+\sigma)d\sigma-\zeta_{\bz}^{n}(s).
\end{aligned}
\end{equation*}
Therefore, it is obtained that
\begin{equation*}
\begin{aligned}
\int_{0}^{\mh}\bE(\tbz^{n}(s))ds=&\int_{0}^{\mh}\bE\left(\bz(\tau_{n})+s\fe^{s\hbB(\bz(\tau_{n}))}\dot{\bz}(\tau_{n})\right)ds +\int_{0}^{h}\int_{0}^{1}\nabla\bE(s_{\rho})d\rho\\
& \left[-s\int_{0}^{s}\hbB\fe^{\alpha \hbB}d\alpha \dot{\bz}(\tau_{n})+\int_{0}^{s}(s-\sigma)\ddot{\bz}(\tau_{n}+\sigma)d\sigma-\zeta_{\bz}^{n}(s)\right]ds,
\end{aligned}
\end{equation*}
where 
$
s_{\rho}=\bz(\tau_{n})+s\fe^{s\hbB(\bz(\tau_{n}))}\dot{\bz}(\tau_{n})+\rho\left[-s\int_{0}^{s}\hbB\fe^{\alpha \hbB}d\alpha \dot{\bz}(\tau_{n})+\int_{0}^{s}(s-\sigma)\ddot{\bz}(\tau_{n}+\sigma)d\sigma-\zeta_{\bz}^{n}(s)\right].
$
Moreover, 
\begin{equation*}
\begin{aligned}
\abs{\int_{0}^{h}\int_{0}^{1}\nabla\bE(s_{\rho})d\rho \left[-s\int_{0}^{s}\hbB\fe^{\alpha \hbB}d\alpha \dot{\bz}(\tau_{n})+\int_{0}^{s}(s-\sigma)\ddot{\bz}(\tau_{n}+\sigma)d\sigma-\zeta_{\bz}^{n}(s)\right]ds} 
\lesssim  \eps\mh^{3}+\eps^{3}\mh^{5}\lesssim \eps\mh^{3},
\end{aligned}
\end{equation*}
which implies that
\begin{equation}\label{equ-10-7-6}
\abs{\int_{0}^{\mh}\bE(\tbz^{n}(s))ds-\int_{0}^{\mh}\bE\left(\bz(\tau_{n})+s\fe^{s\hbB(\bz(\tau_{n}))}\dot{\bz}(\tau_{n})\right)ds }\lesssim \eps\mh^{3}.
\end{equation}
Inserting \eqref{equ-10-7-5} and \eqref{equ-10-7-6} into \eqref{equ-10-7-4}, and noting that $\tilde{\gamma}_{n}=\mathcal{O}(\eps^{2}\mh^{2})$, one obtains
\begin{equation}\label{equ-10-7-7}
\abs{\xi_{\dot{\bz},2}^{n}}\lesssim \eps^{4}\mh^{3}+\eps^{3}\mh^{3}+\eps^{4}\mh^{3}+\eps^{2}\mh^{3}+\eps^{2}\mh^{3}+\eps^{3}\mh^{3}\lesssim \eps^{2}\mh^{3}.
\end{equation}

Combining \eqref{equ-10-8-1} with \eqref{equ-10-7-7} yields $
\abs{\xi_{\dot{\bz}}^{n}}\lesssim \abs{\xi_{\dot{\bz},1}^{n}}+\abs{\xi_{\dot{\bz},2}^{n}}\lesssim \eps^{3}\mh^{3}+\eps^{2}\mh^{3} \lesssim \eps^{2}\mh^{3}
$, which establishes the second estimate in \eqref{equ-10-9-1}. 
\end{proof}

\begin{proposition}\label{pro-2}
Under the conditions of Lemma \ref{loca-err}, there exists a constant $\mh_0>0$ that is independent of $\eps$, such that if the time step satisfies $0<\mh\leq \mh_{0}$, we have
\begin{equation*}
\abs{\bz^{n}-\bz(\tau_{n})}\lesssim \mh^{2}, \quad \abs{\dot{\bz}^{n}-\dot{\bz}(\tau_{n})}\lesssim \eps\mh^{2}, \quad 0\leq n \leq T/(\eps\mh), 
\end{equation*}
and
\begin{equation}\label{solu-bound}
\abs{\bz^{n}}\leq \norm{\bz}_{L^{\infty}(0,T/\eps)}+1, \quad \abs{\dot{\bz}^{n}}\leq \norm{\dot{\bz}}_{L^{\infty}(0,T/\eps)}+\eps, \quad 0\leq n\leq T/(\eps\mh).
\end{equation}
\end{proposition}

\begin{proof}
The error introduced by truncating the system is first analyzed. Defining
\begin{equation*}
\zeta_{\bz}^{n}(s):=\bz(\tau_{n}+s)-\tilde{\bz}^{n}(s), \ \zeta_{\dot{\bz}}^{n}(s):=\dot{\bz}(\tau_{n}+s)-\dot{\tilde{\bz}}^{n}(s), \ 0\leq n<T/(\eps\mh),
\end{equation*}
and subtracting \eqref{trun-sys} from \eqref{long-sys} yields, for $0\leq n<T/(\eps\mh)$,
\begin{equation}\label{equ-10-4-4}
\ddot{\zeta}_{\bz}^{n}(s)=\zeta_{\dot{\bz}}^{n}(s)\times \bB(\bz(\tau_{n}+ \mh/2 ))+\eps^{2}\bE(\bz(\tau_{n}+s))-\eps^{2}\bE(\tilde{\bz}^{n}(s))+\xi^{n}_{0}(s), \
\zeta_{\bz}^{n}(0)=0,\  \zeta_{\dot{\bz}}^{n}(0)=0,
\end{equation}
where
$
\xi_{0}^{n}(s)=\dot{\bz}(\tau_{n}+s)\times[\bB(\bz(\tau_{n}+s))-\bB(\bz(\tau_{n}+ \mh/2 ))].
$
An application of the Duhamel's formula to \eqref{equ-10-4-4} gives
$
\zeta_{\bz}^{n}(\mh)=\int_{0}^{\mh}\zeta_{\dot{\bz}}^{n}(s)ds$ and 
\begin{equation*}
\begin{aligned}\zeta_{\dot{\bz}}^{n}(s)=&\int_{0}^{\mh}\fe^{(\mh-s)\widehat{\bB}(\bz(\tau_n+ \mh/2 ))}[\eps^{2}\bE(\bz(\tau_n+s))-\eps^{2}
\bE(\tilde{\bz}(s))+\xi_{0}^{n}(s)]ds  \\
=&\int_{0}^{\mh}\fe^{(\mh-s)\widehat{\bB}(\bz(\tau_n+ \mh/2 ))}\left[\eps^{2}\int_{0}^{1}\nabla\bE \left(\bz
(\tau_n+s)+(\rho-1)\zeta_{\bz}^{n}(s)\right)d\rho+\xi_{0}^n(s)\right]ds. \end{aligned}
\end{equation*}
Combining these two results yields
\begin{equation*}
\begin{aligned}
\zeta_{\bz}^{n}(\mh)=&\eps^{2}\int_{0}^{\mh}\int_{0}^{s}\fe^{(s-\sigma)\widehat{\bB}(\bz(\tau_n+ \mh/2 ))}\Big(\int_{0}^{1}\nabla\bE\left(\bz
(\tau_n+\sigma)+(\rho-1)\zeta_{\bz}^{n}(\sigma)\right)d\rho  + \xi_{0}^{n}(\sigma)\Big)d\sigma ds.
\end{aligned}
\end{equation*}
Denoting $\widehat{\bB}_{n}=\widehat{\bB}(\bz(\tau_n+\mh/2))$ for brevity, we have
\begin{equation*}
\begin{aligned}
&\abs{\int_{0}^{\mh}\fe^{(\mh-s)\hbB(\bz(\tau_n+ \mh/2 ))}\xi_{0}^{n}(s)ds}=\abs{\int_0^\mh \fe^{(\mh-s)\hbB_{n}}\dot{\bz}(\tau_n+s)\times\left[\bB(\bz(\tau_n+s)-\bz(\tau_n+ \mh/2 ))\right]ds} \\
\lesssim &\eps^{3}\abs{\int_0^\mh \fe^{(\mh-s)\hbB_{n}}\frac{\dot{\bz}(\tau_n+s)}{\norm{\dot{\bz}}_{L^\infty(0,\mh)}}\times\left(\int_{ \mh/2 }^{s}
\nabla\bB_{1}(\bz(\tau_n+\rho))\frac{\dot{\bz}(\tau_n+\rho)}{\norm{\dot{\bz}}_{L^{\infty}(0,\mh)}}
d\rho \right)ds} =\eps^{3}\abs{\int_0^\mh \fe^{(\mh-s)\hbB_{n}} F_{n}(s)ds},
\end{aligned}
\end{equation*}
where 
$
F_{n}(s)=\frac{\dot{\bz}(\tau_n+s)}{\norm{\dot{\bz}}_{L^\infty(0,\mh)}}\times \int_{ \mh/2 }^{s}
\nabla\bB_{1}(\bz(\tau_n+\rho))\frac{\dot{\bz}(\tau_n+\rho)}{\norm{\dot{\bz}}_{L^{\infty}(0,\mh)}}
d\rho. 
$
Applying the midpoint quadrature rule and using $F(\mh/2)=0$ gives
$$
\abs{\int_{0}^{\mh}\fe^{(\mh-s)\hbB(\bz(\tau_n+ \mh/2 ))}\xi_{0}^{n}(s)ds}\lesssim \eps^{3}\abs{\int_0^\mh \fe^{(\mh-s)\hbB_{n}} F_{n}(s)ds} \lesssim \eps^{3}\mh^{3}.
$$
A standard bootstrap argument then yields the bound $\abs{\zeta_{\bz}^{n}(s)}\lesssim \eps^{3}s^{4}$, $s\in [0,\mh]$ for $h\lesssim 1$. This estimate implies that for all $0\leq n<T/(\eps\mh)$,
\begin{equation}\label{equ-10-5-1}
\abs{\zeta_{\bz}^{n}(\mh)}\lesssim \eps^{3}\mh^{4}, \ \abs{\zeta_{\dot{\bz}}^{n}(\mh)}\lesssim \eps^{3}\mh^{3}.
\end{equation}

We now estimate the error of the scheme
\begin{equation*}
e_{\bz}^{n+1}:=\bz(\tau_{n+1})-\bz^{n+1}, \ e_{\dot{\bz}}^{n+1}:=\dot{\bz}(\tau_{n+1})-\dot{\bz}^{n+1}, \ 0\leq n<T/(\eps\mh).
\end{equation*}
Inserting the truncated solution
\begin{equation}\label{equ-10-5-2}
e_{\bz}^{n+1}=\tilde{e}_{\bz}^{n}+\zeta_{\bz}^{n}(\mh), \ e_{\dot{\bz}}^{n+1}=\tilde{e}_{\dot{\bz}}^{n}+\zeta_{\dot{\bz}}^{n}(\mh),
\end{equation}
so that the analysis reduces to estimating $
\tilde{e}_{\bz}^{n}:=\tilde{\bz}^{n}(\mh)-\bz^{n+1}, \ \tilde{e}_{\dot{\bz}}^{n}:=\dot{\tilde{\bz}}^{n}(\mh)-\dot{\bz}^{n+1}.$

With the preceding preparations, an inductive argument is employed to establish the boundedness of the numerical solution in \eqref{solu-bound}. For $n=0$, \eqref{solu-bound} holds trivially since $\bz^{0}=\bx_{0}$, $\dot{\bz}^{0}=\eps\bv_{0}$. Assuming that \eqref{solu-bound} is true up to some $0\leq m<T/(\eps\mh)$, it remains to verify that the bound also holds for $m+1$.

For $n\geq m$, subtracting \eqref{equ-10-5-3} form the scheme \eqref{st-rsv-resc}, and using \eqref{equ-10-5-2} yields
\begin{subequations}\label{equ-10-9-6}
\begin{align}
e_{\bz}^{n+1}=&e_{\bz}^{n}+\mh\fe^{ \mh/2 \hbB(\bz(\tau_{n}))}e_{\dot{\bz}}^{n}+\eta_{\bz}^{n}+\xi_{\bz}^{n}+\zeta_{\bz}^{n}, \label{equ-10-9-6a} \\
e_{\dot{\bz}}^{n+1}=&\fe^{ \mh/2 \hbB(\tbz^{n}(\mh))}\fe^{ \mh/2 \hbB(\bz(\tau_{n}))}e_{\dot{\bz}}^{n}+\eta_{\dot{\bz}}^{n}+\xi_{\dot{\bz}}^{n}+\zeta_{\dot{\bz}}^{n},
\ 0\leq n\leq m, \label{equ-10-9-6b}
\end{align}
\end{subequations}
with the auxiliary terms defined as
\begin{equation*}
\begin{aligned}
\eta_{\bz}^{n}=&\mh\left(\fe^{ \frac{\mh}{2}  \hbB(\bz(\tau_{n}))}-\fe^{ \frac{\mh}{2}\hbB(\bz^{n})}\right)\dot{\bz}^{n}+\frac{\eps^{2}\mh^{2}}{2}
\bigg[\bE\left(\bz(\tau_{n})+\frac{\mh}{2} \fe^{\frac{\mh}{2} \hbB(\bz(\tau_{n}))}\dot{\bz}(\tau_{n})\right)  
-\bE \left(\bz^{n}+\frac{\mh}{2} \fe^{\frac{\mh}{2} \hbB(\bz^{n})}\dot{\bz}^{n}\right) 
\bigg], \\
\eta_{\dot{\bz}}^{n}=&\left[\fe^{ \mh/2 \hbB(\tbz^{n}(\mh))}\fe^{ \mh/2 \hbB(\bz(\tau_{n}))}-\fe^{ \mh/2 \hbB(\bz^{n+1})}\fe^{ \mh/2 \hbB(\bz^{n})} \right]\dot{\bz}^{n}+r_{\dot{\bz}}^{n}+\mh\eps^{2}(1+\tilde{\gamma}_{n})\fe^{ \mh/2 \hbB(\tbz^{n}(\mh))} \\
&\bigg[\bE \left(\bz(\tau_{n})+ \mh/2 \fe^{ \mh/2 \hbB(\bz(\tau_{n}))}\dot{\bz}(\tau_{n})\right) -\bE \left(\bz^{n}+ \mh/2 \fe^{ \mh/2 \hbB(\bz^{n})}\dot{\bz}^{n}\right) 
\bigg],
\end{aligned}
\end{equation*}
and
\begin{equation*}
r_{\dot{\bz}}^{n}=\mh\eps^{2}(1+\tilde{\gamma}_{n})\left[\fe^{ \mh/2 \hbB(\tbz^{n}(\mh))}-\fe^{ \mh/2 \hbB(\bz^{n+1})}\right]\bE
\left(\bz^{n}+ \mh/2 \fe^{ \mh/2 \hbB(\bz^{n})}\dot{\bz}^{n}\right).
\end{equation*} 
It is dirct to observe that
$
\abs{r_{\dot{\bz}}^{n}}\lesssim \eps^{3}\mh^{2}(\abs{\zeta_{\bz}^{n}(\mh)}+\abs{e_{\bz}^{n+1}}).
$
The induction hypothesis gives the bounds
\begin{equation}\label{equ-10-9-7}
\abs{\eta_{\bz}^{n}}\lesssim \mh^{2}\eps^{2}(\abs{e_{\bz}^{n}}+\abs{e_{\dot{\bz}}^{n}}), \quad \abs{\eta_{\dot{\bz}}^{n}}\lesssim \eps^{2}\mh(\abs{e_{\bz}^{n}}+\abs{\zeta_{\bz}^{n}(\mh)}+\abs{e_{\bz}^{n+1}}+\abs{e_{\dot{\bz}}^{n}}), \quad 0\leq n<m.
\end{equation}
By taking the Euclidean norm on both sides of \eqref{equ-10-9-6a} and \eqref{equ-10-9-6b} and adopting the triangle inequality, while noting the orthogonality of the matrix $\fe^{\mh\hbB}$, we get
\begin{equation}\label{equ-10-9-8}
\abs{e_{\bz}^{n+1}}\leq \abs{e_{\bz}^{n}}+\mh\abs{e_{\dot{\bz}}^{n}}+\abs{\eta_{\bz}^{n}}+\abs{\xi_{\bz}^{n}}+\abs{\zeta_{\bz}^{n}(\mh)}, \ 
\abs{e_{\dot{\bz}}^{n+1}}\leq \abs{e_{\dot{\bz}}^{n}}+\abs{\eta_{\dot{\bz}}^{n}}+\abs{\xi_{\dot{\bz}}^{n}}+\abs{\zeta_{\dot{\bz}}^{n}(\mh)}. 
\end{equation}
Multiplying \eqref{equ-10-9-8} by $1/\eps$ and applying \eqref{equ-10-9-7} yields
\begin{equation*}
\abs{e_{\bz}^{n+1}}+\frac{1}{\eps}\abs{e_{\dot{\bz}}^{n+1}}-\abs{e_{\bz}^{n}}-\frac{1}{\eps}\abs{e_{\dot{\bz}}^{n}}\leq h\eps (\abs{e_{\bz}^{n}}+\abs{e_{\bz}}^{n+1}+\frac{1}{\eps}\abs{e_{\dot{\bz}}^{n}})+\abs{\xi_{\bz}^{n}}+\abs{\zeta_{\bz}^{n}(\mh)}
+\frac{1}{\eps}\abs{\xi_{\dot{\bz}}^{n}}+\frac{1}{\eps}\abs{\zeta_{\dot{\bz}}^{n}(\mh)}
\end{equation*}
for $0\leq n\leq m$. Summing them up for $0\leq n\leq m$ and noting that $e_{\bz}^{0}=e_{\dot{\bz}}^{0}=0$ gives
\begin{equation*}
\abs{e_{\bz}^{m+1}}+\frac{1}{\eps}\abs{e_{\dot{\bz}}^{n+1}}\lesssim \mh\eps \sum\limits_{n=0}^{m}(\abs{e_{\bz}^{n}}+\abs{e_{\bz}^{n+1}}+\frac{1}{\eps}\abs{e_{\dot{\bz}}^{n}})+\sum\limits_{n=0}^{m}(\abs{\xi_{\bz}^{n}}+\frac{1}{\eps}\abs{\xi_{\dot{\bz}}^{n}}
+\abs{\zeta_{\bz}^{n}(\mh)}+\frac{1}{\eps}\abs{\zeta_{\dot{\bz}}^{n}(\mh)}).
\end{equation*}
From the truncation error estimates in \eqref{equ-10-5-1} and \eqref{equ-10-9-1}, along with the fact that  $m\eps\mh \lesssim 1$, one obtains
$
\abs{e_{\bz}^{m+1}}+\frac{1}{\eps}\abs{e_{\dot{\bz}}^{n+1}}\lesssim \mh\eps \sum\limits_{n=0}^{m}(\abs{e_{\bz}^{n}}+\abs{e_{\bz}^{n+1}}+\frac{1}{\eps}\abs{e_{\dot{\bz}}^{n}})+\mh^{2}.
$
It then follows from Gronwall's inequality that
$
\abs{e_{\bz}^{m+1}}+\frac{1}{\eps}\abs{e_{\dot{\bz}}^{n+1}}\lesssim \mh^{2}$ for $ 0\leq m <T/(\eps\mh).
$
Since 
\begin{equation*}
\abs{\bz^{m+1}}\leq \abs{\bz(\tau_{m+1})}+\abs{e_{\bz}^{m+1}}, \quad \abs{\dot{\bz}^{m+1}}\leq \abs{\dot{\bz}(\tau_{m+1})}+\abs{e_{\dot{\bz}}^{m+1}},
\end{equation*}
there exists a constant $\mh_{0}>0$ independent of $\eps$ and $m$, such that for any $0<\mh<\mh_{0}$, the estimate \eqref{solu-bound} holds for $m+1$. This completes the induction and establishes the convergence result.
\end{proof}

\subsection{Optimal error estimate}

To refine the error bounds to an optimal dependence of $\eps$, the following theorem is established.

\begin{theorem}\label{the-2}
Assume that $\bB(\cdot), \bE(\cdot)\in C^{1}(\mathbb{R}^d)$. For a fixed time $T>0$, let $\bz^{n}, \dot{\bz}^{n}$ be the numerical solution obtained by the RS2-PIC \eqref{st-rsv-resc} for solving \eqref{long-sys} up to $T/\eps$. Then there exists a constant $N_{0}>0$ independent of $\eps$, such that for any integer $N\geq N_{0}$ and the time step $\mh=\frac{T_{0}}{N}$, we have for some $m_0>0$ arbitrarily large
\begin{equation}\label{opt-err}
\abs{\bz^{n}-\bz(\tau_{n})}\lesssim \eps\mh^{2}+N^{-m_0}, \quad \abs{\dot{\bz}^{n}-\dot{\bz}(\tau_{n})}\lesssim \eps\mh^{2}+N^{-m_0}, \quad 
0\leq n \leq T/(\eps\mh).
\end{equation}
\end{theorem}

\begin{proof}
For any fixed $T>0$, the relation $ T/\eps =T_{0}M+\tau_{r}$ holds with $0\leq \tau_{r}<T_{0}$, where the integer
$M=\left\lfloor\frac{T}{\eps T_{0}}\right\rfloor=\mathcal{O}(1/\eps)$. In the subsequent analysis, $\tau_{r}$ can be set to zero without loss of generality.

\textbf{Update of notation.} With $N_0 > 0$ chosen to satisfy $\mh = T_0/N \leq \mh_0$ in Proposition \ref{pro-2}, the boundedness \eqref{solu-bound} is guaranteed for any $N \geq N_0$. To clarify the time scale, we denote by $\tau_n^m$ $(0 \leq n \leq N)$ the time grids in the $m$-th period, i.e.,
$\tau_{n}^{m}=mT_{0}+n\mh, 0\leq m<M$, then the numerical solutions obtained from the scheme \eqref{st-rsv-resc} at $\tau_{n}^{m}$ are denoted as
$
\bz_{n}^{m}\approx \bz(\tau_{n}^{m}), \ \dot{\bz}_{n}^{m}\approx \dot{\bz}(\tau_{n}^{m}), \  0\leq m<M, \ 0\leq n\leq N,
$
and the error as
$
e_{\bz}^{n,m}=\bz(\tau_{n}^{m})-\bz_{n}^{m}, \  e_{\dot{\bz}}^{n,m}=\dot{\bz}(\tau_{n}^{m})-\dot{\bz}_{n}^{m}.
$ 
Under the present notation, the identities $e_{\bz}^{0,m+1}=e_{\bz}^{N,m}$ and $e_{\dot{\bz}}^{0,m+1}=e_{\dot{\bz}}^{N,m}$ hold. Consequently, the error \eqref{equ-10-9-6} now reads as
\begin{subequations}\label{equ-10-9-9}
\begin{align}
e_{\bz}^{n+1,m}=&e_{\bz}^{n,m}+\mh\fe^{ \mh/2 \hbB(\bz(\tau_{n}^{m}))}e_{\dot{\bz}}^{n,m}+\eta_{\bz}^{n,m}+\xi_{\bz}^{n,m}+\zeta_{\bz}^{n,m}, \label{equ-10-9-9a} \\
e_{\dot{\bz}}^{n+1,m}=&\fe^{ \mh/2 \hbB(\tbz_{n}^{m}(\mh))}\fe^{ \mh/2 \hbB(\bz(\tau_{n}^{m}))}e_{\dot{\bz}}^{n,m}+\eta_{\dot{\bz}}^{n,m}+\xi_{\dot{\bz}}^{n,m}+\zeta_{\dot{\bz}}^{n,m},
\quad 0\leq n\leq N-1, \ 0\leq m< M. \label{equ-10-9-9b}
\end{align}
\end{subequations}
The notations for all remaining error terms are introduced in a consistent manner. Specifically, the local error at $\tau_n^m$ is denoted as
\begin{equation}\label{equ-10-9-10}
\xi_{\bz,1}^{n,m}=\int_{0}^{\mh}\fe^{s\hbB(\bz(\tau_{n}^{m}+ \mh/2 ))}ds \dot{\bz}(\tau_{n}^{m})-\mh\fe^{ \mh/2 \hbB(\bz(\tau_{n}^{m}))}\dot{\bz}(\tau_{n}^{m}).
\end{equation}

Following the argument in Proposition \ref{pro-2}, the error equation \eqref{equ-10-9-9} yields that
\begin{equation*}
\begin{aligned}
\frac{1}{\eps}\abs{e_{\bz}^{j,m}}-\frac{1}{\eps}\abs{e_{\bz}^{j-1,m}}\lesssim & \ \frac{\mh}{\eps}\abs{e_{\dot{\bz}}^{j-1,m}}+\eps\mh^{2}\left(\abs{e_{\bz}^{j-1,m}}+\abs{e_{\dot{\bz}}^{j-1,m}}\right)+\mh^{3}, \\
\frac{1}{\eps}\abs{e_{\dot{\bz}}^{j,m}}-\frac{1}{\eps}\abs{e_{\dot{\bz}}^{j-1,m}}\lesssim & \ \eps\mh\left(\abs{e_{\bz}^{j-1,m}}+\abs{e_{\bz}^{j,m}}+\abs{e_{\dot{\bz}}^{j-1,m}}\right)+\eps\mh^{3}, \quad 1\leq j\leq N, \ 0\leq m<M.
\end{aligned}
\end{equation*}
To achieve tighter control of the error in $\bz$, equation \eqref{equ-10-9-9} is scaled by $\eps$. Combining the above two inequalities, summing over $j=1,\ldots,n$ for any $1 \leq n \leq N$, and applying Gronwall's inequality yields the error estimate within each period
\begin{equation*}
\frac{1}{\eps}\abs{e_{\bz}^{n,m}}+\frac{1}{\eps}\abs{e_{\dot{\bz}}^{n,m}}\lesssim \mh^{2}+\frac{1}{\eps}\abs{e_{\bz}^{0,m}}+\frac{1}{\eps}\abs{e_{\dot{\bz}}^{0,m}}, \quad 1\leq n\leq N, \ 0\leq m<M.
\end{equation*}
In light of the fact that $\abs{e_{\dot{\bz}}^{0,m}}\lesssim \eps\mh^{2}$, it follows that
\begin{equation}\label{equ-10-9-11}
\abs{e_{\bz}^{n,m}}\lesssim \ \eps\mh^{2}+\abs{e_{\bz}^{0,m}},  \ \
\abs{e_{\dot{\bz}}^{n,m}}\lesssim \ \eps\mh^{2}+\abs{e_{\bz}^{0,m}}, \quad 1\leq n\leq N, \ 0\leq m<M.  
\end{equation}

\textbf{Refined local error.} The estimate for $\xi_{\bz,1}^{n,m}$ is now refined. It is directly observed that
\begin{equation}\label{equ-10-9-12}
\abs{\bB(z(\tau))-\bB_{0}}=\abs{\eps\bB_{1}(\bz(\tau))}\lesssim \eps, \quad 0\leq\tau\leq T/\eps.
\end{equation}
Rewrite the equation \eqref{long-sys} at $\tau=mT_{0}+s$ as
\begin{equation}\label{equ-10-9-13}
\ddot{\bz}(mT_{0}+s)=\dot{\bz}(mT_{0}+s)\times \bB_{0}+f^{m}(s), \quad 0\leq s\leq T_{0},
\end{equation}
with $f^{m}(s):=\dot{\bz}(mT_{0}+s)\times [\bB(\bz(mT_{0}+s))-\bB_{0}]+\eps^{2}\bE(\bz(mT_{0}+s))$. The fact \eqref{equ-10-9-12} implies that $\abs{f^{m}(s)}\lesssim \eps^{2}$. Applying Duhamel's principle to \eqref{equ-10-9-13} then gives
$
\dot{\bz}(mT_{0}+s)=\fe^{s\hbB_{0}}\dot{\bz}(mT_{0})+\int_{0}^{s}\fe^{(s-\rho)\hbB_{0}}f^{m}(\rho)d\rho,
$
which implies that
\begin{equation}\label{equ-10-9-14}
\abs{\dot{\bz}(mT_{0}+s)-\fe^{s\hbB_{0}}\dot{\bz}(mT_{0})}\leq Cs\eps^{2}, \quad 0\leq s\leq T_{0},
\end{equation}
for some constant $C>0$ independent of $\eps, s$ or $m$. In view of the estimates \eqref{equ-10-9-12} and \eqref{equ-10-9-14}, we decompose $\xi_{\bz,1}^{n,m}$ in \eqref{equ-10-9-10} into two parts:
$
\xi_{\bz,1}^{n,m}=\xi_{\bz,1,1}^{n,m}+\xi_{\bz,1,2}^{n,m}, 0\leq n<N,
$
where
\begin{equation*}
\xi_{\bz,1,1}^{n,m}=\int_{0}^{\mh}\fe^{s\hbB_{0}}ds \fe^{\tau_{n}^{m}\hbB_{0}}\dot{\bz}(mT_{0})-\mh\fe^{ \mh/2 \hbB_{0}}\fe^{\tau_{n}^{m}\hbB_{0}}\dot{\bz}(mT_{0})
\end{equation*}
and
\begin{equation*}
\begin{aligned}
\xi_{\bz,1,2}^{n,m}=&\int_{0}^{\mh}\left(\fe^{s\hbB(\bz(\tau_{n}^{m}))}-\fe^{s\hbB_{0}}\right)ds \dot{\bz}(\tau_{n}^{m})-\mh(\fe^{ \mh/2 \hbB(\bz(\tau_{n}^{m}))}-\fe^{ \mh/2 \hbB_{0}})\dot{\bz}(\tau_{n}^{m})  \\
&+\int_{0}^{\mh}\fe^{s\hbB_{0}}ds(\dot{\bz}(\tau_{n}^{m})-\fe^{\tau_{n}^{m}\hbB_{0}}\dot{\bz}(mT_{0}))-\mh\fe^{ \mh/2 \hbB_{0}}(\dot{\bz}(\tau_{n}^{m})-\fe^{\tau_{n}^{m}\hbB_{0}}\dot{\bz}(mT_{0})) \\
&+\int_{0}^{\mh}\left(\fe^{s\hbB(\bz(\tau_{n}^{m}+ \mh/2 ))}-\fe^{s\hbB(\bz(\tau_{n}^{m}))}\right)ds \dot{\bz}(\tau_{n}^{m}).
\end{aligned}
\end{equation*}
Attention now turns to the term $\xi_{\bz,1,2}^{n,m}$. The integral midpoint formula in conjunction with \eqref{equ-10-9-12} gives
\begin{equation*}
\abs{\int_{0}^{\mh}\left(\fe^{s\hbB(\bz(\tau_{n}^{m}))}-\fe^{s\hbB_{0}}\right)ds -\mh(\fe^{ \mh/2 \hbB(\bz(\tau_{n}^{m}))}-\fe^{ \mh/2 \hbB_{0}})}
\lesssim \eps\mh^{3}.
\end{equation*}
Noting that
$
\abs{\fe^{s\hbB(\bz(\tau_{n}^{m}+ \mh/2 ))}-\fe^{s\hbB(\bz(\tau_{n}^{m}))}}\lesssim \eps^{2}sh
$
and regarding the last two terms in $\xi_{\bz,1,2}^{n,m}$, it can be observed that
\begin{equation*}
\begin{aligned}
&\int_{0}^{\mh}\fe^{s\hbB_{0}}ds(\dot{\bz}(\tau_{n}^{m})-\fe^{\tau_{n}^{m}\hbB_{0}}\dot{\bz}(mT_{0}))
-\mh\fe^{ \mh/2 \hbB_{0}}(\dot{\bz}(\tau_{n}^{m})-\fe^{\tau_{n}^{m}\hbB_{0}}\dot{\bz}(mT_{0}))  \\
=&\left(\int_{0}^{\mh}\fe^{s\hbB_{0}}ds-\mh\fe^{ \mh/2 \hbB_{0}}\right)(\dot{\bz}(\tau_{n}^{m})-\fe^{\tau_{n}^{m}\hbB_{0}}\dot{\bz}(mT_{0})),
\end{aligned}
\end{equation*}
where $\abs{\int_{0}^{\mh}\fe^{s\hbB_{0}}ds-\mh\fe^{ \mh/2 \hbB_{0}}}\lesssim \mh^{3}$ is estimated again by the integral midpoint formula.
Additionally, invoking periodicity and \eqref{equ-10-9-14} leads to
\begin{equation*}
\dot{\bz}(\tau_{n}^{m})-\fe^{\tau_{n}^{m}\hbB_{0}}\dot{\bz}(mT_{0})=\dot{\bz}(mT_{0}+n\mh)-\fe^{n\mh\hbB_{0}}\dot{\bz}(mT_{0})=\mathcal{O}(\eps^{2}).
\end{equation*}
Therefore, combining the above estimates yields
$
\abs{\xi_{\bz,1,2}^{n,m}}\lesssim \eps^{2}\mh^{3}.
$
For $\xi_{\bz,1,1}^{n,m}$, summing over $n=0,\ldots,N-1$ gives
\begin{equation*}
\chi^{m}:=\sum\limits_{n=0}^{N-1}\xi_{\bz,1,1}^{n,m}=\int_{0}^{T_{0}}\fe^{s\hbB_{0}}ds \dot{\bz}(mT_{0})-\mh\sum\limits_{n=0}^{N-1}\fe^{(n+\frac{1}{2})\mh\hbB_{0}}\dot{\bz}(mT_{0}), \quad 0\leq m<M.
\end{equation*}
Note that $\chi^{m}$ corresponds exactly to the quadrature error of the midpoint rule applied to the smooth periodic function $\fe^{s\hbB_{0}}$ over a single period, and so
$
\abs{\chi^{m}}\lesssim \eps N^{-m_{0}}, 0\leq m<M,
$
for some $m_{0}>0$ arbitrarily large. Consequently, it follows that
$
\abs{\sum\limits_{n=0}^{N-1}\xi_{\bz,1}^{n,m}}\leq \abs{\chi^{m}}+\abs{\sum\limits_{n=0}^{N-1}\xi_{\bz,1,2}^{n,m}}\lesssim \eps^{2}\mh^{2}+\eps N^{-m_{0}}.
$

\textbf{Refined error equation.} We now present a clearer characterization of the error propagation over each period. For some $0\leq m<M$, summing equation \eqref{equ-10-9-9a} over $n=0,\ldots,N-1$ produces
\begin{equation*}
e_{\bz}^{N,m}=e_{\bz}^{0,m}+\mh\sum\limits_{n=0}^{N-1}\fe^{ \mh/2 \hbB(\bz(\tau_{n}^{m}))}e_{\dot{\bz}}^{n,m}+\sum\limits_{n=0}^{N-1}(\eta_{\bz}^{n,m}+\zeta_{\bz}^{n,m}(\mh))
+\sum\limits_{n=0}^{N-1}(\xi_{\bz,1}^{n,m}+\xi_{\bz,2}^{n,m}).
\end{equation*}
Then, an application of \eqref{equ-10-9-12} shows that 
\begin{equation}\label{equ-10-10-1}
e_{\bz}^{N,m}=e_{\bz}^{0,m}+\mh\sum\limits_{n=0}^{N-1}\fe^{ \mh/2 \hbB_{0}}e_{\dot{\bz}}^{n,m}+\sum\limits_{n=0}^{N-1}(\eta_{\bz}^{n,m}+\zeta_{\bz}^{n,m}(\mh))
+\sum\limits_{n=0}^{N-1}(\xi_{\bz,1}^{n,m}+\xi_{\bz,2}^{n,m})+\delta_{\bz}^{n,m},
\end{equation}
where $\abs{\delta_{\bz}^{n,m}}\lesssim \eps^{2}\mh^{3},\ 0\leq m<M$ owing to $e_{\dot{\bz}}^{n,m}=\mathcal{O}(\eps\mh^{2})$ from Proposition \ref{pro-2}. On the other hand, equation \eqref{equ-10-9-9b} can be rewritten using \eqref{equ-10-9-12} as
\begin{equation}\label{equ-10-10-2}
e_{\dot{\bz}}^{n+1,m}=\fe^{\mh\hbB_{0}}e_{\dot{\bz}}^{n,m}+\eta_{\dot{\bz}}^{n,m}+\xi_{\dot{\bz}}^{n,m}+\zeta_{\dot{\bz}}^{n,m}+\delta_{\dot{\bz}}^{n,m},
\quad 1\leq n\leq N, \ 0\leq m<M,
\end{equation}
where
$
\delta_{\dot{\bz}}^{n,m}=\left(\fe^{ \mh/2 \hbB(\tbz_{n}^{m}(\mh))}\fe^{ \mh/2 \hbB(\bz(\tau_{n}^{m}))}-\fe^{\mh\hbB_{0}}\right)e_{\dot{\bz}}^{n,m},
$
and it follows that
\begin{equation}\label{equ-10-10-4}
\abs{\delta_{\dot{\bz}}^{n,m}}\lesssim \eps^{2}\mh^{3}.
\end{equation}
A recursive derivation from \eqref{equ-10-10-2} gives, for any $1\leq n\leq N,\ 0\leq m<M$,
\begin{equation*}
e_{\dot{\bz}}^{n,m}=\fe^{n\mh\hbB_{0}}e_{\dot{\bz}}^{0,m}+\sum\limits_{j=0}^{n-1}\fe^{(n-1-j)\mh\hbB_{0}}\left[\eta_{\dot{\bz}}^{j,m}+\xi_{\dot{\bz}}^{j,m}
+\zeta_{\dot{\bz}}^{j,m}(\mh)+\delta_{\dot{\bz}}^{j,m}\right],
\end{equation*} 
and thus
\begin{equation*}
\mh\sum\limits_{n=0}^{N-1}\fe^{ \mh/2 \hbB_{0}}e_{\dot{\bz}}^{n,m}=\mh\sum\limits_{n=0}^{N-1}\fe^{(n+\frac{1}{2})\mh\hbB_{0}}e_{\dot{\bz}}^{0,m}
+\mh\sum_{n=0}^{N-1}\sum_{j=0}^{n-1}\fe^{(n-\frac{1}{2}-j)\mh\hbB_{0}}\left[\eta_{\dot{\bz}}^{j,m}+\xi_{\dot{\bz}}^{j,m}
+\zeta_{\dot{\bz}}^{j,m}(\mh)+\delta_{\dot{\bz}}^{j,m}\right].
\end{equation*}
With this result, equation \eqref{equ-10-10-1} can be rewritten as
\begin{equation}\label{equ-10-10-3}
e_{\bz}^{N,m}=e_{\bz}^{0,m}+\mh\sum\limits_{n=0}^{N-1}\fe^{(n+\frac{1}{2})\mh\hbB_{0}}e_{\dot{\bz}}^{0,m}+\mu^{m}, \quad 0\leq m<M,
\end{equation}
where
\begin{equation*}
\mu^{m}=\sum\limits_{n=0}^{N-1}(\eta_{\bz}^{n,m}+\xi_{\bz}^{n,m}+\zeta_{\bz}^{n,m}(\mh))+\delta_{\bz}^{n,m}+\mh\sum_{n=0}^{N-1}\sum_{j=0}^{n-1}\fe^{(n-\frac{1}{2}-j)\mh\hbB_{0}}\left[\eta_{\dot{\bz}}^{j,m}+\xi_{\dot{\bz}}^{j,m}
+\zeta_{\dot{\bz}}^{j,m}(\mh)+\delta_{\dot{\bz}}^{j,m}\right].
\end{equation*}
From \eqref{equ-10-5-1}, \eqref{equ-10-9-1}, \eqref{equ-10-9-7} and \eqref{equ-10-10-4}, the last term above satisfies
\begin{equation*}
\abs{\mh\sum_{n=0}^{N-1}\sum_{j=0}^{n-1}\fe^{(n-\frac{1}{2}-j)\mh\hbB_{0}}\left[\eta_{\dot{\bz}}^{j,m}+\xi_{\dot{\bz}}^{j,m}
+\zeta_{\dot{\bz}}^{j,m}(\mh)+\delta_{\dot{\bz}}^{j,m}\right]}\lesssim \eps^{2}\mh^{2}+\eps^{2}\mh\sum\limits_{n=0}^{N-1}(\abs{e_{\bz}^{n,m}}+\abs{e_{\bz}^{n+1,m}}+\abs{e_{\dot{\bz}}^{n,m}}),
\end{equation*}
and it can be inferred that $$
\abs{\mu^{m}}\lesssim \eps^{2}\mh^{2}+\eps N^{-m_{0}}+\eps^{2}\mh \sum\limits_{n=0}^{N-1}(\abs{e_{\bz}^{n,m}}+\abs{e_{\bz}^{n+1,m}}+\abs{e_{\dot{\bz}}^{n,m}}),\  0\leq m<M.$$ 
Reapplying the quadrature error estimate of the midpoint rule to \eqref{equ-10-10-3} leads to
\begin{equation}\label{equ-10-10-5}
\abs{e_{\bz}^{N,m}}-\abs{e_{\bz}^{0,m}}\lesssim \abs{\int_{0}^{T_{0}}\fe^{s\hbB_{0}}ds e_{\dot{\bz}}^{0,m}}+\eps^{2}\mh^{2}+\eps N^{-m_{0}}+\eps^{2}\mh \sum\limits_{n=0}^{N-1}(\abs{e_{\bz}^{n,m}}+\abs{e_{\bz}^{n+1,m}}+\abs{e_{\dot{\bz}}^{n,m}}), \ 0\leq m<M.
\end{equation}
By the Rodrigue's formula,
\begin{equation*}
\fe^{s\hbB_{0}}e^{0,m}_{\dot{\bz}}=\cos(s\abs{\bB_{0}})e^{0,m}_{\dot{\bz}}+\sin(s\abs{\bB_{0}})e^{0,m}_{\dot{\bz}}\times \tilde{\bB}_{0}+
(1-\cos(s\abs{\bB_{0}}))(\tilde{\bB}_{0}\cdot e^{0,m}_{\dot{\bz}})\tilde{\bB}_{0},
\end{equation*}
where $\tilde{\bB}_{0}$ denotes the normalized magnetic field vector, i.e., $\tilde{\bB}_{0}=\bB_{0}/\abs{\bB_{0}}$. Integrating the above term over one period gives
$
\int_{0}^{T_{0}}\fe^{s\hbB_{0}}ds e_{\dot{\bz}}^{0,m}=T_{0}(\tilde{\bB}_{0}\cdot e_{\dot{\bz}}^{0,m})\tilde{\bB}_{0}.
$
Thus, \eqref{equ-10-10-5} indicates that
\begin{equation*}
\begin{aligned}
\abs{e_{\bz}^{N,m}}-\abs{e_{\bz}^{0,m}}\lesssim & \abs{(\tilde{\bB}_{0}\cdot e_{\dot{\bz}}^{0,m})\tilde{\bB}_{0}}+\eps^{2}\mh^{2}+\eps N^{-m_{0}}+\eps^{2}\mh
\sum\limits_{n=0}^{N-1}(\abs{e_{\bz}^{n,m}}+\abs{e_{\bz}^{n+1,m}}+\abs{e_{\dot{\bz}}^{n,m}}) \\
\lesssim & \abs{e_{\dot{\bz},\|}^{0,m}}+\eps^{2}\mh^{2}+\eps N^{-m_{0}}+\eps^{2}\mh
\sum\limits_{n=0}^{N-1}(\abs{e_{\bz}^{n,m}}+\abs{e_{\bz}^{n+1,m}}+\abs{e_{\dot{\bz}}^{n,m}}), \quad 0\leq m<M,
\end{aligned}
\end{equation*}
where $e_{\dot{\bz},\|}^{0,m}$ denotes the component of the error $e_{\dot{\bz}}^{0,m}$ parallel to the magnetic field $\bB(\bz(mT_{0}))$, i.e.,
\begin{equation*}
e_{\dot{\bz},\|}^{n,m}:=(\tilde{\bB}^{n,m}\cdot e_{\dot{\bz}}^{n,m})\tilde{\bB}^{n,m}, \quad \tilde{\bB}^{n,m}:=\frac{\bB(\bz(\tau_{n}^{m}))}{\bB(\bz(\tau_{n}^{m}))}, \quad 0\leq n\leq N,\ 0\leq m <M.
\end{equation*}
Then by \eqref{equ-10-9-11} and noting $e_{\bz}^{N,m}=e_{\bz}^{0,m+1}$, we arrive that
\begin{equation}\label{equ-10-10-6}
\abs{e_{\bz}^{0,m+1}}-\abs{e_{\bz}^{0,m}}\lesssim \abs{e_{\dot{\bz},\|}^{0,m}}+\eps^{2}(\abs{e_{\bz}^{0,m}}+\abs{e_{\bz}^{0,m+1}})+\eps^{2}\mh^{2}+\eps N^{-m_{0}}, \quad 0\leq m<M.
\end{equation}
Taking the inner product of \eqref{equ-10-9-9b} with the unit vector $\tilde{\bB}^{n+1,m}$ yields
\begin{equation}\label{equ-10-10-7}
\abs{e_{\dot{\bz},\|}^{n+1,m}}\leq \abs{\tilde{\bB}^{n+1,m}\cdot \left(\fe^{ \mh/2 \hbB(\tbz_{n}^{m}(\mh))}\fe^{ \mh/2 \hbB(\bz(\tau_{n}^{m}))}e_{\dot{\bz}}^{n,m} \right)}+\abs{\eta_{\dot{\bz}}^{n,m}+\zeta_{\dot{\bz}}^{n,m}(\mh)}+\abs{\xi_{\dot{\bz}}^{n,m}\cdot \tilde{\bB}^{n+1,m}}.
\end{equation}
From the relation
$
\tilde{\bB}^{n+1,m}=\tilde{\bB}^{n,m}+\mathcal{O}(\eps^{2}\mh)
$
and Rodrigue's formula, one obtains
\begin{equation*}
\abs{\tilde{\bB}^{n+1,m}\cdot \left(\fe^{ \mh/2 \hbB(\tbz_{n}^{m}(\mh))}\fe^{ \mh/2 \hbB(\bz(\tau_{n}^{m}))}e_{\dot{\bz}}^{n,m} \right)}\lesssim
\abs{e_{\dot{\bz},\|}^{n,m}}+\eps^{3}\mh^{3}.
\end{equation*}
Together with \eqref{equ-10-5-1} and \eqref{equ-10-9-7}, it then follows from \eqref{equ-10-10-7} that for $0\leq n\leq N,\ 0\leq m<M$,
\begin{equation}\label{equ-10-10-8}
\abs{e_{\dot{\bz},\|}^{n+1,m}}-\abs{e_{\dot{\bz},\|}^{n,m}}\lesssim \eps^{2}\mh(\abs{e_{\bz}^{n,m}}+\abs{e_{\bz}^{n+1,m}})+\abs{\xi_{\dot{\bz}}^{n,m}\cdot \tilde{\bB}^{n+1,m}}+\eps^{3}\mh^{3}.
\end{equation}
Recalling \eqref{equ-10-7-3} and applying Rodrigue's formula shows simply that $\abs{\xi_{\dot{\bz}}^{n,m}\cdot \tilde{\bB}^{n,m}}\lesssim \eps^{3}\mh^{3}$.
Hence, \eqref{equ-10-10-8} leads to
\begin{equation}\label{equ-10-10-9}
\abs{e_{\dot{\bz},\|}^{n+1,m}}-\abs{e_{\dot{\bz},\|}^{n,m}}\lesssim \eps^{2}\mh \sum\limits_{n=0}^{N-1}(\abs{e_{\bz}^{n,m}}+\abs{e_{\bz}^{n+1,m}})+\eps^{3}\mh^{3}, \quad 0\leq n<N, \ 0\leq m<M.
\end{equation}
Summing up \eqref{equ-10-10-9} for $n=0,\ldots,N-1$ gives
$
\abs{e_{\dot{\bz},\|}^{0,m+1}}-\abs{e_{\dot{\bz},\|}^{0,m}}\lesssim \eps^{2}\mh \sum\limits_{n=0}^{N-1}(\abs{e_{\bz}^{n,m}}+\abs{e_{\bz}^{n+1,m}})+\eps^{3}\mh^{3}.
$
Inserting \eqref{equ-10-9-11} into the preceding inequality and dividing both sides by $1/\eps$, we get
\begin{equation}\label{equ-10-10-10}
\frac{1}{\eps}\abs{e_{\dot{\bz},\|}^{0,m+1}}-\frac{1}{\eps}\abs{e_{\dot{\bz},\|}^{0,m}}\lesssim \eps\abs{e_{\bz}^{0,m}}+\eps^{2}\mh^{2}, \quad 0\leq m<M.
\end{equation}
Finally, combining \eqref{equ-10-10-6} and \eqref{equ-10-10-10} yields
\begin{equation*}
\abs{e_{\bz}^{0,m+1}}+\frac{1}{\eps}\abs{e_{\dot{\bz},\|}^{0,m+1}}-\abs{e_{\bz}^{0,m}}-\frac{1}{\eps}\abs{e_{\dot{\bz},\|}^{0,m}}\lesssim
\eps(\abs{e_{\bz}^{0,m}}+\abs{e_{\bz}^{0,m+1}}+\frac{1}{\eps}\abs{e_{\dot{\bz}}^{0,m}})+\eps^{2}\mh^{2}+\eps N^{-m_{0}}, \ \ 0\leq m<M. 
\end{equation*}
By invoking Gronwall's inequality and noting the initial conditions $e_{\bz}^{0,0}=e_{\dot{\bz},\|}^{0,0}=0$, it follows that
\begin{equation*}
\abs{e_{\bz}^{0,m}}+\frac{1}{\eps}\abs{e_{\dot{\bz},\|}^{0,m}}\lesssim \eps\mh^{2}+N^{-m_{0}}, \quad 0\leq m \leq M.
\end{equation*}
The estimates for intermediate time grids including $e_{\bz}^{n,m}$ and $e_{\dot{\bz},\|}^{n,m}$ with $0<n<N$ follow directly from \eqref{equ-10-10-9} and \eqref{equ-10-9-11}, thereby yielding
\begin{equation*}
\abs{e_{\bz}^{n,m}}\lesssim \eps\mh^{2}+N^{-m_{0}}, \ \abs{e_{\dot{\bz},\|}^{n,m}}\lesssim \eps^{2}\mh^{2}+\eps N^{-m_{0}}, \quad 0\leq n<N,\ \ 0\leq m<M.
\end{equation*}
From Proposition \ref{pro-2} and \eqref{equ-10-9-11}, it can be derived that
$
\abs{e_{\dot{\bz}}^{n,m}}\lesssim \eps\mh^{2}+N^{-m_{0}}
$
, which concludes the proof.
\end{proof}
\begin{remark}
The relations $\dot{\bx}(t)=\frac{1}{\eps}\dot{\bz}(\tau)$ and $\dot{\bx}^{n}=\frac{1}{\eps}\dot{\bz}^{n}$, derived from $h=\eps\mh$ and \eqref{equ-10-4-3}, immediately yield Theorem \ref{theo-2}. Although the RS1-PIC and RS2-PIC frameworks can be extended to higher-order schemes, the corresponding error analysis is non-trivial. We therefore defer the development of such schemes to future work.
\end{remark}

\section{Applications to Vlasov-Poisson system under various scalings}\label{sec-4}

In this section, we apply the ER-PIC methods to Vlasov-Poisson system under various scalings and evaluate the accuracy and energy conservation of the proposed  two ER-PIC methods: RS1-PIC and RS2-PIC.

\subsection{Vlasov-Poisson equation in fluid scaling}
The two-dimensional Vlasov-Poisson system \eqref{vp} in fluid scaling under a strong magnetic field \eqref{mag-1} reduces to the following characteristic equation via the PIC method  
\begin{equation}\label{vp-2}
\dot{\bx}(t)=\bv(t), \ \dot{\bv}(t)=\dfrac{b(\bx)}{\eps}\bv^{\perp}(t)+\bE_{[\bx^{p}(t)]}(\bx(t)), \
\bx(0)=\bx_{0}, \ \bv(0)=\bv_{0}, \ \bx=(x_{1},x_{2})^{\intercal}, \ t>0,
\end{equation}
where $\bv^{\perp}=J\bv=(v_{2},-v_{1})^{\intercal}$ with $J=\begin{pmatrix}  0 & 1 \\ -1 & 0 \end{pmatrix}$, and $b(\bx)=b_{0}+\eps b_{1}(\bx)\geq 0$ with $b_{0}$ is a constant. For the PIC method, we project the particles onto a uniform spatial grid using quintic B-splines. The reference solution is obtained by using a fourth order Runge-Kutta method with step size $h=10^{-6}$.
\begin{figure}[t!]
$$\begin{array}{cc}
\psfig{figure=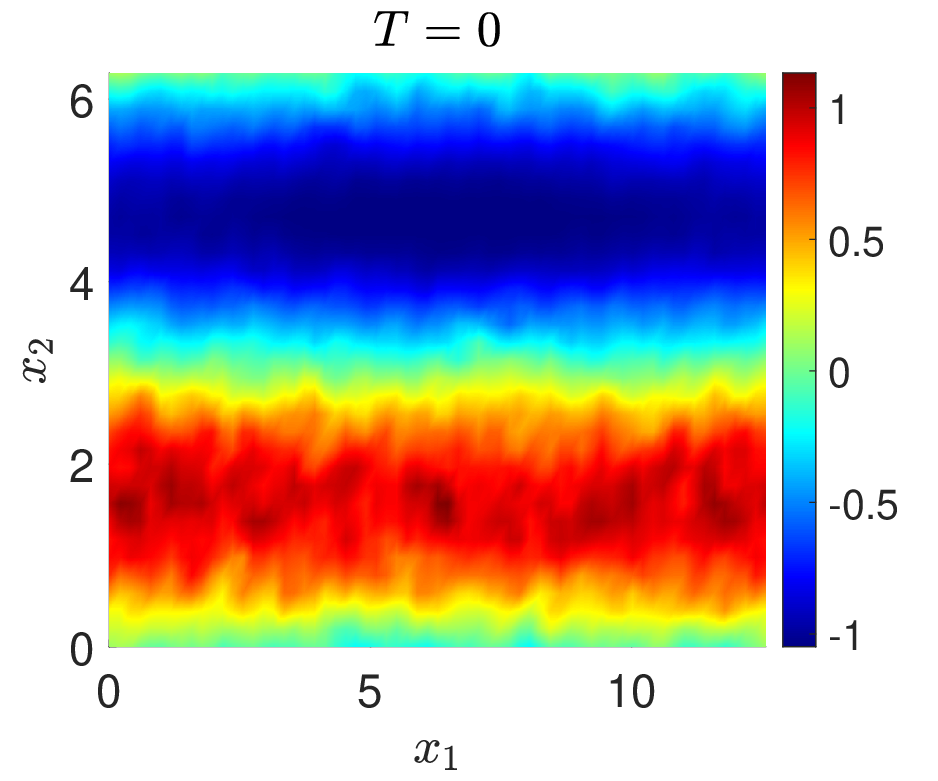,height=3.8cm,width=4.8cm}
\psfig{figure=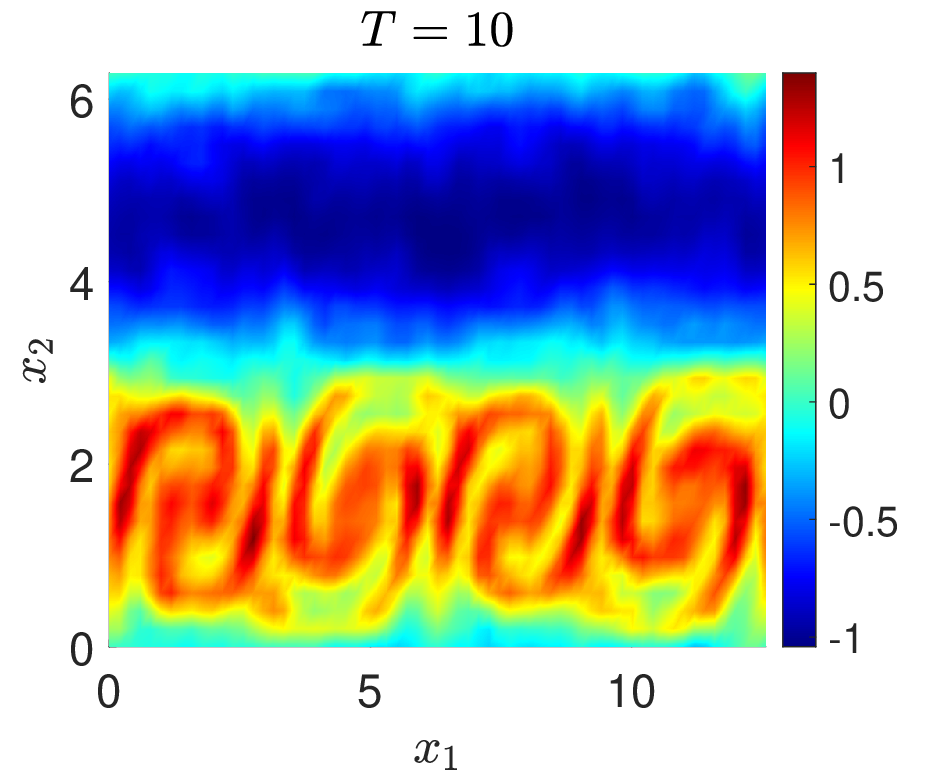,height=3.8cm,width=4.8cm}
\psfig{figure=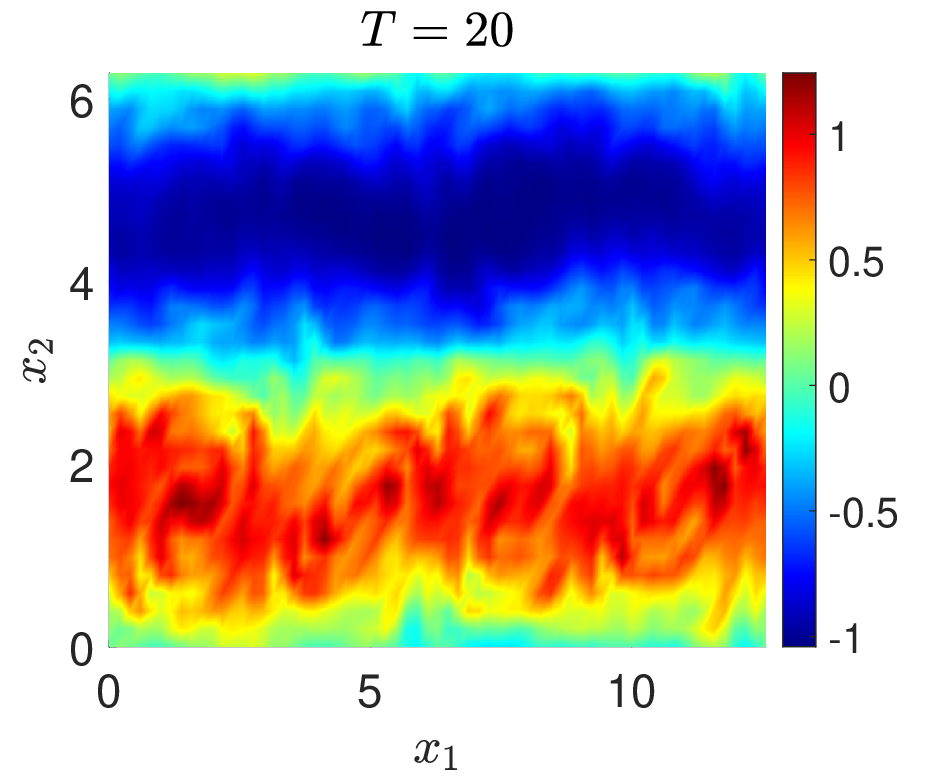,height=3.8cm,width=4.8cm}
\end{array}$$
\caption{Example 1. Contour plot of quantity $\rho(t,\bx)-n_{i}$ at different $t$ with $\eps=0.005$ for RS2-PIC.}\label{fig-1-5}
\end{figure}

\begin{figure}[t!]
$$\begin{array}{cc}
\psfig{figure=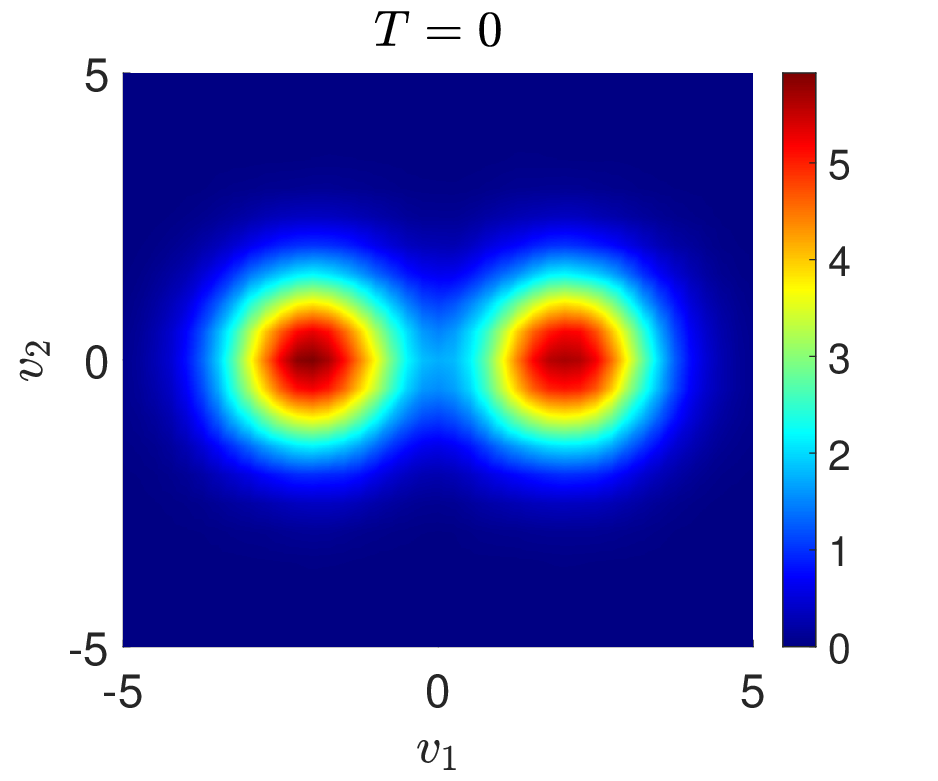,height=3.8cm,width=4.8cm}
\psfig{figure=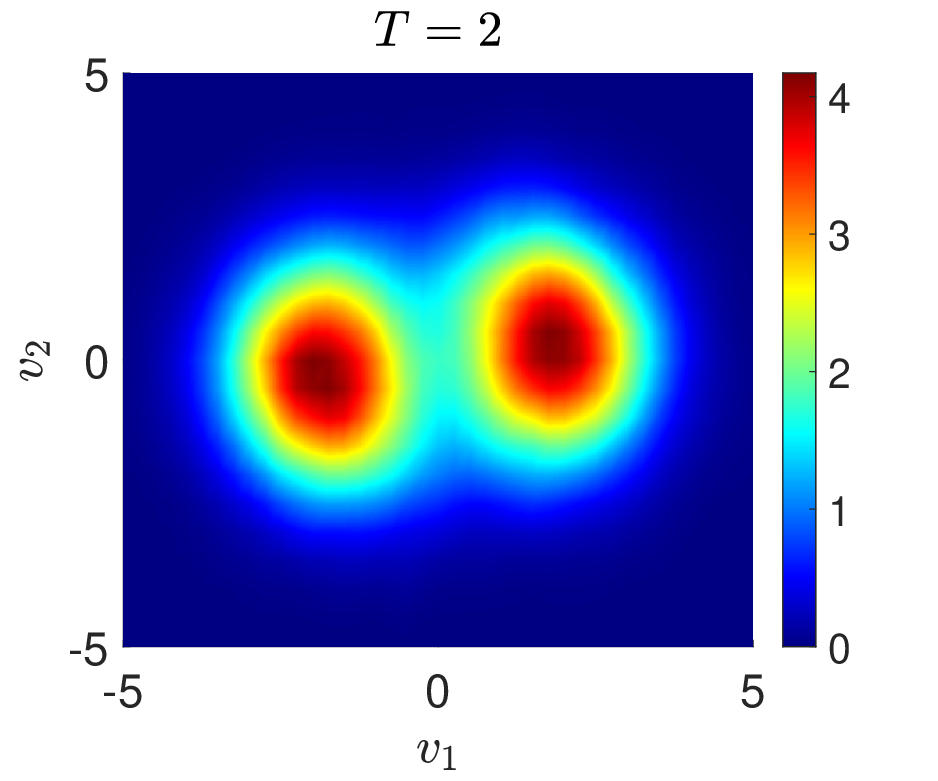,height=3.8cm,width=4.8cm}
\psfig{figure=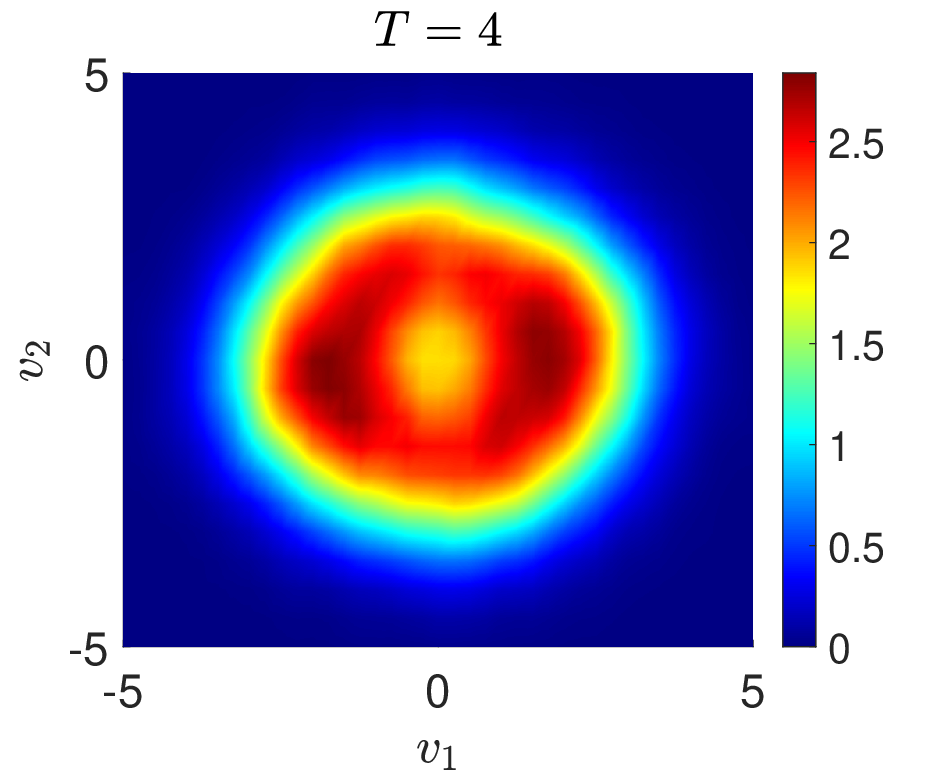,height=3.8cm,width=4.8cm} \\
\psfig{figure=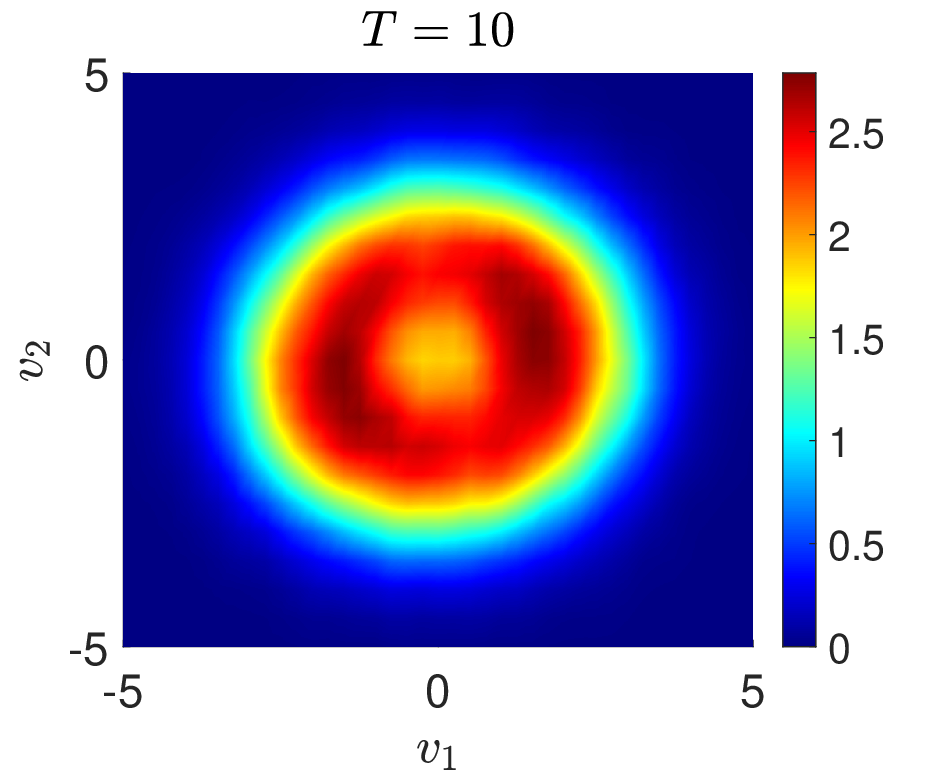,height=3.8cm,width=4.8cm}
\psfig{figure=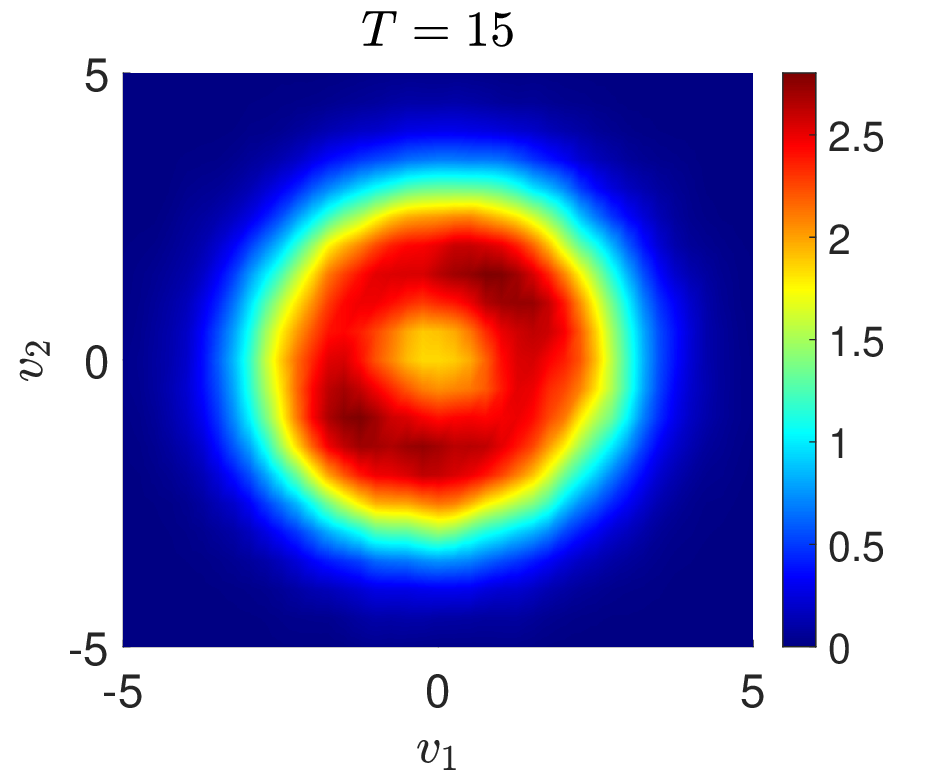,height=3.8cm,width=4.8cm}
\psfig{figure=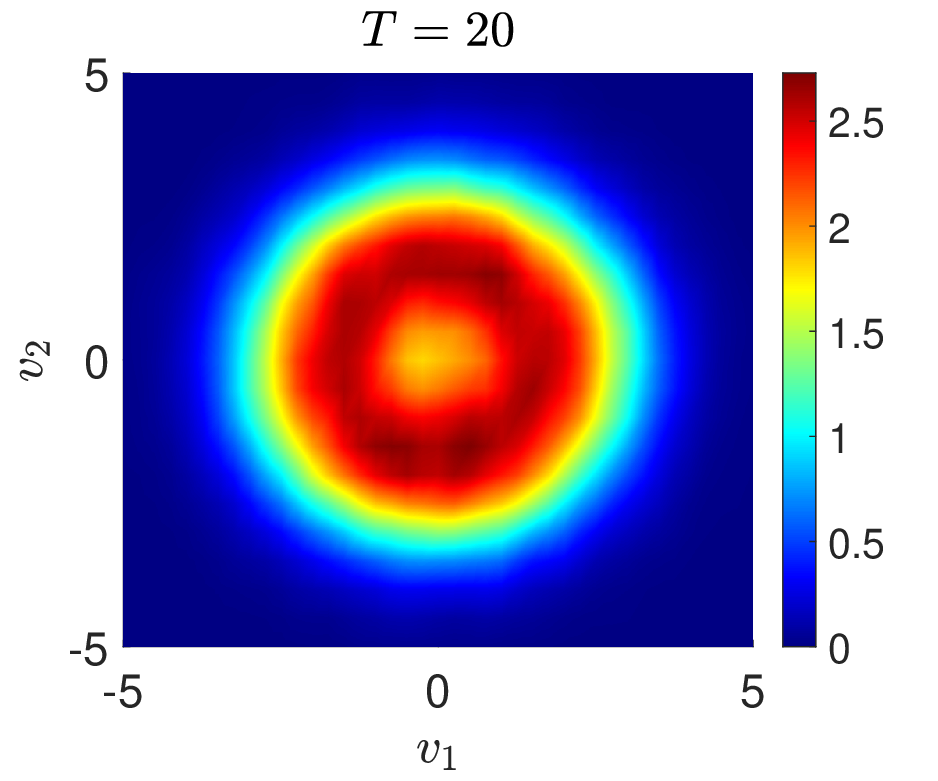,height=3.8cm,width=4.8cm}
\end{array}$$
\caption{Example 1. Contour plot of quantity $\chi(t,\bv)$ at different $t$ with $\eps=0.005$ for RS2-PIC.}\label{fig-1-6}
\end{figure}
\textbf{Example 1.}
We first consider the Vlasov-Poisson system \eqref{vp} with the following initial data
\begin{equation}\label{ini-1}
f_{0}(\bx,\bv)=\frac{1}{4\pi}(1+\sin(x_{2})+\eta\cos(kx_{1}))\left(\fe^{-\frac{(v_{1}+2)^{2}+v_{2}^{2}}{2}}+\fe^{-\frac{(v_{1}-2)^{2}+v_{2}^{2}}{2}}\right),
\end{equation}
and the strong magnetic field \eqref{equ-9-28-2} is given as
\begin{equation}\label{mag-1}
\bB(\bx)=(0,0,1+\eps b_{1}(\bx))^{\intercal}
\end{equation}
with the non-homogeneous field $b_{1}(\bx)=1+\sin(x_{1})\sin(x_{2})/2$.  
The spatial domain is given by $\bx\in\Omega=[0,2\pi/k]\times[0,2\pi]$ for parameters $k,\eta>0$, with $\eta=0.05$ and $k=0.5$ chosen in our simulations. The domain $\Omega$ is discretized using 64 points in the $x_1$-direction and 32 points in the $x_2$-direction, and the total number of particles is 102400. To test the accuracy, we consider the following two quantities
\begin{equation*}
\rho(t,\bx)=\int_{\mathbb{R}^{2}} f(t,\bx,\bv)d\bv, \quad \rho_{\bv}(t,\bx)=\int_{\mathbb{R}^{2}} \abs{\bv}^{2}f(t,\bx,\bv)d\bv, \quad \bx\in\Omega.
\end{equation*}

We first study the space profiles $\rho(t,\bx)-n_{i}$ and velocity profiles $\chi(t,\bv)=\int_{\Omega}f(t,\bx,\bv)d\bx$ with $\eps=0.005$ at different times. The profiles of $\rho(t,\bx)-n_{i}$ is plotted as a function of $\bx$ at different times with $h=0.1$ in Figure \ref{fig-1-5}. The dynamics of $\chi(t,\bv)$ is shown in Figure \ref{fig-1-6}. The velocity distribution $\chi$ evolves from an initial rotation of its two bumps to their subsequent merger around $t \approx 4$, ultimately reaching an isotropic state.

In Tables \ref{tab-1-1}-\ref{tab-1-2} and Figure \ref{fig-1-1}, we   compute the relative errors in time of the RS1-PIC and RS2-PIC regarding both $\rho$ and $\rho_{\bv}$ in maximum norm, i.e., $err_{\rho}+err_{\rho_{\bv}}=\frac{\norm{\rho^n(\cdot)-\rho(t_{n},\cdot)}_{l^{\infty}}}{\norm{\rho(t_{n},\cdot)}_{l^{\infty}}}+\frac{\norm{\rho_{\bv}^n(\cdot)-\rho_{\bv}(t_{n},\cdot)}_{l^{\infty}}}
{\norm{\rho_{\bv}(t_{n},\cdot)}_{l^{\infty}}}$. Numerical results indicate that for the $\rho(t,\bx)$ and $\rho_{\bv}(t,\bx)$, RS1-PIC admits a uniform first-order error bound that is independent of the $\eps$, whereas RS2-PIC scheme achieves second-order accuracy $ \mathcal{O}(h^2/\eps)$. These observations support the theoretical results of Theorem \ref{theo-2}. 

 Then, we evaluate the relative error of the total energy $
err_{H}=\abs{H(t_{n})-H(0)}/H(0)$ 
of the two methods. The energy error of the RS1-PIC and RS2-PIC schemes are presented until $T=100$ with step size $h=0.1$ and different $\eps$ in Figure \ref{fig-1-3}. It can be observed that both the RS1-PIC and RS2-PIC schemes can precisely conserve energy well. Although at a small number of steps, the energy exhibits long-term approximate conservation at the level of $10^{-6}\sim 10^{-5}$. This is a consequence of the occurrence of $\mathcal{C}^{2}-2\mathcal{A}\widetilde{H}<0$, which results directly from our choice of $\gamma=0$ in \eqref{equ-10-1-6}.  In addition, we also draw the evolutions of relaxation coefficients $\gamma_{n}$ in Figure \ref{fig-1-7} with step size $h=0.02$,
which demonstrates that the order of $\gamma_{n}$ is independent of $\eps$.

\begin{table}[htbp]
\centering
\caption{Example 1. Time error of RS1-PIC for different $\eps$ and $h$.}
\begin{tabular}{ccccccc}
\toprule
$err_{\rho}+err_{\rho_{\bv}} (t=1)$      &  $h_{0}=1/16$ &   $h_{0}/2$    &  $h_{0}/2^{2}$  &  $h_{0}/2^{3}$ &  $h_{0}/2^{4}$ &  $h_{0}/2^{5}$ \\
\midrule
$\eps_{0}=1$      &   1.04E-1  &  5.62E-2  &  2.91E-2  &  1.52E-2  &  8.21E-3  &  4.86E-3   \\
order             &  -         &   0.88    &   0.95    &   0.94    &   0.89    &   0.76      \\
$\eps_{0}/2^{2}$  &   6.91E-2  &  3.60E-2  &  1.82E-2  &  9.09E-3  &  4.56E-3  &  2.33E-3   \\
order             &  -         &   0.94    &   0.99    &   1.00    &   1.00    &   0.97       \\
$\eps_{0}/2^{4}$  &  7.35E-2   &  3.97E-2  &  2.03E-2  &  1.02E-2  &  5.13E-3  &  2.57E-3  \\
order             &  -         &   0.89    &   0.97    &   0.99    &   1.00    &   1.00        \\
$\eps_{0}/2^{6}$  &   9.72E-2  &  4.36E-2  &  2.12E-2  &  1.05E-2  &  5.25E-3  &  2.62E-3     \\
order             &  -         &   1.16    &   1.04    &   1.01    &   1.00    &    1.00        \\
$\eps_{0}/2^{8}$  &   3.33E-2  &  1.91E-2  &  1.13E-2  &  4.22E-3  &  2.00E-3  &  9.86E-4    \\
order             &  -         &   0.81    &   0.76    &   1.42    &   1.08    &   1.02         \\
\bottomrule
\end{tabular}
\label{tab-1-1}
\end{table}

\begin{figure}[t!]
$$\begin{array}{cc}
\psfig{figure=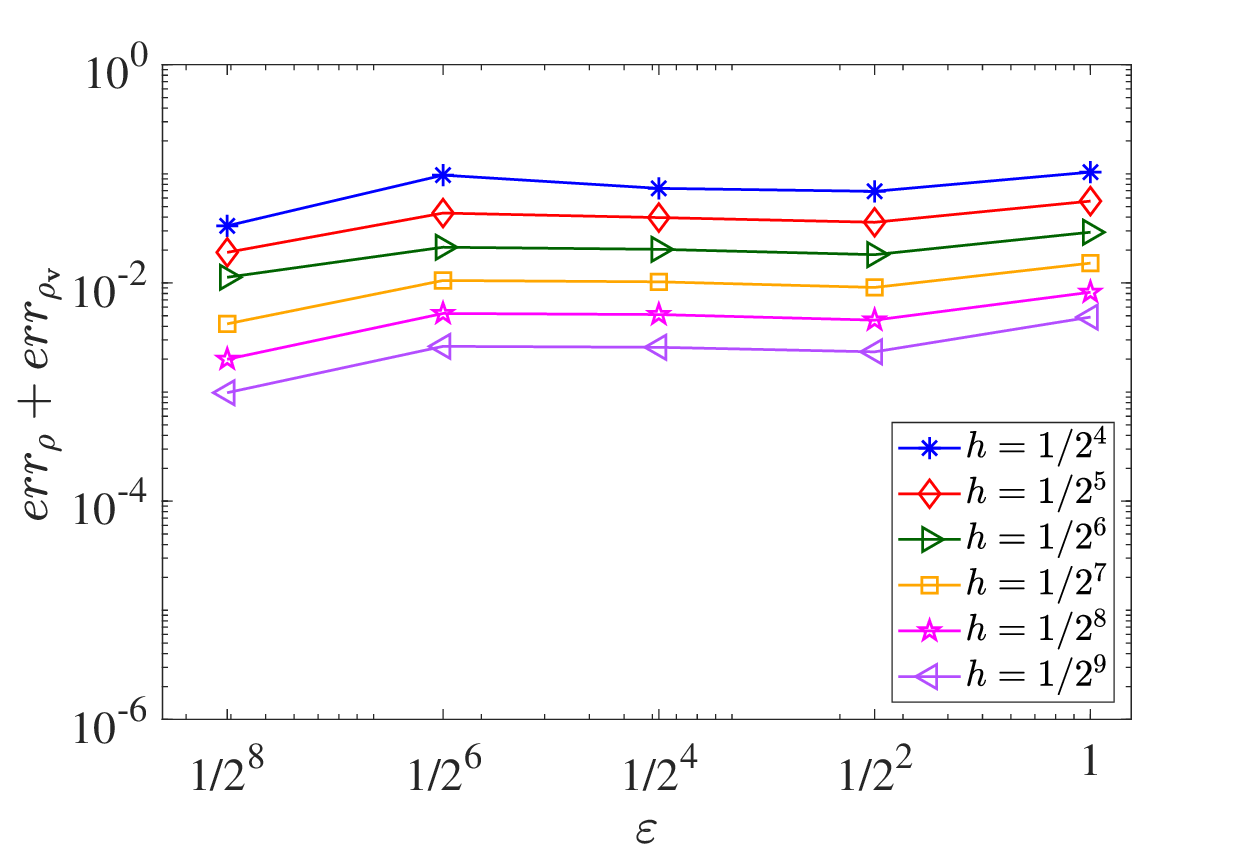,height=3.5cm,width=6.5cm}
\psfig{figure=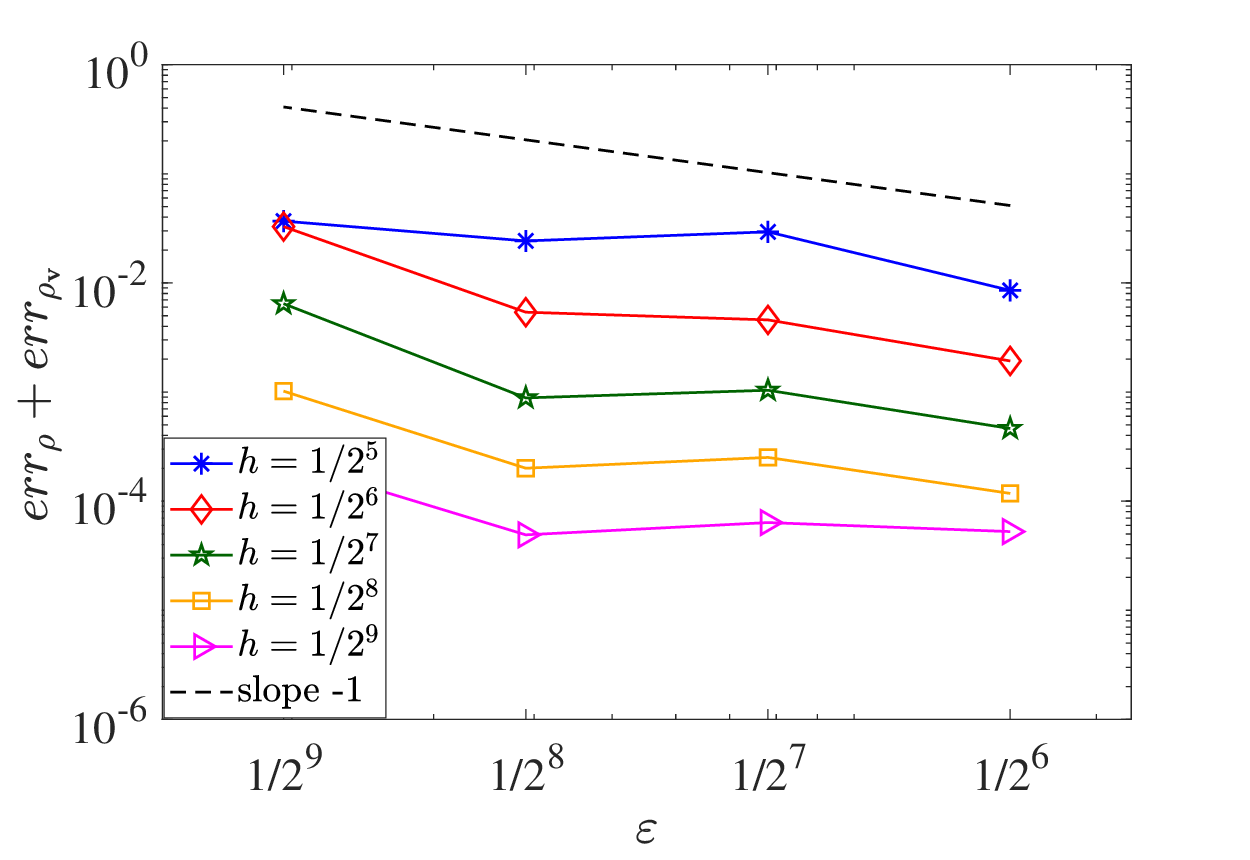,height=3.5cm,width=6.5cm}
\end{array}$$
\caption{Example 1. Errors of RS1-PIC (left) and RS2-PIC (right) with respect to $\eps$ under different $h$ about $\rho$ and $\rho_{\mathbf{v}}$.}\label{fig-1-1}
\end{figure}

\begin{table}[htbp]
\centering
\caption{Example 1. Time error of RS2-PIC for different $\eps$ and $h$.}
\begin{tabular}{cccccc}
\toprule
$err_{\rho}+err_{\rho_{\bv}} (t=1)$        &  $h=1/2^{5}$   &   $h=1/2^{6}$    &   $h=1/2^{7}$  &  $h=1/2^{8}$  &  $h=1/2^{9}$  \\
\midrule
$\eps=1/2^{6}$      &   8.54E-3   &  1.93E-3   &   4.63E-4   &   1.18E-4   &   5.26E-5   \\
order               &  -          &   2.15     &   2.06      &    1.98     &   1.16      \\
$\eps=1/2^{7}$      &   2.94E-2   &  4.58E-3   &   1.04E-3   &   2.52E-4   &   6.36E-5   \\
order               &  -          &   2.68     &    2.14     &    2.05     &    1.98        \\
$\eps=1/2^{8}$      &   2.42E-2   &  5.39E-3   &   8.84E-4   &   2.00E-4   &   4.91E-5    \\
order               &  -          &   2.17     &     2.61    &    2.14     &    2.03          \\
$\eps=1/2^{9}$      &   3.69E-2   &  3.28E-2   &   6.45E-3   &   1.02E-3   &   2.32E-4  \\
order               &  -          &   0.17     &    2.35     &    2.66     &   2.14       \\
\bottomrule
\end{tabular}
\label{tab-1-2}
\end{table}

\begin{figure}[t!]
$$\begin{array}{cc}
\psfig{figure=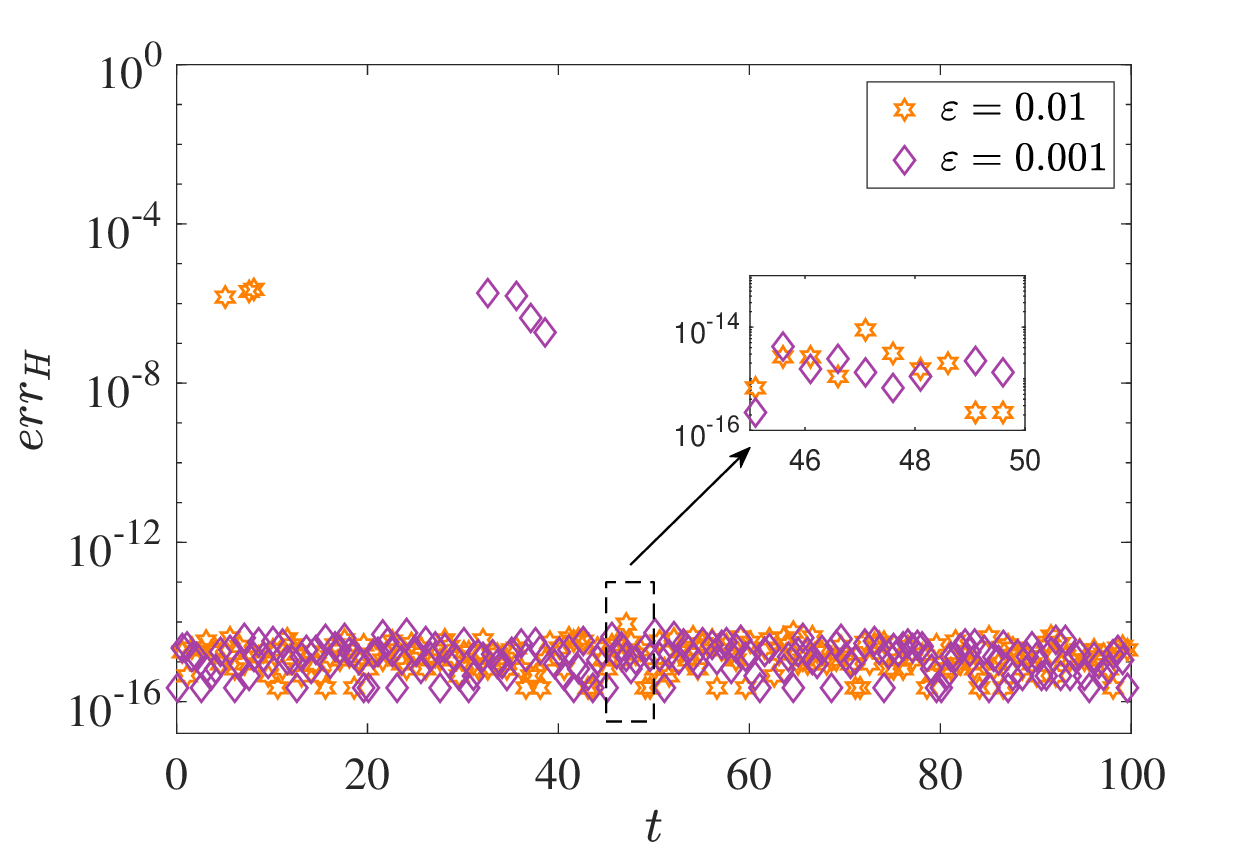,height=3.5cm,width=6.5cm}
\psfig{figure=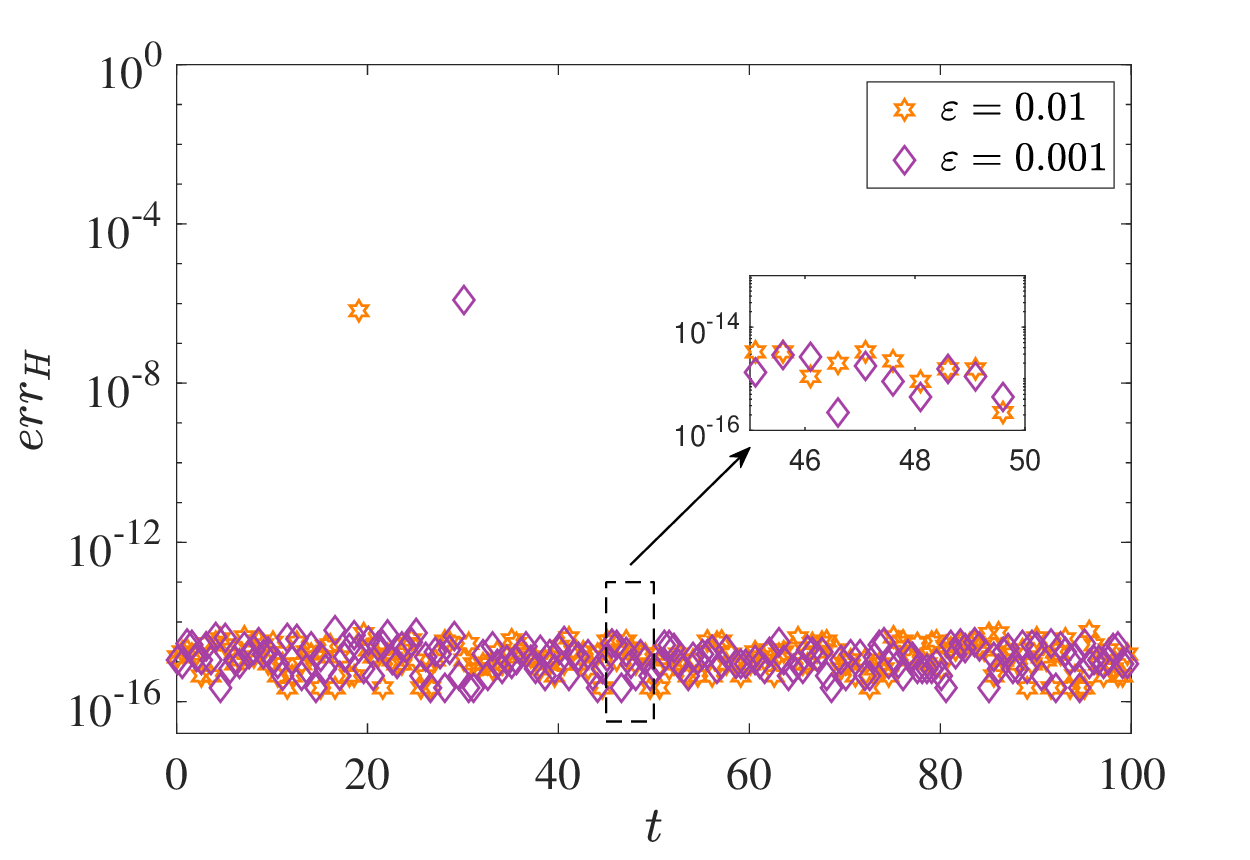,height=3.5cm,width=6.5cm}
\end{array}$$
\caption{Example 1. Energy error of RS1-PIC (left) and RS2-PIC (right) with step size $h=0.1$ until $T=100$.}\label{fig-1-3}
\end{figure}

\begin{figure}[t!]
$$\begin{array}{cc}
\psfig{figure=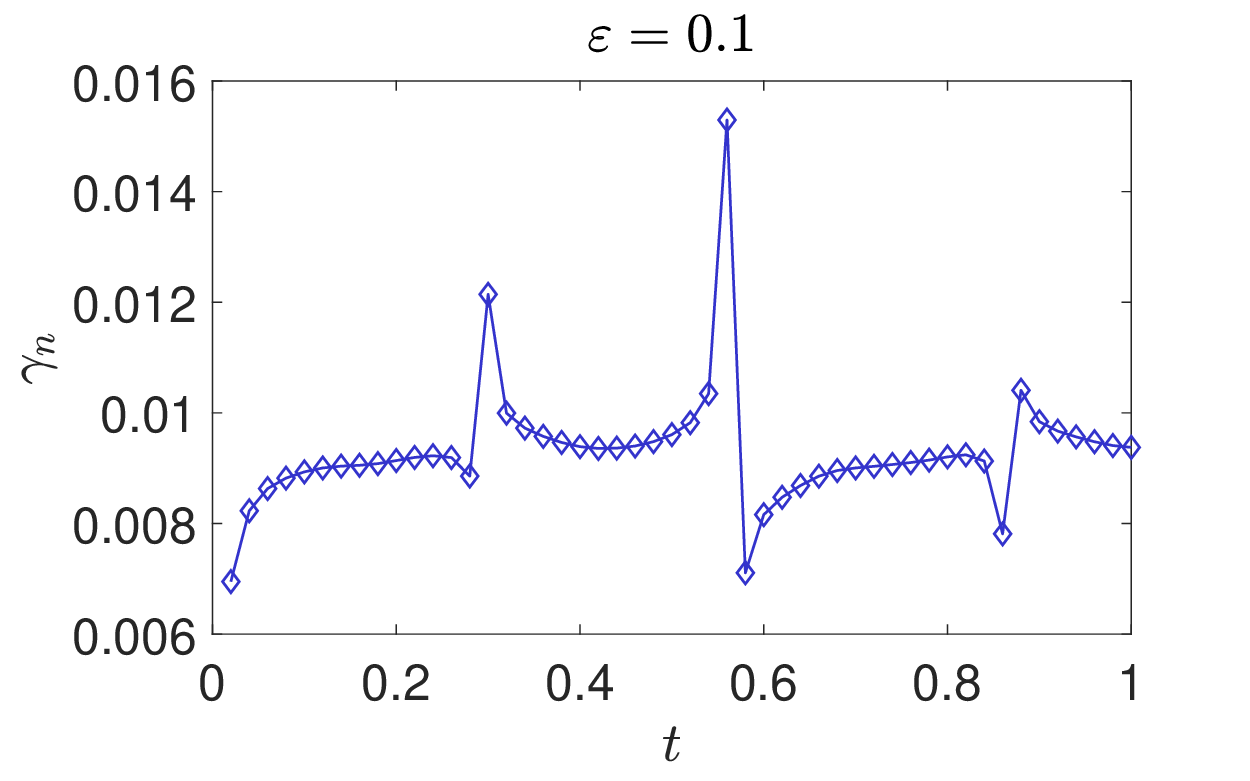,height=3.5cm,width=5cm}
\psfig{figure=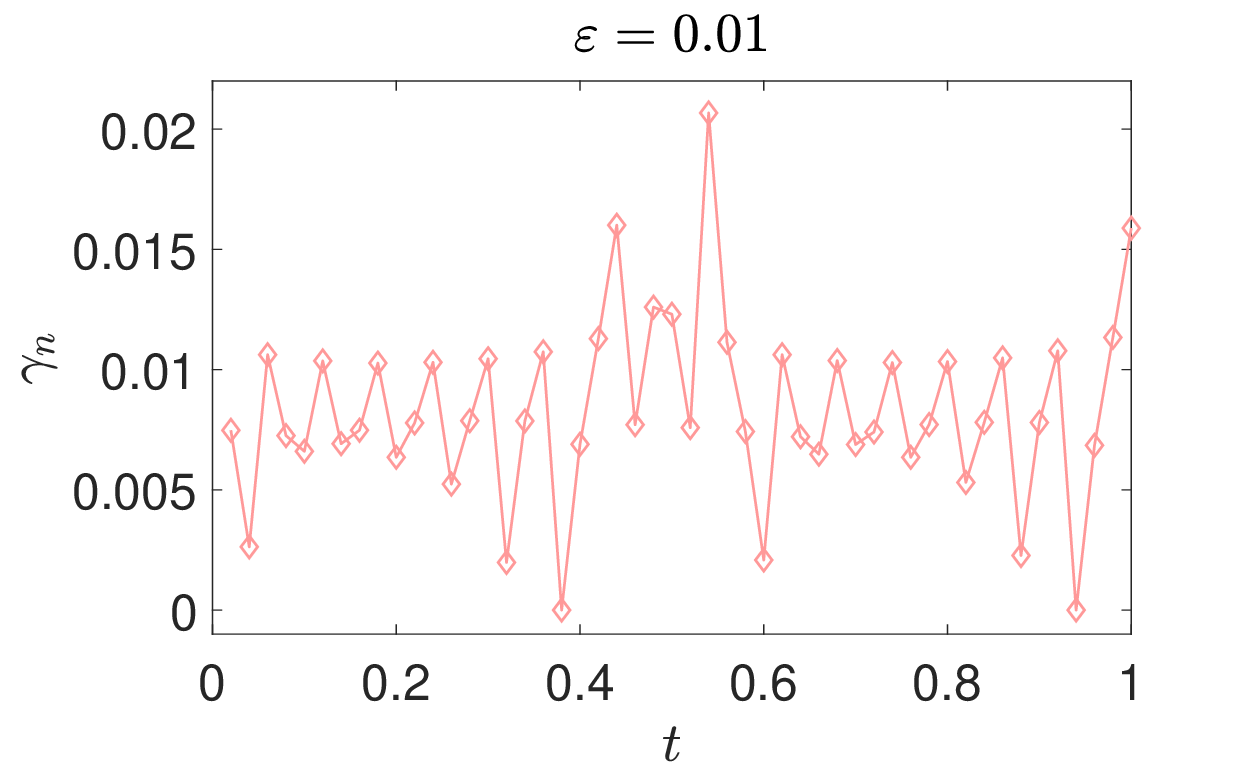,height=3.5cm,width=5cm}
\psfig{figure=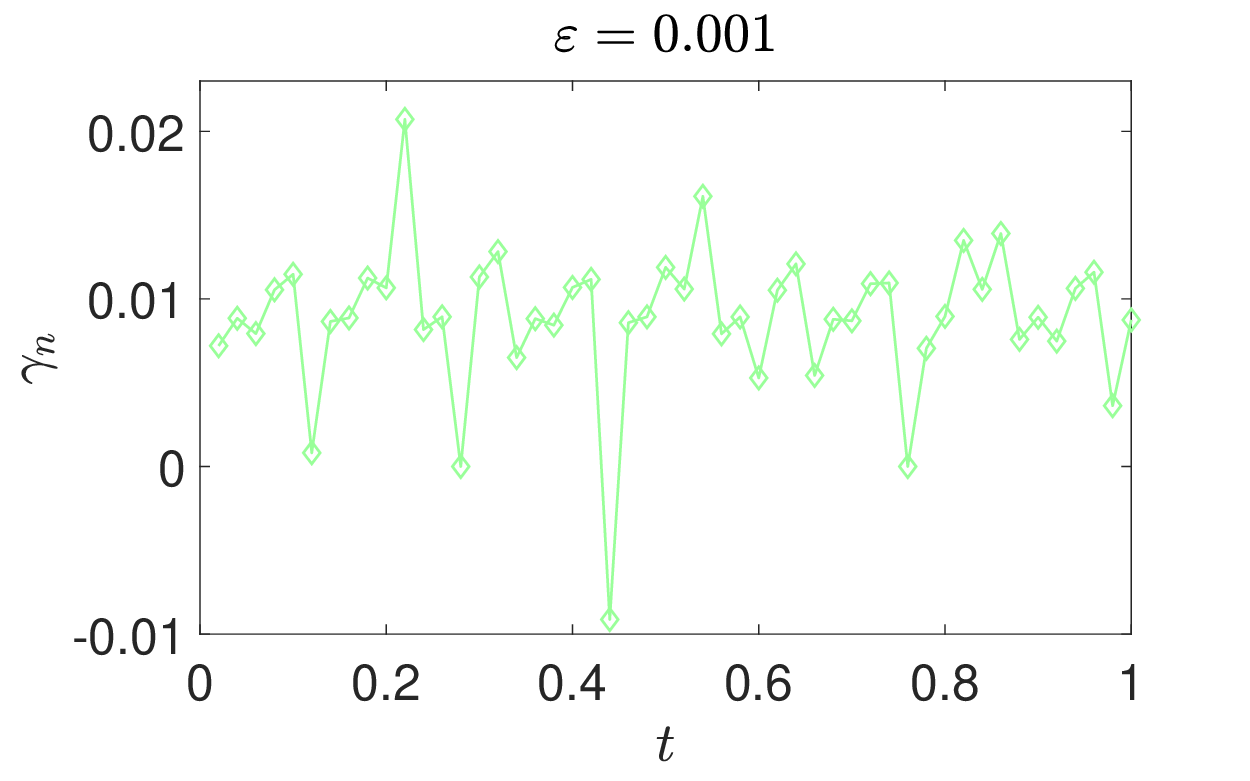,height=3.5cm,width=5cm}
\end{array}$$
\caption{Example 1. Evolution of relaxation parameter $\gamma_{n}$ by using RS2-PIC with $h=0.02$ .}\label{fig-1-7}
\end{figure}

\textbf{Example 2.} (Diocotron instability) The diocotron instability arises in plasmas from the relative drift between adjacent charged sheets. Typically described by the guiding centre model, this phenomenon has been widely studied in numerical simulations. We take the initial distribution function as (\cite{FR1}): $f_{0}(\bx,\bv)=\frac{d_{0}}{2\pi}\exp(-\frac{\norm{\bv}^{2}}{2}), \bx=(x_{1},x_{2})^{\intercal}\in \mathbb{R}^{2}$, where the initial density is 
$d_{0}(\bx)= 
(1+\alpha \cos(l\theta))\exp(-4(\norm{\bx}-6.5)^{2}),  \mbox{if} \ r^{-}\leq \norm{\bx} \leq r^{+},$ and
0 otherwise. Here $\theta=\arctan(x_{2}/x_{1})$ and $l$ the number of vortices. We choose a strong magnetic field $\frac{1}{\eps}\bB(\bx)=\frac{1}{\eps}(0, 0, 1)^{\intercal}$. The simulation adopts the parameters $r^{-}=5$, $r^{+}=8$ and $\alpha=0.2$. The computational domain is set to $[-12, 12]\times [-12, 12]$. A uniform initial grid of $N_{x_{1}}=N_{x_{2}}=128$ is used, with 50 particles per cell, corresponding to a total of $819200$ particles in the simulation.
Figure \ref{fig-2-1} presents the results for $\eps=0.01$, simulated using the RS2-PIC scheme with $h=0.01$ and $l=5$. It is observed that the five vortex structures are well captured. The other performance is similar to Example 1 and is not displayed for brevity.

\begin{figure}[t!]
$$\begin{array}{cc}
\psfig{figure=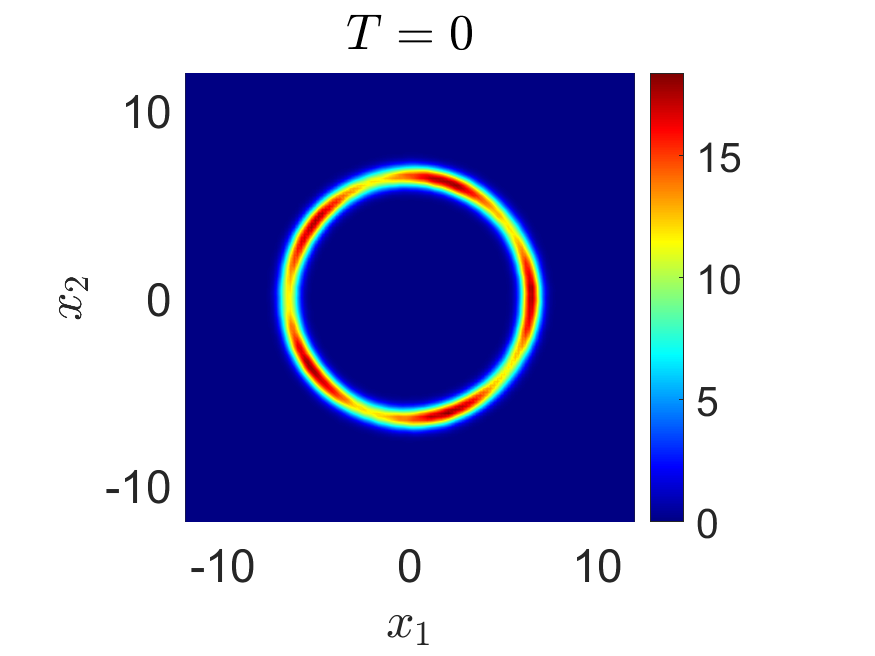,height=3.8cm,width=5cm}
\psfig{figure=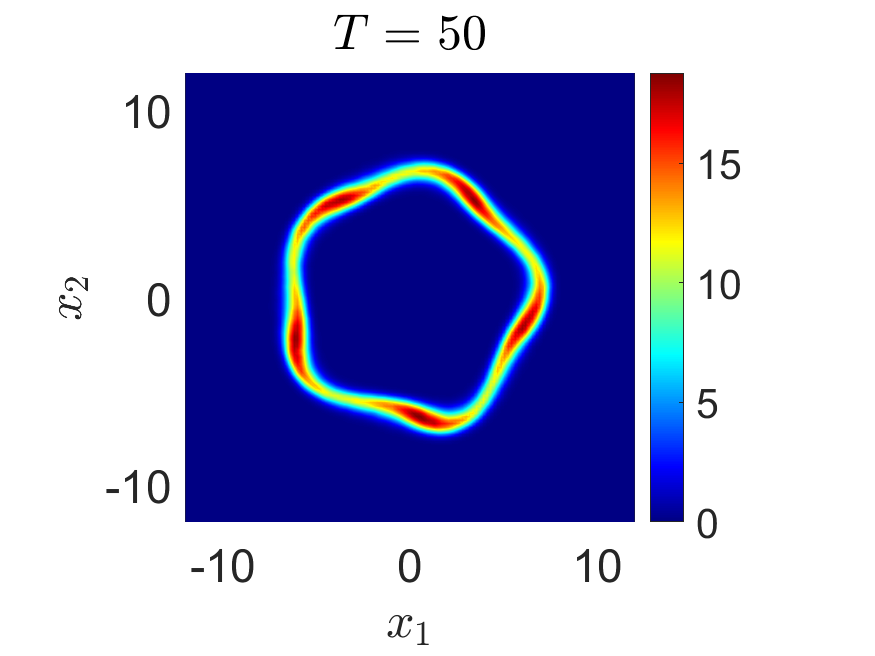,height=3.8cm,width=5cm}
\psfig{figure=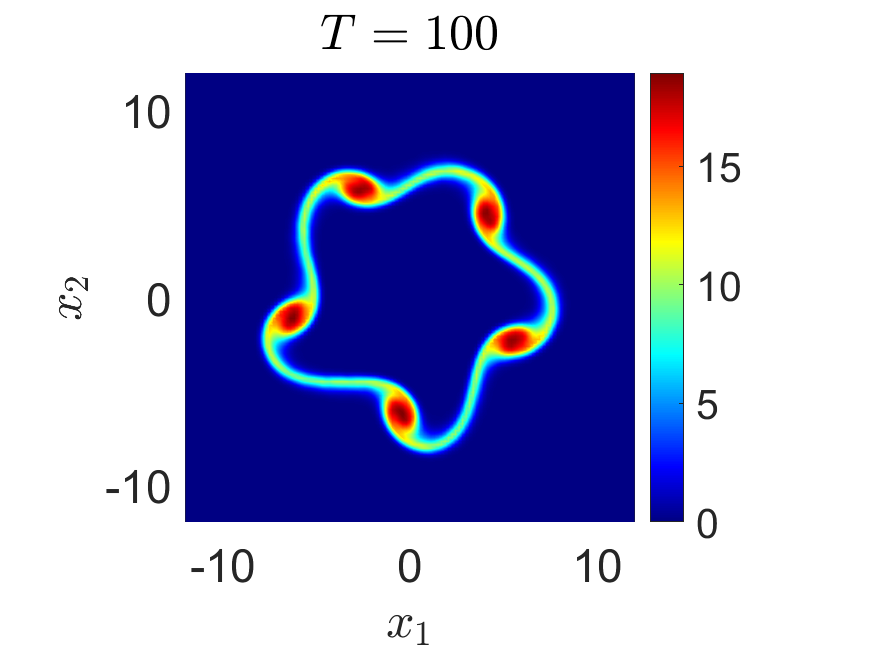,height=3.8cm,width=5cm}
\end{array}$$
\caption{Example 2. Time evolution of the density $\rho(t,\bx)$ until $T=100$ with strength of magnetic field $\eps=0.01$.}\label{fig-2-1}
\end{figure}

\subsection{Vlasov-Poisson equation in Larmor scaling}
In this subsection, we consider the Vlasov-Poisson equation in the finite Larmor radius approximation regime (\cite{FS,CCZ}):
\begin{subequations}\label{vp-larmor}
\begin{align}
&\partial_{t}f(t,\bx,\bv)+\frac{\bv}{\eps}\cdot\nabla_{\bx}f(t,\bx,\bv)+\left(\bE(t,\bx)+\frac{1}{\eps}\bv\times\bB(\bx)\right)\cdot\nabla_{\bv}f(t,\bx,\bv)=0, \label{vp-larmor-a} \\
&\nabla_{\bx}\cdot\bE(t,\bx)=\int_{\mathbb{R}^d}f(t,\bx,\bv)d\bv-n_{i}, \ f(0,\bx,\bv)=f_{0}(\bx,\bv), \ (\bx,\bv)\in \Omega\times\mathbb{R}^{d}, \ d\geq 2, \label{vp-larmor-b}  
\end{align}
\end{subequations}
with the energy conservation law
\begin{equation}\label{vp-lar-ener}
H_{lar}(t)=\frac{1}{2}\int_{\mathbb{R}^d}\int_{\Omega}\abs{\bv}^{2}f(t,\bx,\bv)d\bx d \bv+\frac{\eps}{2}\int_{\Omega}\abs{\bE(t,\bx)}^{2}d\bx=H_{lar}(0).
\end{equation}
Using the PIC discretization, the characteristic equation of \eqref{vp-larmor} is given as
\begin{equation}\label{char-larmor}
\dot{\bx}(t)=\dfrac{\bv(t)}{\eps}, \ \dot{\bv}(t)=\dfrac{1}{\eps}\bv(t)\times \bB(\bx(t))+\bE_{[\bx^{p}(t)]}(\bx(t)), \ 
\bx(0)=\bx_{0}, \ \bv(0)=\bv_{0}, \ 0\leq t\leq T.
\end{equation}
In order to ensure that the relaxation factor $\gamma_{n}$ at each step is not affected by $\eps$, we introduce a time-scale transformation $\tau=t/\eps$, the system \eqref{char-larmor} is transformed into the following second-order equation
\begin{equation}\label{lar-2}
\ddot{\bx}(\tau)=\dot{\bx}(\tau)\times \bB(\bx(\tau))+\eps\bE(\bx(\tau)), \ \bx(0)=\bx_{0}, \ \dot{\bx}(0)=\bv_{0}, \quad \tau\in [0,T/\eps].
\end{equation}
To maintain the simplicity of the presentation, we still use the original notation in the rescaled system \eqref{lar-2}.
Similarly, two energy-conserving PIC methods can be obtained for this system, with error bounds of $\mathcal{O}(\triangle\tau)$ and $\mathcal{O}(\triangle\tau^2)$, respectively.


\textbf{Example 3}. 
We next conduct numerical simulations over the extended interval $\tau\in [0,T/\eps]$. The simulation setup mirrors that of Example 1, employing the same magnetic field \eqref{mag-1}, initial conditions, and total particle count for the PIC method. A highly accurate solution, computed using a fourth-order Runge-Kutta method with a time step of $10^{-5}$, serves as the reference.  To illustrate the solution dynamics, Figures \ref{fig-3-4} and \ref{fig-3-5} present the evolution of $\chi(t,\bv)$ and $\rho(t,\bx)-n_{i}$, respectively, computed with a time step of $\triangle\tau=0.1$. We then assess the convergence of the proposed schemes for $\rho$ and $\rho_{\mathbf{v}}$ at a final time of $T=1$ ($\tau=T/\eps$). Table \ref{tab-3-1} and Figure \ref{fig-3-2} confirm that, over long times, the RS1-PIC and RS2-PIC schemes exhibit first- and second-order convergence, respectively. Lastly, we examine the long-term energy conservation up to $T=100$. Figure \ref{fig-3-3} reveals that both methods maintain exceptional energy accuracy over this interval.
\begin{figure}[t!]
$$\begin{array}{cc}
\psfig{figure=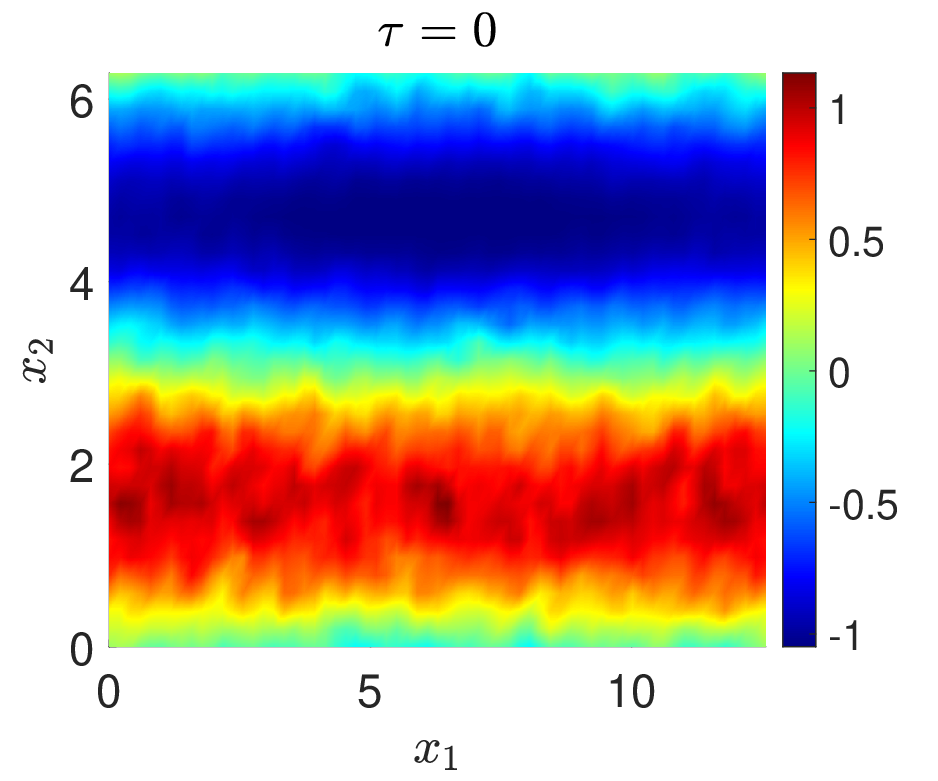,height=3.8cm,width=4.8cm}
\psfig{figure=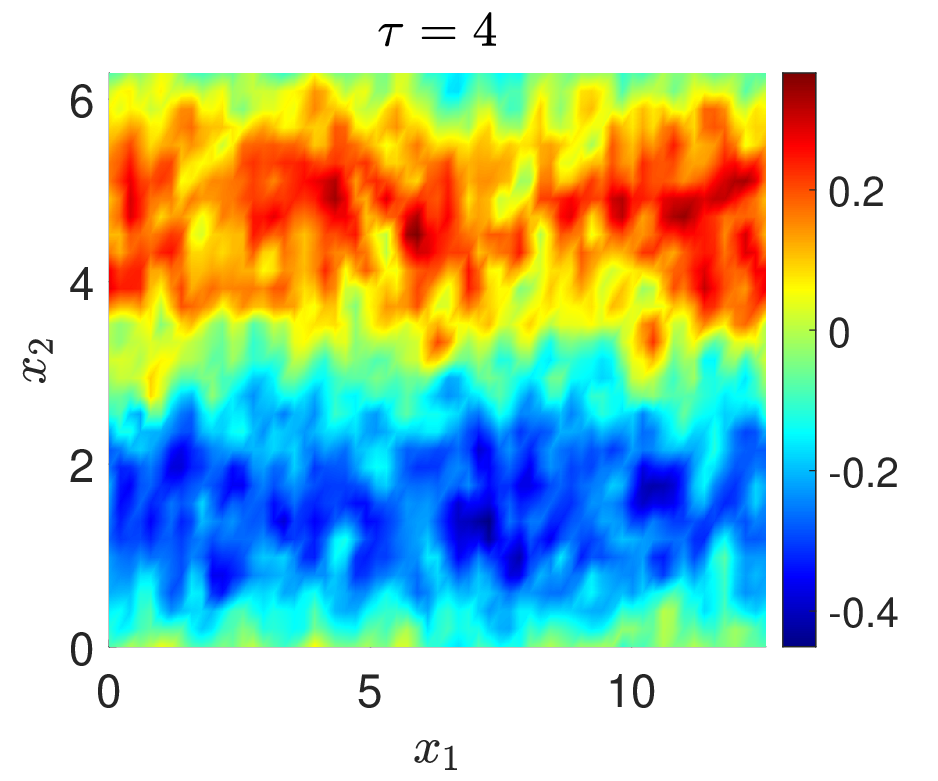,height=3.8cm,width=4.8cm}
\psfig{figure=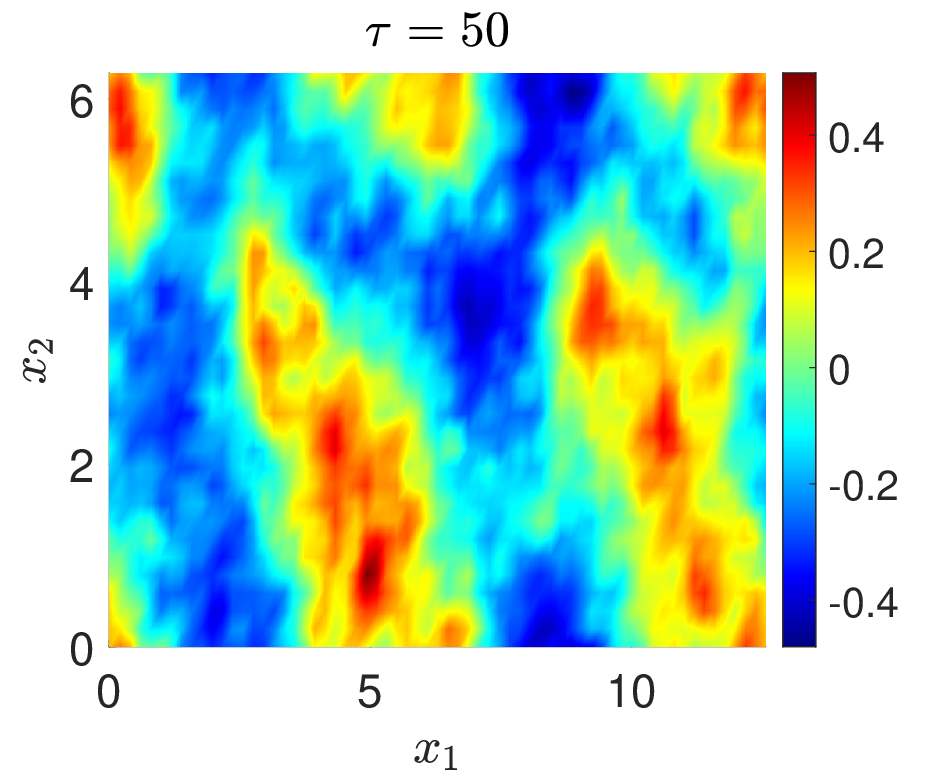,height=3.8cm,width=4.8cm}
\end{array}$$
\caption{Example 3. Contour plot of quantity $\rho(t,\bx)-n_{i}$ at different $\tau$ with $\eps=0.1$ for RS2-PIC.}\label{fig-3-4}
\end{figure}

\begin{figure}[t!]
$$\begin{array}{cc}
\psfig{figure=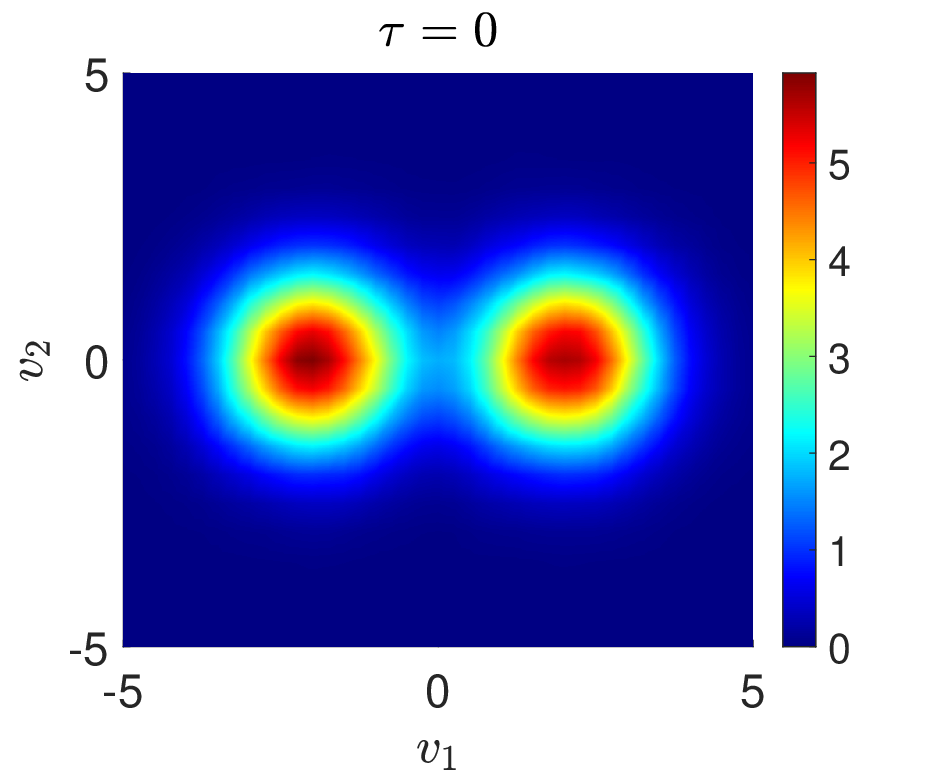,height=3.8cm,width=4.8cm}
\psfig{figure=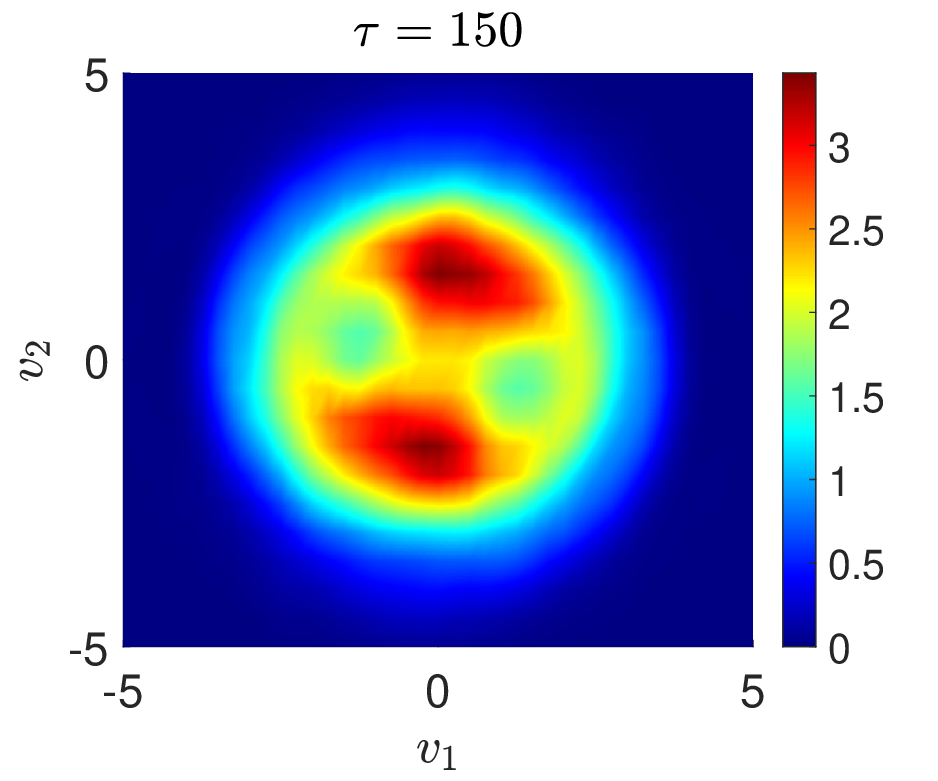,height=3.8cm,width=4.8cm}
\psfig{figure=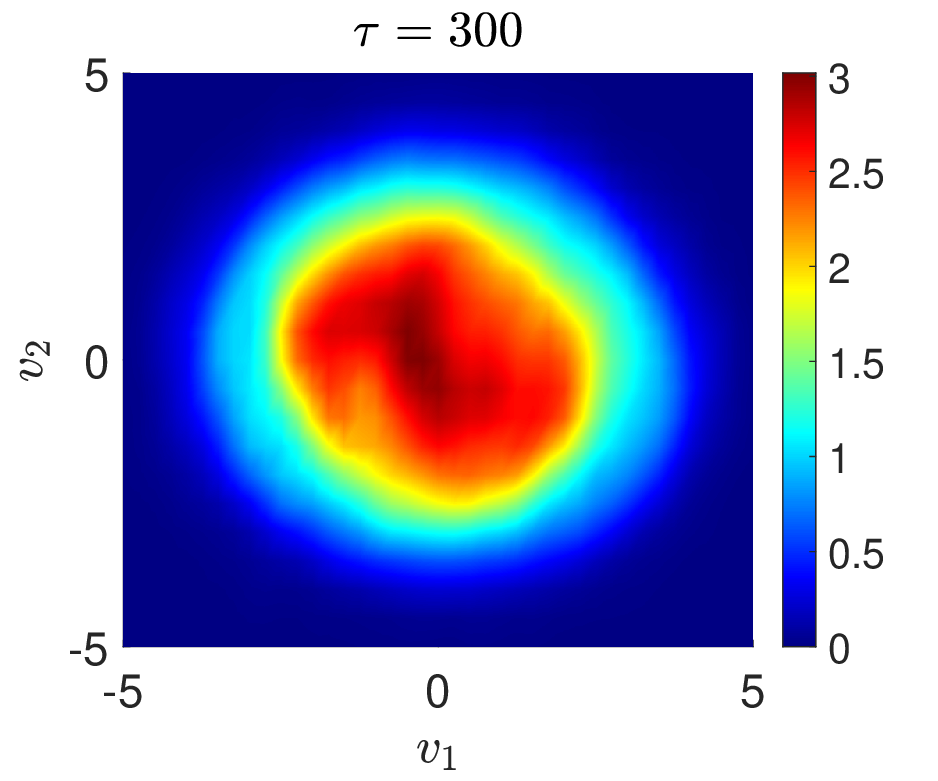,height=3.8cm,width=4.8cm}
\end{array}$$
\caption{Example 3. Contour plot of quantity $\chi(t,\bv)$ at different $\tau$ with $\eps=0.1$ for RS2-PIC.}\label{fig-3-5}
\end{figure}


\begin{table}[htbp]
\centering
\caption{Example 3. Time error of RS1-PIC (top) and RS2-PIC (bottom) for different $\eps$ and $h$.}
\begin{tabular}{cccccc}
\toprule
$err_{\rho}+err_{\rho_{\bv}} (T=1)$        &   $\triangle\tau=1/2^{4}$    &   $\triangle\tau=1/2^{5}$  &  $\triangle\tau=1/2^{6}$  &  $\triangle\tau=1/2^{7}$  &  $\triangle\tau=1/2^{8}$  \\
\midrule
$\eps=1/2$          &   1.03E-1   &   5.80E-2  &   3.02E-2   &   1.54E-2   &   7.96E-3    \\
order               &    -        &   0.83     &   0.94      &    0.97     &   0.95          \\
$\eps=1/2^{2}$      &   8.89E-2   &   4.65E-2  &   2.37E-2   &   1.20E-2   &   6.20E-3   \\
order               &   -         &   0.94     &    0.97     &   0.98      &   0.96           \\
$\eps=1/2^{3}$      &   1.00E-1   &   5.54E-2  &   2.88E-2   &   1.48E-2   &   7.76E-3  \\
order               &   -         &   0.86     &   0.94      &   0.96      &    0.94          \\
\toprule
$err_{\rho}+err_{\rho_{\bv}} (T=1)$        &   $\triangle\tau=1/2$    &   $\triangle\tau=1/2^{2}$  &  $\triangle\tau=1/2^{3}$  &  $\triangle\tau=1/2^{4}$  &  $\triangle\tau=1/2^{5}$  \\
\midrule
$\eps=1/2$          &   7.34E-2   &   1.91E-2   &   4.77E-3   &   1.99E-3   &   1.70E-3     \\
order               &    -        &    1.94     &    2.00     &    1.26     &    0.22        \\
$\eps=1/2^{2}$      &   4.18E-2   &   1.08E-2   &   2.54E-3   &   1.09E-3   &   1.15E-3    \\
order               &   -         &    1.95     &   2.09      &    1.22     &    -0.08        \\
$\eps=1/2^{3}$      &   5.37E-2   &   1.28E-2   &   3.56E-3   &   1.72E-3   &   1.70E-3     \\
order               &   -         &    2.07     &    1.85     &   1.05      &    0.01         \\
\bottomrule
\end{tabular}
\label{tab-3-1}
\end{table}

\begin{figure}[t!]
$$\begin{array}{cc}
\psfig{figure=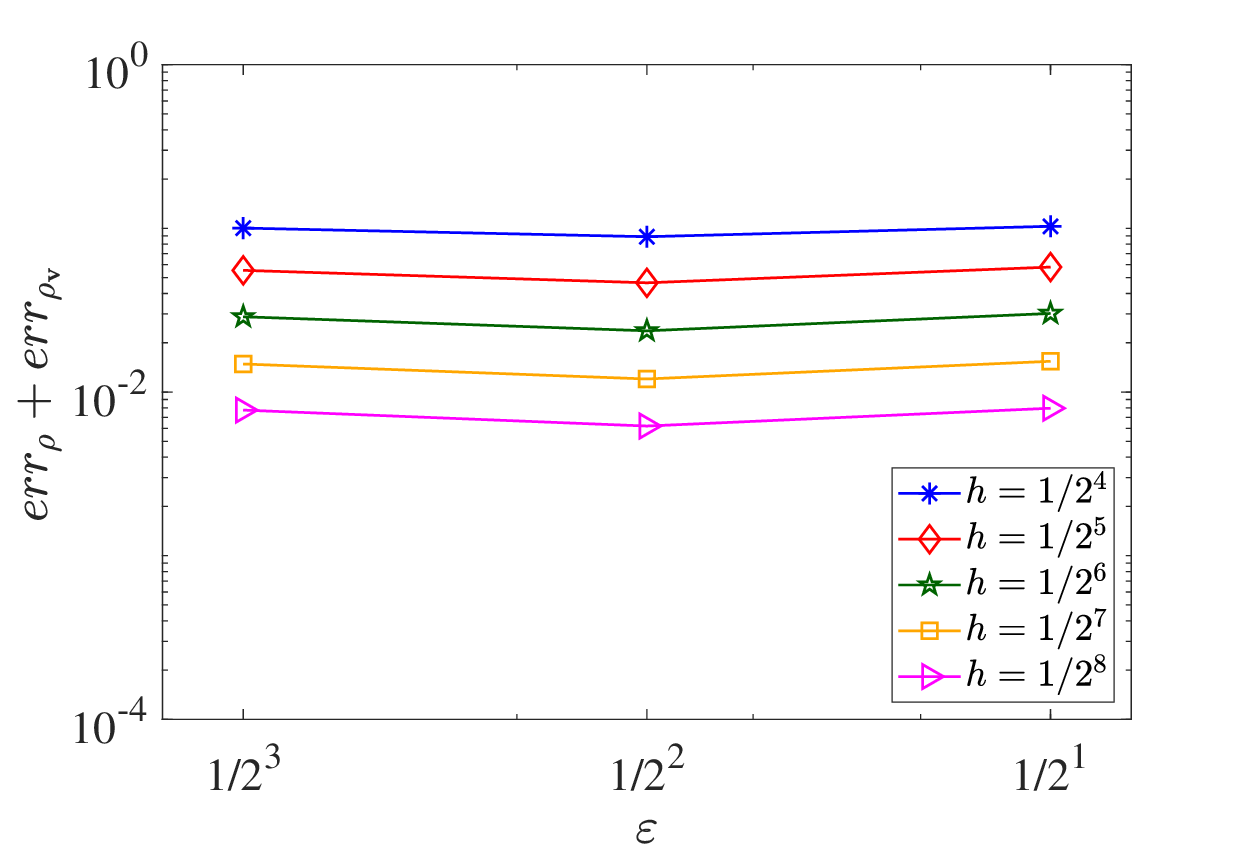,height=3.5cm,width=6.5cm}
\psfig{figure=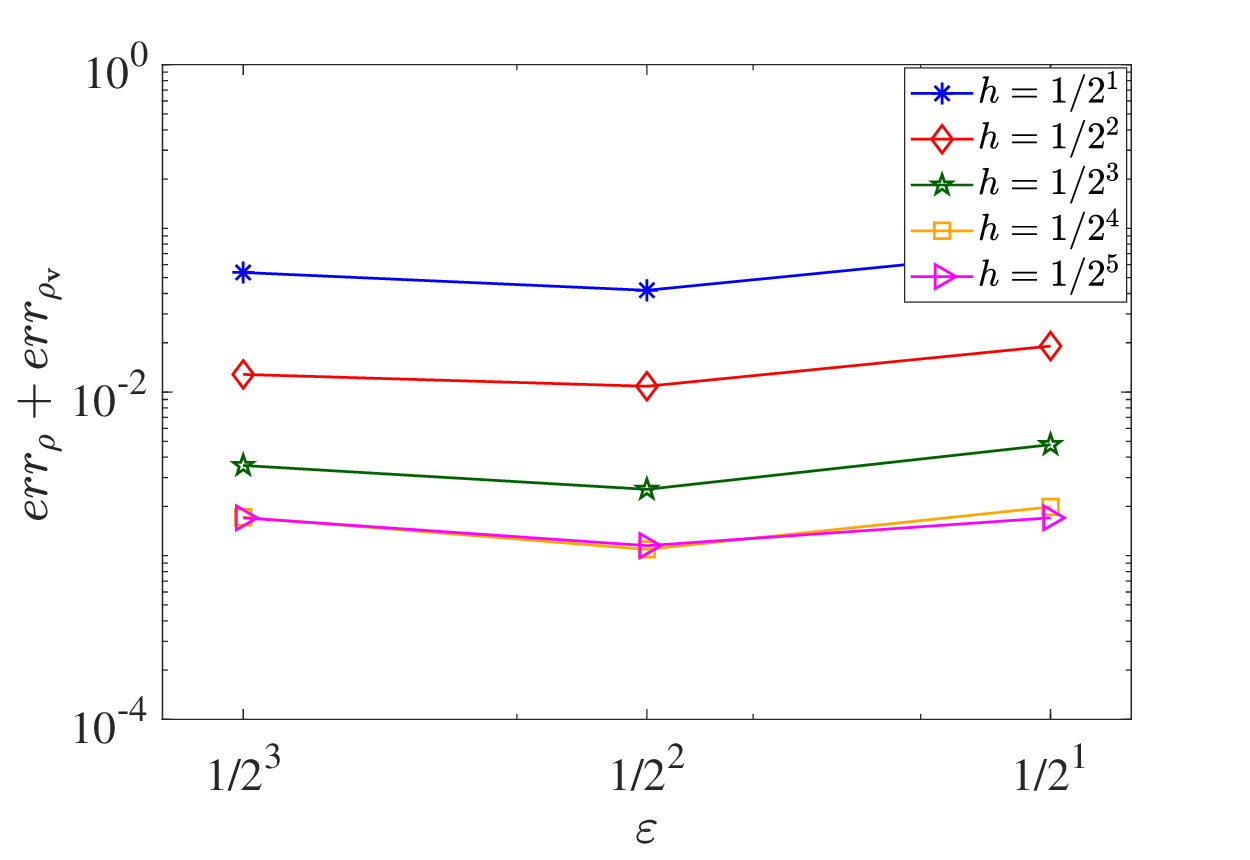,height=3.5cm,width=6.5cm} 
\end{array}$$
\caption{Example 3. Errors of RS1-PIC (left) and RS2-PIC (right) with respect to $\eps$ under different $h$ until $T=1$.}\label{fig-3-2}
\end{figure}

\begin{figure}[t!]
$$\begin{array}{cc}
\psfig{figure=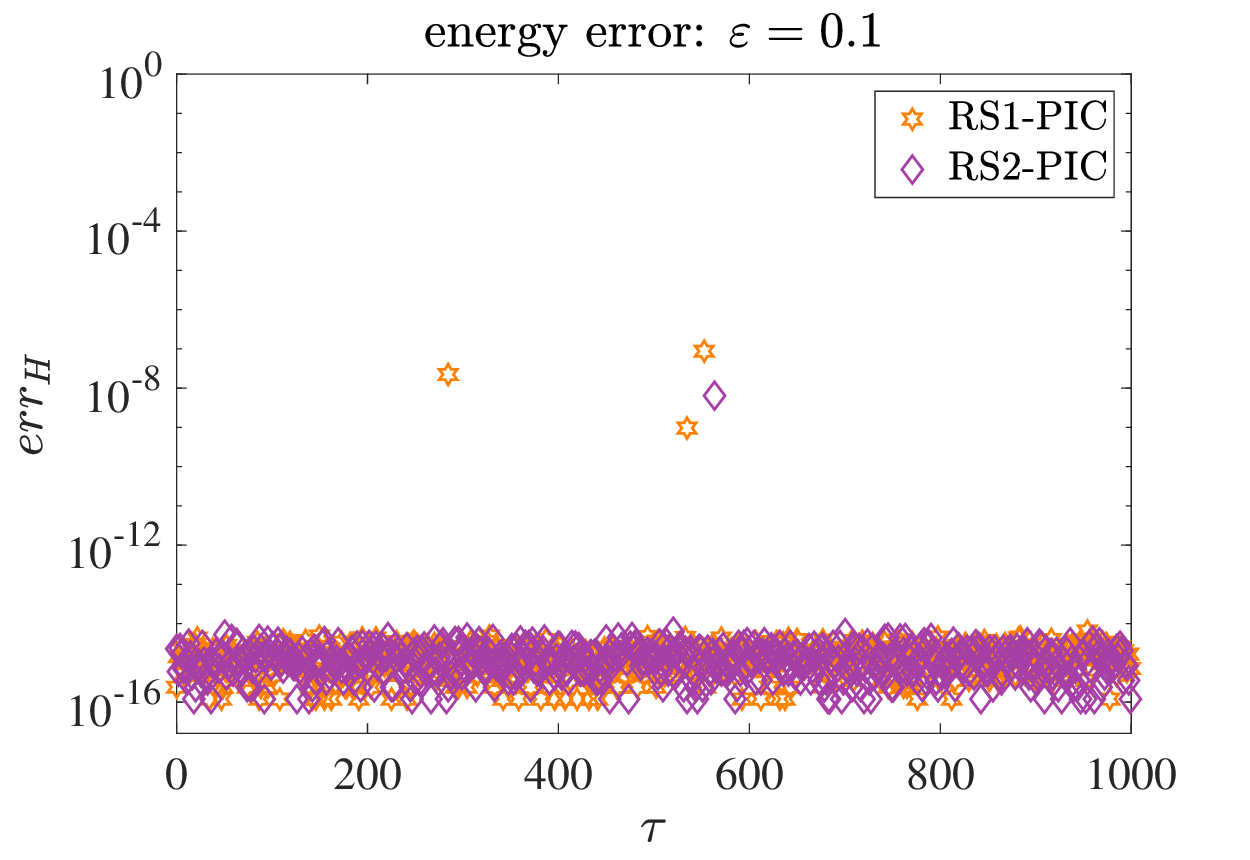,height=3.5cm,width=6.5cm}
\psfig{figure=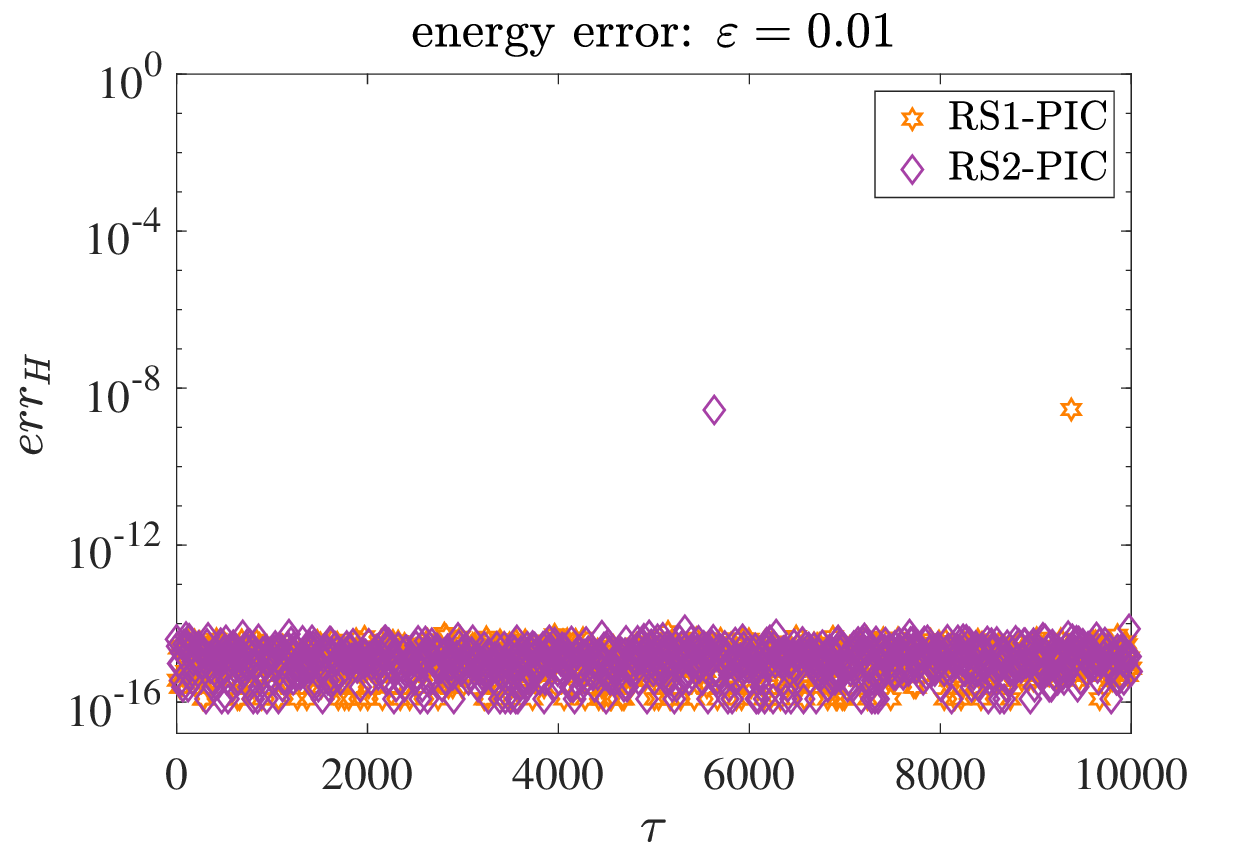,height=3.5cm,width=6.5cm}
\end{array}$$
\caption{Example 3. Energy error of RS1-PIC and RS2-PIC with step size $h=0.1$ under different $\eps$ until $T=100$.}\label{fig-3-3}
\end{figure}

\subsection{Vlasov-Poisson equation in Diffusion scaling}
Finally, we consider the Vlasov-Poisson equation in the Diffusion scaling (\cite{FR2,CLMZ}):
\begin{subequations}\label{vp-diffu}
\begin{align}
&\partial_{t}f(t,\bx,\bv)+\frac{\bv}{\eps}\cdot\nabla_{\bx}f(t,\bx,\bv)+\frac{1}{\eps}\left(\bE(t,\bx)+\frac{1}{\eps}\bv\times\bB(\bx)\right)\cdot\nabla_{\bv}f(t,\bx,\bv)=0, \label{vp-diffu-a} \\
&\nabla_{\bx}\cdot\bE(t,\bx)=\int_{\mathbb{R}^d}f(t,\bx,\bv)d\bv-n_{i}, \ f(0,\bx,\bv)=f_{0}(\bx,\bv), \ (\bx,\bv)\in \Omega\times\mathbb{R}^{d}, \ d\geq 2, \label{vp-diffu-b}  
\end{align}
\end{subequations}
with the energy conservation law \eqref{vp-ener}. In the PIC discretization, the characteristic equation of \eqref{vp-diffu} is given by
\begin{equation}\label{char-diffu}
\dot{\bx}(t)=\dfrac{\bv(t)}{\eps}, \ \dot{\bv}(t)=\dfrac{1}{\eps^{2}}\bv(t)\times \bB(\bx(t))+\dfrac{1}{\eps}\bE_{[\bx^{p}(t)]}(\bx(t)), \
\bx(0)=\bx_{0}, \ \bv(0)=\bv_{0}, \ 0\leq t\leq T.
\end{equation}
Introducing a time-scale transformation $\tau=t/\eps$, system \eqref{char-diffu} is changed into
\begin{equation}\label{diffu-2}
\ddot{\bx}(\tau)=\frac{1}{\eps}\dot{\bx}(\tau)\times \bB(\bx(\tau))+\bE(\bx(\tau)), \ \bx(0)=\bx_{0},\ \dot{\bx}(0)=\bv_{0}, \ \tau\in[0,T/\eps].
\end{equation}
We can also derive the RS1-PIC and RS2-PIC schemes analogously, yielding error bounds of $\mathcal{O}(\triangle\tau/\eps)$ and $\mathcal{O}(\triangle\tau^2/\eps^2)$, respectively. The performance of these schemes is then assessed through numerical simulations of the 2D system \eqref{vp-diffu} over the extended interval $\tau\in [0,T/\eps]$, where a fourth-order Runge-Kutta method with a time step of $10^{-5}$ provides a reference solution.

\textbf{Example 4}. In this example, we consider a Vlasov-Poisson system \eqref{vp-diffu} defined on a disk $\Omega=D(0,6)$ centered at the origin with a radius of $6$. The initial data is a Maxwellian in velocity, with a macroscopic density given by the sum of two Gaussians (\cite{FR2}):
\begin{align*}
f_{0}(\bx,\bv)=\frac{1}{16\pi^{2}}\left[\exp\left(-\frac{\norm{\bx-\tilde{\bx}_{0}}^2}{2}\right)+\exp\left(-\frac{\norm{\bx+\tilde{\bx}_{0}}^2}{2}\right) \right]\exp\left(-\frac{\norm{\bv}^2}{4}\right),\ (\bx,\bv)\in \Omega\times\mathbb{R}^{2},
\end{align*}
with $\tilde{\bx}_{0}=(1.5,-1.5)^{\intercal}$. Moreover, a non-homogeneous external strong magnetic field of the form $\frac{1}{\eps}b(\bx)=\frac{1}{\eps}+b_{1}(\bx)$ is considered, with $b_{1}(\bx)=\frac{10}{\sqrt{100-\norm{\bx}^2}}$. The spatial domain is discretized with $N_{x_{1}}=N_{x_{2}}=64$, the number of particles per cell is set to 50. 

\begin{figure}[t!]
$$\begin{array}{cc}
\psfig{figure=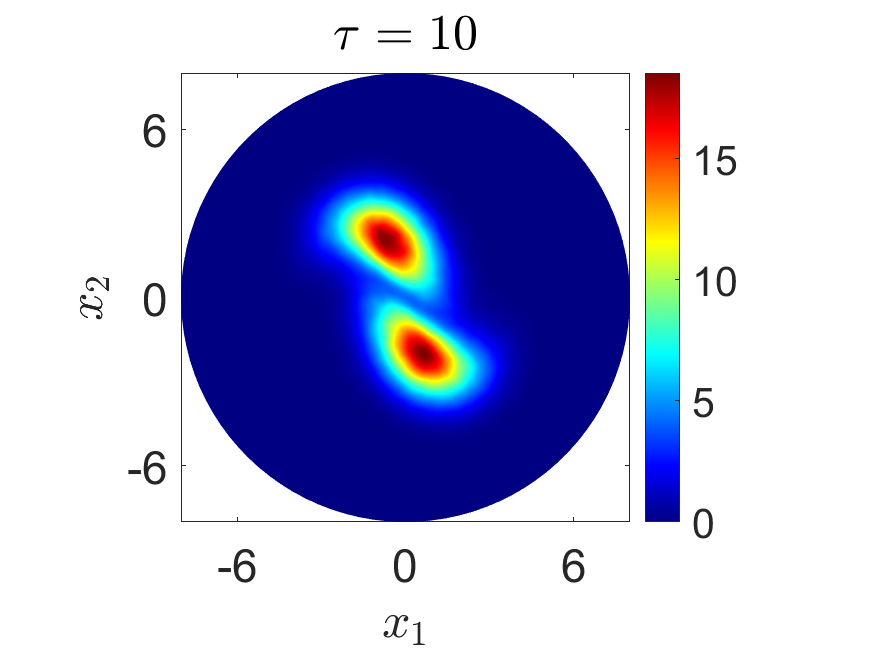,height=3.8cm,width=5cm}
\psfig{figure=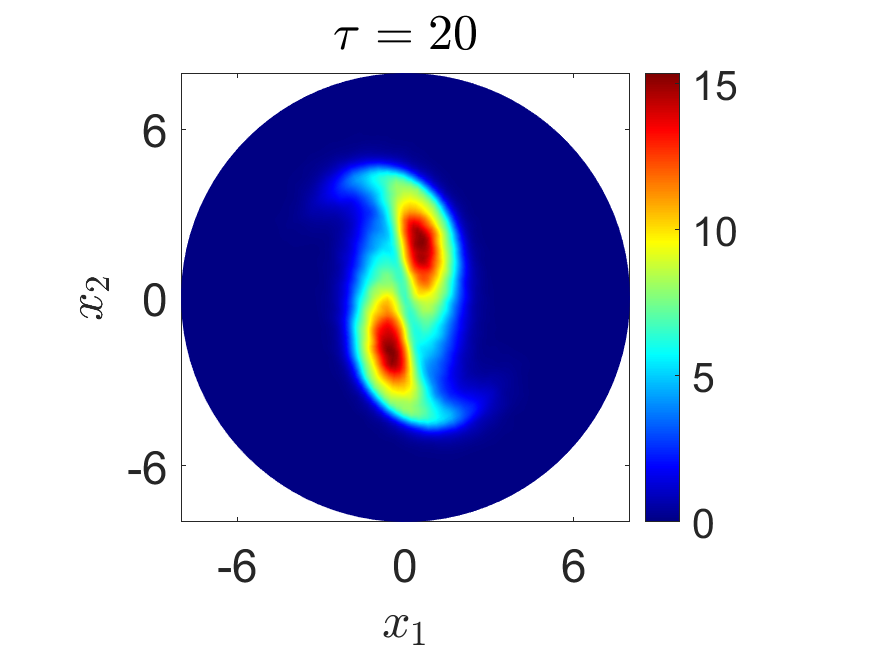,height=3.8cm,width=5cm}
\psfig{figure=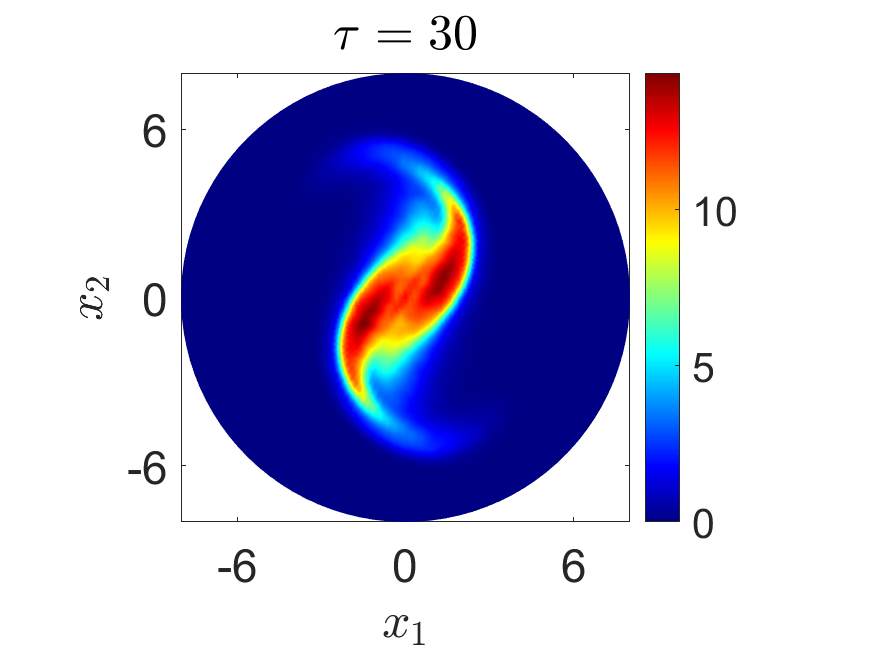,height=3.8cm,width=5cm}\\
\psfig{figure=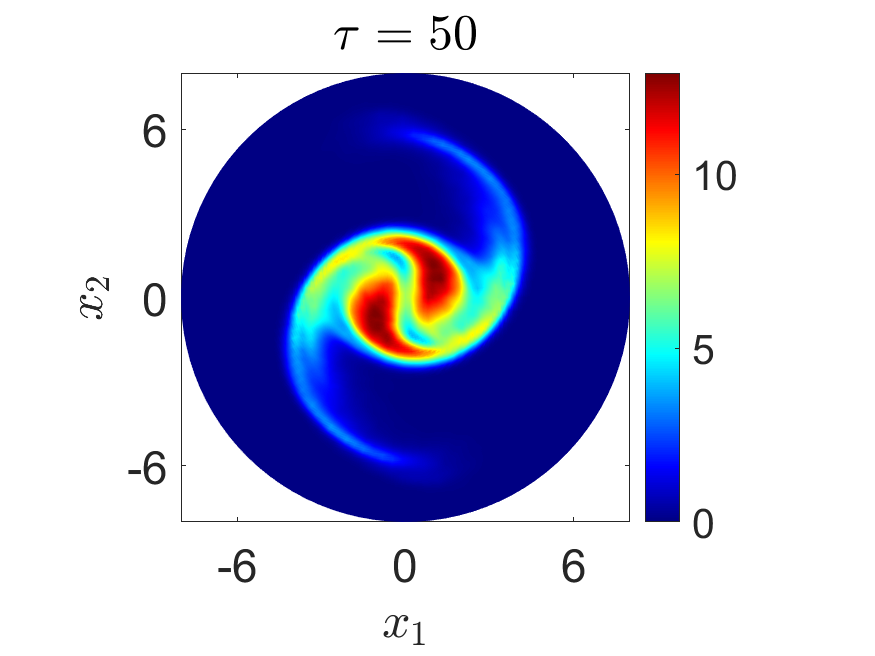,height=3.8cm,width=5cm}
\psfig{figure=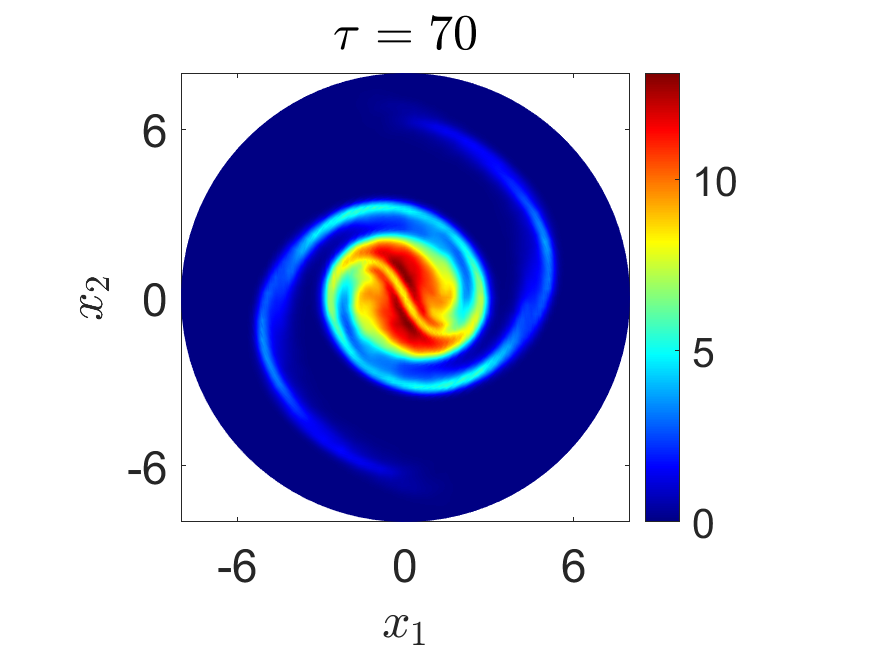,height=3.8cm,width=5cm}
\psfig{figure=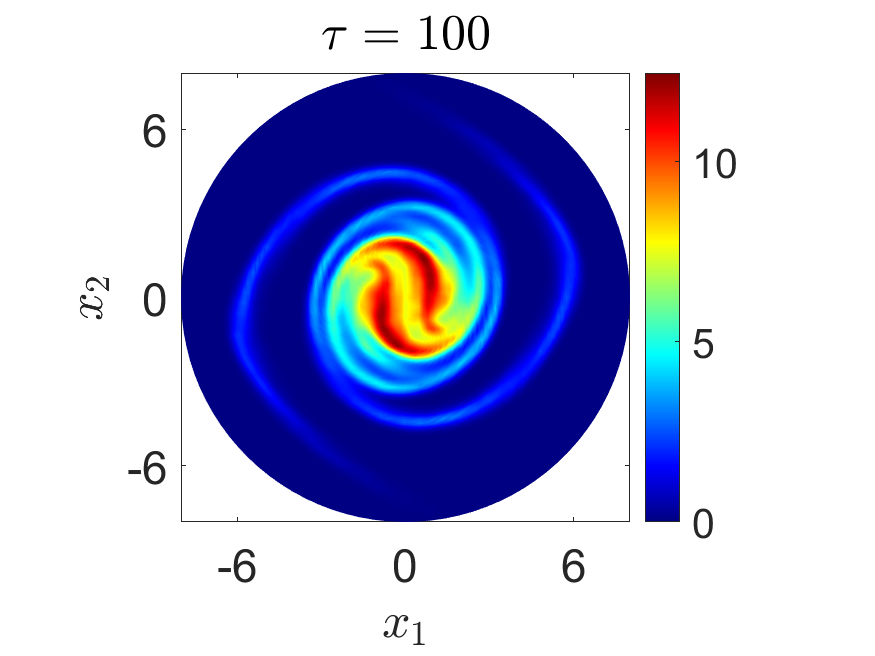,height=3.8cm,width=5cm}
\end{array}$$
\caption{Example 4. Snapshots of the time evolution of the density $\rho(t,\bx)$ until $\tau=100$ with strength of magnetic field $\eps=0.05$.}\label{fig-4-4}
\end{figure}
Figure \ref{fig-4-4} presents snapshots of the charge density evolution, illustrating the rapid merger of the two initial vortices for $\eps=0.05$ and a step size of $\triangle\tau=0.1$.
We then evaluate the convergence of $\rho$ and $\rho_{\mathbf{v}}$ for the Vlasov-Poisson system \eqref{vp-diffu} using the two schemes at $T=1$. As shown in Table \ref{tab-4-1} and Figure \ref{fig-4-2}, the RS1-PIC and RS2-PIC schemes achieve $\mathcal{O}(\triangle\tau/\eps)$ and $\mathcal{O}(\triangle\tau^2/\eps^2)$ accuracy, respectively, over long time scales.  Finally, the energy errors up to $T=100$ given in Figure \ref{fig-4-3} confirm the exceptional energy conservation properties of both methods.


\begin{table}[htbp]
\centering
\caption{Example 4. Time error of RS1-PIC (top) and RS2-PIC (bottom) for different $\eps$ and $h$.}
\begin{tabular}{cccccc}
\toprule
$err_{\rho}+err_{\rho_{\bv}} (T=1)$        &   $\triangle\tau=1/2^{2}$    &   $\triangle\tau=1/2^{3}$  &  $\triangle\tau=1/2^{4}$  &  $\triangle\tau=1/2^{5}$  &  $\triangle\tau=1/2^{6}$  \\
\midrule
$\eps=1/2$          &   1.37E-1   &   6.01E-2   &   3.67E-2   &    2.02E-2     &   1.09E-2      \\
order               &    -        &    1.19     &    0.71     &    0.86        &    0.89       \\
$\eps=1/2^{2}$      &   4.69E-1   &   1.25E-1   &   4.19E-2   &    2.43E-2     &   1.85E-2    \\
order               &    -        &   1.91      &    1.57     &     0.79       &    0.40       \\
$\eps=1/2^{3}$      &   1.14E+0   &   3.38E-1   &   7.96E-2   &   3.74E-2      &   2.73E-2    \\
order               &    -        &    1.76     &    2.09     &    1.09        &    0.45       \\
\toprule
$err_{\rho}+err_{\rho_{\bv}} (T=1)$        &   $\triangle\tau=1/2^{2}$    &   $\triangle\tau=1/2^{3}$  &  $\triangle\tau=1/2^{4}$  &  $\triangle\tau=1/2^{5}$  &  $\triangle\tau=1/2^{6}$  \\
\midrule
$\eps=1/2$          &   1.23E-1   &   1.66E-2   &   3.26E-3   &   3.71E-3   &   4.15E-3   \\
order               &    -        &    2.89     &    2.35     &    -0.19    &    -0.16         \\
$\eps=1/2^{2}$      &   4.72E-1   &   1.12E-1   &   1.41E-2   &   9.90E-3   &   1.42E-2     \\
order               &    -        &    2.07     &    2.99     &    0.51     &   -0.52         \\
$\eps=1/2^{3}$      &   1.15E+0   &   3.40E-1   &   6.45E-2   &   2.34E-2   &   2.11E-2   \\
order               &    -        &    1.75     &    2.40     &    1.46     &    0.15        \\
\bottomrule
\end{tabular}
\label{tab-4-1}
\end{table}

\begin{figure}[t!]
$$\begin{array}{cc}
\psfig{figure=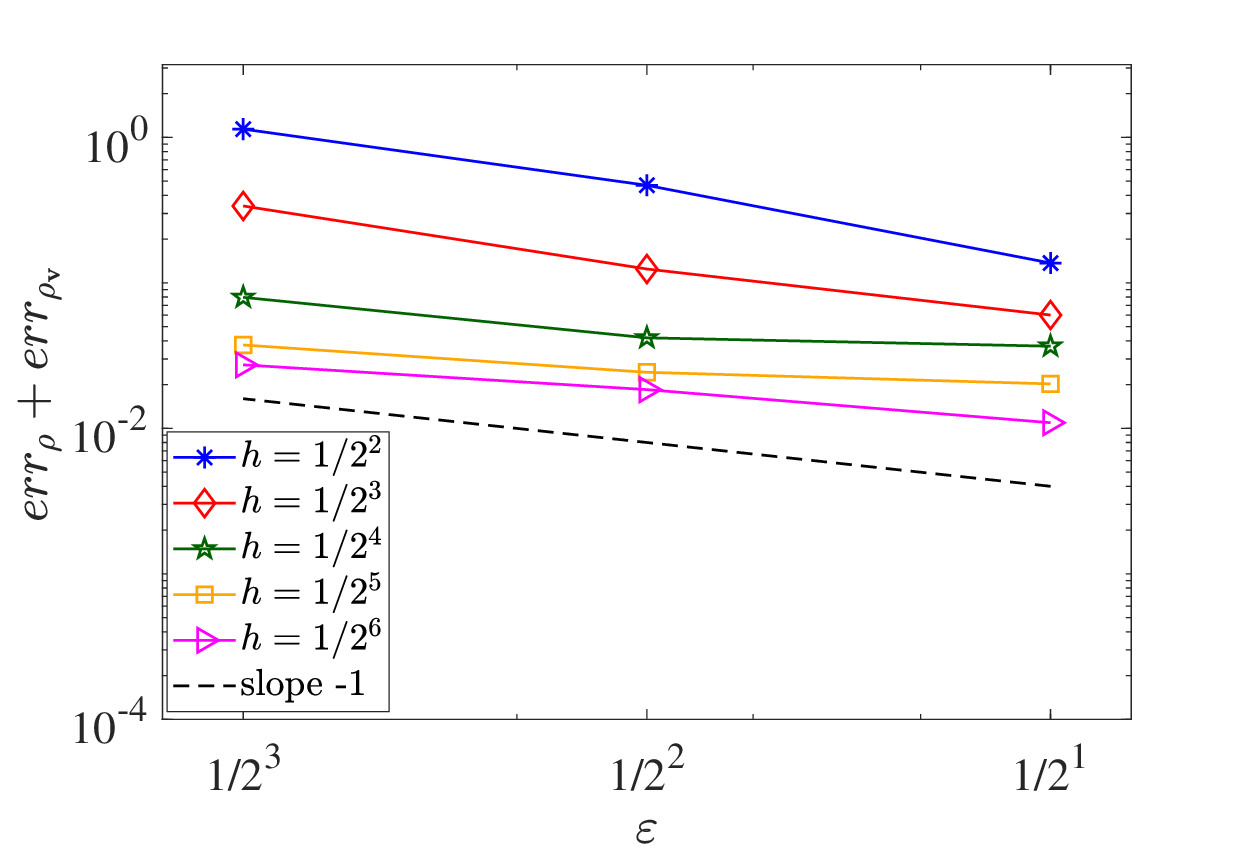,height=3.5cm,width=6.5cm}
\psfig{figure=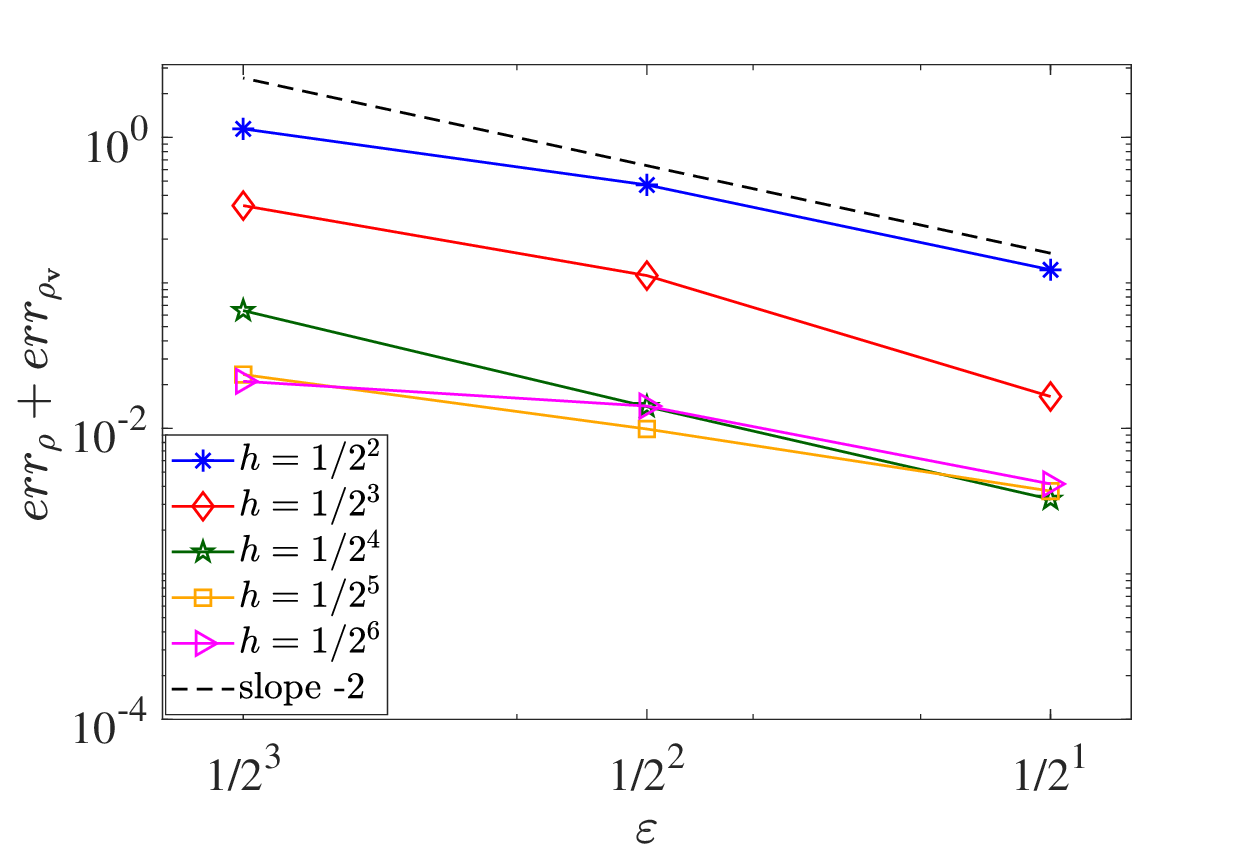,height=3.5cm,width=6.5cm}
\end{array}$$
\caption{Example 4. Errors of RS1-PIC (left) and RS2-PIC (right) with respect to $\eps$ under different $h$ until $T=1$.}\label{fig-4-2}
\end{figure}

\begin{figure}[t!]
$$\begin{array}{cc}
\psfig{figure=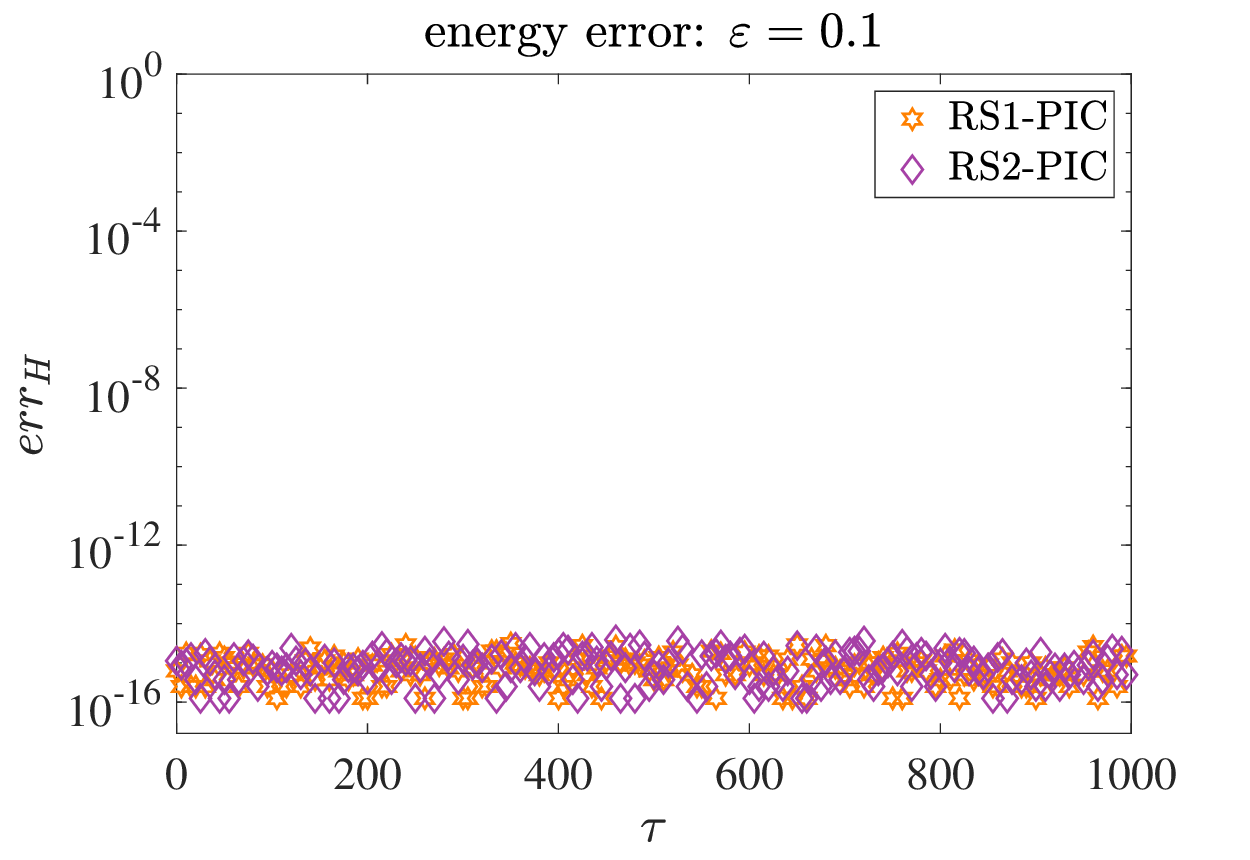,height=3.5cm,width=6.5cm}
\psfig{figure=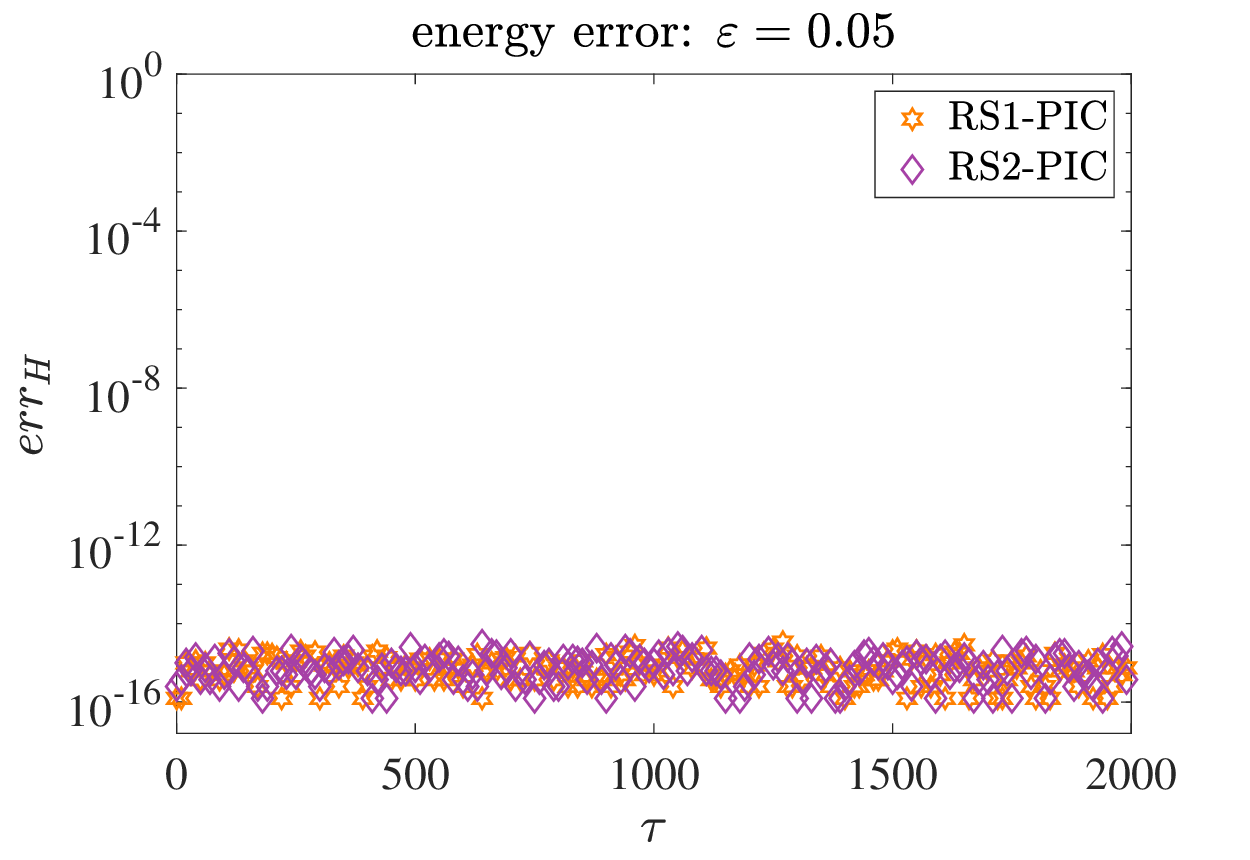,height=3.5cm,width=6.5cm}
\end{array}$$
\caption{Example 4. Energy errors of RS1-PIC and RS2-PIC with step size $h=0.1$ under different $\eps$ until $T=100$.}\label{fig-4-3}
\end{figure}


\section{Conclusion}\label{sec:con}
This work presents a novel class of explicit relaxation Particle-in-Cell (ER-PIC) methods for the Vlasov-Poisson system in the presence of a strong magnetic field. These methods are built upon a splitting framework and are distinguished by a unique relaxation parameter. Its dynamic adjustment preserves the accuracy of the base scheme while guaranteeing exact energy conservation. A rigorous error analysis for the Strang-type ER-PIC scheme is also provided, utilizing an averaging technique.
A primary direction for future research is the design of explicit methods that can simultaneously achieve high-order uniform accuracy and exact energy conservation for the Vlasov-Poisson system under  a strong magnetic field.

\bibliographystyle{model-num-names}

\end{document}